\documentclass[12pt]{article}

\usepackage[colorlinks = true, pdfstartview = FitV, linkcolor = blue, citecolor = blue, urlcolor = blue]{hyperref}
\usepackage{booktabs}
\usepackage{graphicx}
\usepackage{amsmath}
\usepackage{amssymb}
\usepackage{latexsym}
\usepackage{subfigure}
\usepackage{crop}
\usepackage{algorithmic,algorithm}
\usepackage{multirow}
\usepackage{bm}
\usepackage{bbm}
\usepackage{enumerate}
\usepackage{url}
\usepackage{array}
\usepackage{paralist}
\usepackage{diagbox}
\usepackage{booktabs}
\usepackage[dvipsnames]{xcolor}

\usepackage{rotating}

\usepackage[capitalise]{cleveref}
\crefname{equation}{}{}
\crefname{figure}{Figure}{Figures}
\creflabelformat{equation}{\textup{(#2#1#3)}}
\crefname{assumption}{Assumption}{Assumptions}
\crefname{condition}{Condition}{Conditions}
\usepackage{xspace}

\renewcommand\th{\textsuperscript{th}\xspace}
\usepackage{fullpage}
\usepackage{multirow}

\newcommand\undermat[2]{%
	\makebox[0pt][l]{$\smash{\underbrace{\phantom{%
					\begin{matrix}#2\end{matrix}}}_{\text{$#1$}}}$}#2}
\usepackage[sort,nocompress]{cite}
\usepackage{arydshln}
\usepackage{enumitem}
\setlist[enumerate,1]{leftmargin=*,wide=0em, noitemsep,nolistsep, label = {\bfseries \arabic*.}}
\setlist[itemize,1]{leftmargin=*,wide=0em, noitemsep,nolistsep}
\usepackage{pifont}% http://ctan.org/pkg/pifont
\usepackage{accents}

\usepackage{stackengine}
\stackMath
\newcommand\tsup[2][2]{%
	\def\useanchorwidth{T}%
	\ifnum#1>1%
	\stackon[-.5pt]{\tsup[\numexpr#1-1\relax]{#2}}{\scriptscriptstyle\sim}%
	\else%
	\stackon[.5pt]{#2}{\scriptscriptstyle\sim}%
	\fi%
}
\newcommand{\df}{\mathrm{d}}
\newcommand{\reals}{\mathbb{R}}
\DeclareMathOperator*{\argmin}{arg\,min}
\DeclareMathOperator*{\Argmin}{Arg\,min}

\renewcommand {\AA}  { {\mathbf{A}} }

\newcommand {\BB}  { {\mathbf{B}} }
\newcommand {\DD}  { {\mathbf{D}} }
\newcommand {\EE}  { {\mathbf{E}} }
\newcommand {\HH}  { {\mathbf{H}} }
\newcommand {\VV}  { {\mathbf{V}} }

\newcommand {\UU}  { {\mathbf{U}} }
\newcommand {\PP}  { {\mathbf{P}} }
\newcommand {\QQ}  { {\mathbf{Q}} }

\newcommand{\eye}{\mathbf{I}}

\newcommand {\zz}  { {\bf z} }
\newcommand {\bgg}  { {\bf g} }
\newcommand {\xx}  { {\bf x} }
\newcommand {\yy}  { {\bf y} }

\renewcommand {\aa}  { {\bf a} }

\newcommand {\qq}  { {\bf q} }
\newcommand {\pp}  { {\bf p} }
\newcommand {\uu}  { {\bf u} }
\newcommand {\vv}  { {\bf v} }
\newcommand {\ww}  { {\bf w} }

\newcommand {\bb}  { {\bf b} }

\newcommand {\zero}  { {\bf 0} }

\newcommand {\alphak}  { {{\alpha}_{k}} }

\newcommand {\HHk}  { {\HH_{k}} }
\newcommand {\bggk}  { {{\bgg}_{k}} }
\newcommand {\bggkk}  { {{\bgg}_{k+1}} }
\newcommand {\yys}  { {\yy^{\star}} }
\newcommand {\xxs}  { {\xx^{\star}} }

\newcommand {\xxo}  { {{\xx}_{0}} }
\newcommand {\xxk}  { {{\xx}_{k}} }
\newcommand {\xxkk}  { {{\xx}_{k+1}} }

\newcommand {\ppk}  { {{\pp}_{k}} }

\newcommand {\rank}  { {\textnormal{Rank}} }

\newcommand {\range}  { {\textnormal{Range}} }
\newcommand {\Null}  { {\textnormal{Null}} }
\newcommand {\Span}  { {\textnormal{Span}} }

\newcommand {\teps}  { {\tilde{\epsilon}} }

\newcommand {\tnu}  { {\tilde{\nu}} }
\newcommand {\tr}  { {\tilde{r}} }
\newcommand {\tuu}  { {\tilde{\uu}} }

\newcommand {\tgamma}  { {\tilde{\gamma}} }

\newcommand {\tSigma}  { {\tilde{\Sigma}} }
\newcommand {\tHHd}  { {{\tilde{\HH}}^{\dagger}} }
\newcommand {\tHH}  { {\tilde{\HH}} }
\newcommand {\HHd}  { {\HH^{\dagger}} }

\newcommand {\HHdk}  { {\left[\HH_{k}\right]^{\dagger}} }

\newcommand {\tHHk}  { {\tilde{\HH}_{k}} }

\newcommand {\tHHdk}  { {\left[\tilde{\HH}_{k}\right]^{\dagger}} }

\newcommand {\UUt}  { {\UU^{\intercal}} }
\newcommand {\tUU}  { {\tilde{\UU}} }
\newcommand {\UUp}  { {\UU_{\perp}} }
\newcommand {\tUUt}  { {\tilde{\UU}^{\intercal}} }
\newcommand {\VVt}  { {\VV^{\intercal}} }
\newcommand {\tVV}  { {\tilde{\VV}} }
\newcommand {\tVVt}  { {\tilde{\VV}^{\intercal}} }

\newcommand {\tAA}  { {\tilde{\AA}} }

\newcommand {\ppkt}  { {\pp_{k}^{(t)}} }
\newcommand {\QQt}  { {\QQ_{t}} }

\newcommand {\TT}  { {\mathbf{T}} }
\newcommand {\TTt}  { {\mathbf{T}_{t}} }

\newcommand {\tPP}  { {\tilde{\PP} }}

\newcommand{\hf}{\frac12}

\newcommand{\defeq}{\triangleq}

\renewcommand{\Pr}{\hbox{\bf{Pr}}}

\definecolor{forestgreen}{rgb}{0.13, 0.55, 0.13}

\definecolor{amber}{rgb}{1.0, 0.75, 0.0}

\definecolor{bananayellow}{rgb}{.8, 0.6, 0}

\definecolor{cadmiumyellow}{rgb}{1.0, 0.96, 0.0}

\definecolor{magenta}{rgb}{1, 0, 1}

\newcounter{comment}

\usepackage{amsthm}
\usepackage[framemethod=TikZ]{mdframed}

\newmdtheoremenv[%
linewidth = 1pt,%
roundcorner = 10pt,%
leftmargin = 0,%
rightmargin = 0,%
backgroundcolor = green!3,%
outerlinecolor = blue!70!black,%
splittopskip = \topskip,%
ntheorem = true,%
]{theorem}{Theorem}

\newmdtheoremenv[%
linewidth = 1pt,%
roundcorner = 10pt,%
leftmargin = 0,%
rightmargin = 0,%
backgroundcolor = green!3,%
outerlinecolor = blue!70!black,%
splittopskip = \topskip,%
ntheorem = true,%
]{corollary}{Corollary}

\newmdtheoremenv[%
linewidth = 1pt,%
roundcorner = 10pt,%
leftmargin = 0,%
rightmargin = 0,%
backgroundcolor = green!3,%
outerlinecolor = blue!70!black,%
splittopskip = \topskip,%
ntheorem = true,%
]{lemma}{Lemma}

\newmdtheoremenv[%
linewidth = 1pt,%
roundcorner = 10pt,%
leftmargin = 0,%
rightmargin = 0,%
backgroundcolor = blue!3,%
outerlinecolor = blue!70!black,%
splittopskip = \topskip,%
ntheorem = true,%
]{definition}{Definition}

\newmdtheoremenv[%
linewidth = 1pt,%
roundcorner = 10pt,%
leftmargin = 0,%
rightmargin = 0,%
backgroundcolor = green!3,%
outerlinecolor = blue!70!black,%
splittopskip = \topskip,%
ntheorem = true,%
]{proposition}{Proposition}

\newmdtheoremenv[%
linewidth = 1pt,%
roundcorner = 10pt,%
leftmargin = 0,%
rightmargin = 0,%
backgroundcolor = green!3,%
outerlinecolor = blue!70!black,%
splittopskip = \topskip,%
ntheorem = true,%
]{condition}{Condition}

\newmdtheoremenv[%
linewidth = 1pt,%
roundcorner = 10pt,%
leftmargin = 0,%
rightmargin = 0,%
backgroundcolor = yellow!3,%
outerlinecolor = blue!70!black,%
splittopskip = \topskip,%
ntheorem = true,%
]{assumption}{Assumption}

\theoremstyle{definition}
\newmdtheoremenv[%
linewidth = 1pt,%
roundcorner = 10pt,%
leftmargin = 0,%
rightmargin = 0,%
backgroundcolor = cyan!3,%
outerlinecolor = blue!70!black,%
splittopskip = \topskip,%
ntheorem = true,%
]{example}{Example}

\theoremstyle{definition}
\newmdtheoremenv[%
linewidth = 1pt,%
roundcorner = 10pt,%
leftmargin = 0,%
rightmargin = 0,%
backgroundcolor = red!3,%
outerlinecolor = blue!70!black,%
splittopskip = \topskip,%
ntheorem = true,%
]{remark}{Remark}

\usepackage{tikz}
\usepackage{xparse}% So that we can have two optional parameters

\NewDocumentCommand\DownArrow{O{2.0ex} O{black}}{%
	\mathrel{\tikz[baseline] \draw [<-, line width=0.5pt, #2] (0,0) -- ++(0,#1);}
}

\usepackage{listings} % to inser code

\definecolor{mygreen}{rgb}{0,0.6,0}
\definecolor{mygray}{rgb}{0.5,0.5,0.5}
\definecolor{mymauve}{rgb}{0.58,0,0.82}

\newcommand*\dotprod[1]{\left\langle #1\right\rangle}
\newcommand*\vnorm[1]{\left\| #1\right\|}
\newcommand\abs[1]  {  {\left| #1 \right|} }

\newcommand*\bigO[1]{\mathcal O\left( #1\right)}
\newcommand*\bigOt[1]{\tilde{\mathcal O}\left( #1\right)}

\usepackage{dsfont}

\usepackage[margin=0.5in]{geometry}

\begin{document}

\title{Convergence of Newton-MR under Inexact Hessian Information}

\author{Yang Liu\thanks{School of Mathematics and Physics, University of Queensland, Australia. Email: \tt{yang.liu2@uq.edu.au}}
	\and
	Fred Roosta\thanks{School of Mathematics and Physics, University of Queensland, Australia, and International Computer Science Institute, Berkeley, USA. Email: \tt{fred.roosta@uq.edu.au}}
}
\date{\today}
\maketitle

\abstract{
	Recently, there has been a surge of interest in designing variants of the classical Newton-CG in which the Hessian of a (strongly) convex function is replaced by suitable approximations. This is mainly motivated by large-scale finite-sum minimization problems that arise in many machine learning applications. Going beyond convexity, inexact Hessian information has also been recently considered in the context of algorithms such as trust-region or (adaptive) cubic regularization for general non-convex problems. Here, we do that for Newton-MR, which extends the application range of the classical Newton-CG beyond convexity to invex problems. Unlike the convergence analysis of Newton-CG, which relies on spectrum preserving Hessian approximations in the sense of L\"{o}wner partial order, our work here draws from matrix perturbation theory to estimate the distance between the subspaces underlying the exact and approximate Hessian matrices. Numerical experiments demonstrate a great degree of resilience to such Hessian approximations, amounting to a highly efficient algorithm in large-scale problems.
}

\section{Introduction}
\label{sec:Intro}
Consider the unconstrained optimization problem:
\begin{equation}
\label{eq:obj}
\min_{\xx \in \mathbb{R}^{d}} f(\xx),
\end{equation}
where $ f:\mathbb{R}^{d} \rightarrow \mathbb{R} $. Due to simplicity and solid theoretical foundations, there is an abundance of algorithms designed specifically for the case where $ f $ is convex \cite{bertsekas2015convex,nesterov2004introductory,boyd2004convex}. Strong-convexity, as a special case, allows for the design of algorithms with a great many theoretical and algorithmic properties. 
In such settings, the classical Newton's method and its Newton-CG variant hold a special place. In particular, for strongly-convex functions with sufficient degree of smoothness, they have been shown to enjoy various desirable properties including insensitivity to problem ill-conditioning \cite{ssn2018,xu2016sub}, problem-independent local convergence rates \cite{ssn2018}, and robustness to hyper-parameter tuning \cite{kylasa2018gpu,berahas2017investigation}.

Arguably, the only drawback of Newton's method in large-scale problems is the computational cost of applying the Hessian matrix.  
For example, consider the canonical problem of finite-sum minimization where 
\begin{align}
	\label{eq:obj_sum}
	f(\xx) = \frac{1}{n}\sum_{i=1}^{n} f_{i}(\xx),
\end{align}
and each $ f_{i} $ corresponds to an observation (or a measurement), which models the loss (or misfit) given a particular choice of the underlying parameter $ \xx $. 
Problems of the form \cref{eq:obj_sum} arise very often in machine learning, e.g., \cite{shalev2014understanding}, as well as scientific computing, e.g., \cite{rodoas1}. 
Here, the Hessian matrix can be written as $\HH(\xx) = \sum_{i = 1}^{n} \nabla^2 f_i (\xx)/n$.
In big-data regime where $ n \gg 1 $, operations with the Hessian of $ f $, e.g., matrix-vector products, typically constitute the main bottleneck of computations. 
In this light, several recent efforts have focused on the design and analysis of variants of Newton's method in which the exact Hessian matrix, $ \HH(\xx) \triangleq \nabla^{2} f(\xx) $ is replaced with its suitable approximation $ \tHH(\xx) \approx \nabla^{2} f(\xx) $.
One such approximation strategy is randomized sub-sampling in which by considering a sample of size $ |\mathcal{S}| \ll n$, we can form the sub-sampled Hessian as $\tHH(\xx) = \sum_{j \in \mathcal{S}} \nabla^2 f_j (\xx)/|\mathcal{S}|$.
Under certain conditions, sub-sampling Hessian has been shown to be very effective in reducing the overall computational costs, e.g., \cite{ssn2018,bollapragada2016exact,byrd2011use,byrd2012sample,erdogdu2015convergence}. In certain special cases, the structure of $ f $ could also be used for more sophisticated randomized matrix approximation strategies such as sketching and those based on statistical leverage scores, e.g., \cite{pilanci2015newton,xu2016sub}.

However, in the absence of either sufficient smoothness or strong-convexity, the classical Newton's method and its Newton-CG can simply break down, e.g., their underlying sub-problems may fail to have a solution. Hence, many Newton-type variants have been proposed which aim at extending Newton's method beyond strongly-convex problems, e.g., Levenberg-Marquardt \cite{levenberg1944algorithm,marquardt1963algorithm}, trust-region \cite{conn2000trust}, cubic regularization \cite{nesterov2006cubic,cartis2011adaptiveI,cartis2011adaptiveII}, and various other methods, which make clever use of negative curvature when it arises \cite{carmon2016accelerated,royer2018complexity,royer2020newton,carmon2017convex}. Variants of these methods using inexact function approximations, including approximate Hessian, have also been studied, e.g.,  \cite{cartis2012complexity,chen2015stochastic,blanchet2016convergence,bandeira2014convergence,larson2016stochastic,shashaani2016astro,gratton2017complexity,xuNonconvexTheoretical2017,yao2018inexact,xuNonconvexEmpirical2017}. 
However, many of these methods rely on strict smoothness assumptions such as Lipschitz continuity of gradient and Hessian. In addition, in sharp contrast to Newton's method whose sub-problems are simple linear systems, a vast majority of these methods involve sub-problems that are themselves non-trivial to solve, e.g., the sub-problems of trust-region and cubic regularization methods are non-linear and non-convex.

To extend the application range of the classical Newton's method beyond strongly-convex settings, while maintaining the simplicity of its sub-problems, Newton-MR \cite{roosta2018newton} has recently been proposed. Iterations of Newton-MR, at a high level, can be written as $ \xxkk = \xxk - \alphak \HHdk \bggk $, where $ \alphak $ is some suitably chosen step-size and $ \HHdk $ is the Moore-Penrose generalized inverse of $ \HHk $. On the surface, Newton-MR bares a striking resemblance to the classical Newton's method and shares several of its desirable properties, e.g., Newton-MR involves simple sub-problems in the form of ordinary least squares. However, not only does Newton-MR requires more relaxed smoothness assumptions compared with most non-convex Newton-type methods, but also it can be readily applied to a class of non-convex problems known as \emph{invex} \cite{mishra2008invexity,ben1986invexity} , which subsumes convexity as a sub-class.
Recall that the class of invex functions, first studied in~\cite{hanson1981sufficiency}, extends the sufficiency of the first order optimality condition to a broader class of problems than simple convex programming. In other words, the necessary and sufficient condition for any minimizer of invex problem \cref{eq:obj} is  $\nabla f(\xx^{*}) = 0$. However, this is alternatively equivalent to having $\vnorm{\nabla f(\xx^{*})} = 0$, which in turn give rise to the following auxiliary non-convex optimization problem
\begin{align}
	\label{eq:obj_grad}
	\min_{\xx \in \mathbb{R}^{d}} \hf \vnorm{\nabla f(\xx)}^{2}.
\end{align}
Since the global minimizers of \cref{eq:obj_grad} are the stationary points of \cref{eq:obj}, when $ f $ is invex, the global minimizers of \cref{eq:obj} and \cref{eq:obj_grad} coincide. Newton-MR is built upon considering \cref{eq:obj_grad}, which in turn makes it suitable for \cref{eq:obj} with invex objectives. For more general non-convex functions, instead of minimizing $ f$, the iterations of Newton-MR converge towards the zeros of its gradient field, which are used in many applications, e.g., exploring the loss landscape in chemical physics \cite{ballard2017energy,wales2003energy} and deep neural networks \cite{frye2019numerically}.
Motivated by the potential and advantages of Newton-MR, in this paper we provide its convergence analysis under inexact Hessian information. In this light, we show that appropriately approximating the Hessian matrix allows for efficient application of Newton-MR to large-scale problems. 
Before delving any deeper, we note that, while the main motivating class of problems for our work here is that of finite-sum minimization \cref{eq:obj_sum}, we develop our theory more generally for \cref{eq:obj}. This is so since for more general objectives, it is also often possible to approximate the Hessian using quasi-Newton methods, e.g., symmetric rank one update \cite{byrd1996analysis,conn1991convergence,cartis2011adaptiveI}, or finite-difference approximations \cite{dennis1996numerical,nocedal2006numerical}.

The rest of this paper is organized as follows. We end this section by introducing the notation and the assumptions on $ f $ used in this paper. In \Cref{sec:pert}, we study inexact Hessian information in light of matrix perturbation theory and establish conditions under which such perturbations satisfy a notion of stability. In \Cref{sec:stability}, we leverage this stability analysis and provide convergence results for Newton-MR with Hessian approximations. Numerical experiments are presented in \Cref{sec:experiments}. Conclusions are gathered in \Cref{sec:conclusion}.

%\subsection{Notation and Assumptions}
%\label{sec:notation_assumptions}
%Before presenting the technical aspects of our work, we briefly introduce the notation as well as the assumptions used throughout this paper.

\subsection{Notation}
\label{sec:notation}
Throughout the paper, vectors and matrices are denoted by bold lower-case and bold upper-case letters, respectively, e.g., $ \vv $ and $ \VV $. 
We use regular lower-case and upper-case letters to denote scalar constants, e.g., $ d $  or $ L $. 
For a real vector, $ \vv $, its transpose is denoted by $ \vv^{\intercal} $. 
For two vectors $ \vv,\ww $, their inner-product is denoted as $ \dotprod{\vv, \ww}  = \vv^{\intercal} \ww $. For a vector $\vv$ and a matrix $ \VV $, $ \|\vv\| $ and $ \|\VV\| $ denote vector $ \ell_{2} $ norm and matrix spectral norm, respectively. 
%$ \nabla f(\xx) $ and $ \nabla^{2} f(\xx) $ are the gradient and the Hessian of $ f $ at $ \xx $, respectively. 
Iteration counter for the main algorithm appears as subscript, e.g., $ \ppk $. Iteration counter for sub-problem solver to obtain $ \ppk $ appears as superscript, e.g., $\ppkt$. The vector of all zero components is denoted by $\bm{0}$. 
For two symmetric matrices $ \AA $ and $ \BB $, the L\"{o}wner partial order $\AA \succeq \BB$ indicates that $ \AA-\BB $ is symmetric positive semi-definite.
For any $ \xx,\zz \in \reals^{2} $, $ \yy \in [\xx, \zz] $ denotes $ \yy = \xx + \tau  (\zz - \xx) $ for some $ 0 \leq \tau \leq 1 $.
$ \AA^{\dagger} $ denotes the Moore-Penrose generalized inverse of matrix $ \AA $. 
The element of a matrix $ \AA $ located at the $ i\th $ row and the $ j\th $ column is denoted by $ [\AA]_{ij} $.
$ \sigma_{i}(\AA) $ denotes $ i\th $ largest singular values of a matrix $ \AA $. The eigenvalues of a symmetric matrix $ \AA \in \reals^{d \times d}$ are ordered as $\lambda_{1}(\AA) \geq \lambda_{2}(\AA) \geq \ldots \lambda_{d}(\AA)$.
For simplicity, we use $ \bgg(\xx) \triangleq \nabla f\left(\xx\right) \in \reals^{d} $ and $ \HH(\xx) \triangleq \nabla^{2} f\left(\xx\right) \in \reals^{d \times d} $ for the gradient and the Hessian of $ f $ at $ \xx $, respectively, and at times we drop the dependence on $ \xx $ by simply using $ \bgg $ and $ \HH $, e.g., $ \bggk = \bgg(\xxk) $ and $ \HHk = \HH(\xxk) $. Finally, ``$ \Argmin $'' implies that the minimum may be attained at more than one point.

\subsection{Assumptions on the Objective Function}
\label{sec:assumptions} 
Here, we introduce the assumptions on $ f $ in \cref{eq:obj} that underlie our work. We note that these assumptions are essentially the same as those in \cite{roosta2018newton}.

\begin{assumption}[Differentiability]
	\label{assmpt:moral_diff}
	\vspace*{-2mm}
	The function $ f $ is twice-differentiable. 
\end{assumption}
In particular, all the first partial derivatives are themselves differentiable, but the second partial derivatives are allowed to be discontinuous. Recall that requiring the first partials be differentiable implies the equality of crossed-partials, which amounts to the symmetric Hessian matrix \cite[pp.\ 732-733]{hubbard2015vector}. 

Instead of the typical smoothness assumptions in the form of Lipschitz continuity of gradient and Hessian, i.e., for some $ 0 \leq L_{\bgg} < \infty $ and $ 0 \leq L_{\HH} < \infty $
\begin{subequations}
	\label{eq:lip_usual}
	\begin{align}
	\vnorm{\bgg(\xx) - \bgg(\yy)} &\leq L_{\bgg} \vnorm{\xx - \yy}, \label{eq:lip_usual_grad} \\
	\vnorm{\HH(\xx) - \HH(\yy)} &\leq L_{\HH} \vnorm{\xx - \yy}, \label{eq:lip_usual_hessian}
	\end{align}
\end{subequations}
a more relaxed notion, called moral-smoothness, was introduced in \cite{roosta2018newton}. 
\begin{assumption}[Moral-smoothness]
	\label{assmpt:lipschitz_special}
	\vspace*{-2mm}
	For any $ \xx_{0} \in \reals^{d} $, there is a constant $ 0 < L(\xx_{0}) < \infty $, such that 
	\begin{align}
	\label{eq:lip_special}
	\vnorm{\HH(\yy) \bgg(\yy)  - \HH(\xx) \bgg(\xx)} \leq L(\xx_{0}) \vnorm{\yy - \xx}, \quad \forall (\xx,\yy) \in \mathcal{X}_{0} \times \reals^{d}, 
	\end{align}
	where $ \mathcal{X}_{0} \triangleq \left\{\xx \in \reals^{d} \mid \| \bgg(\xx) \| \leq \| \bgg(\xx_{0}) \| \right\} $.
\end{assumption}
\cref{assmpt:lipschitz_special} is similar to Lipschitz continuity assumption for the gradient of the auxiliary objective \cref{eq:obj_grad}, albeit restricted to $ \mathcal{X}_{0} \times \reals^{d} $ and with the constant $ L(\xx_{0}) $ that depends on the choice of $ \xx_{0} $.
By \cref{eq:lip_special}, it is only the action of Hessian on the gradient that is required to be Lipschitz continuous, and each gradient and/or Hessian individually can be highly irregular, e.g., gradient can be very non-smooth and Hessian can even be discontinuous. In \cite{roosta2018newton}, it has been shown that \cref{assmpt:lipschitz_special} is significantly more relaxed than \cref{eq:lip_usual}, i.e., \cref{assmpt:lipschitz_special} is implied by \cref{eq:lip_usual} and the converse is not true.

\begin{assumption}[Pseudo-inverse Regularity]
	\label{assmpt:pseudo_regularity}
	\vspace*{-2mm}
	There exists a constant $ \gamma > 0 $, such that
	\begin{align}
	\label{eq:pseudo_regularity}
	\vnorm{\HH(\xx) \pp} \geq \gamma \vnorm{\pp}, \quad \forall \pp \in \textnormal{Range}\left(\HH(\xx)\right).
	\end{align}
\end{assumption}
Intuitively, $ \gamma $ is a uniform lower-bound on the smallest, in magnitude, among the non-zero eigenvalues of $ \HH(\xx) $, for any $ \xx $. \cref{assmpt:pseudo_regularity} also implies that $ \gamma $ is required to be uniformly bounded away from zero for all $ \xx $; see \cite[Example 3]{roosta2018newton} for examples of functions satisfying \cref{assmpt:pseudo_regularity}. 
Furthermore, in \cite{roosta2018newton}, \cref{eq:pseudo_regularity} has been shown to be equivalent to 
\begin{align}
\label{eq:pseudo_regularity_H}
\|\HH^{\dagger}(\xx)\| \leq \frac{1}{\gamma}.
\end{align}

\begin{assumption}[Gradient-Hessian Null-space Property]
	\label{assmpt:null_space}
	\vspace*{-2mm}
	For any $ \xx \in \mathbb{R}^{d} $, let $ \UU $ and $ \UU_{\perp} $ denote arbitrary orthogonal bases for $ \range(\HH(\xx)) $ and its orthogonal complement, respectively. A function is said to satisfy the Gradient-Hessian Null-Space property, if there exists $ 0 < \nu \leq 1 $, such that 
	\begin{align}
	\label{eq:null_space}
	\vnorm{\UU_{\perp}^{\intercal} \bgg(\xx)}^{2} \leq \left( \frac{1 - \nu}{\nu} \right) \vnorm{\UU^{\intercal} \bgg(\xx)}^{2}, \quad \forall \xx \in \reals^{d}.
	\end{align}
\end{assumption}

\Cref{assmpt:null_space} ensures that the angle between the gradient and the range-space of the Hessian matrix is uniformly bounded away from zero. In other words, as iterations progress, the gradient will not become arbitrarily orthogonal to the range space of Hessian. Explicit examples of functions that satisfy this assumption, including the special case when $ \nu = 1 $, i.e., the gradient lies fully in the range of the full Hessian, are given in \cite{roosta2018newton}. The following lemma, also from \cite[Lemma 4]{roosta2018newton}, is a direct consequence of \cref{assmpt:null_space}.

\begin{lemma}[{\!\!\cite[Lemma 4]{roosta2018newton}}]
	\label{lemma:null_space_grad}
	\vspace*{-2mm}
	Under \cref{assmpt:null_space}, we have 
	\begin{subequations}
		\label{eq:null_space_grad}  
		\begin{align}
		\vnorm{\UU^{\intercal}\bgg}^2 &\geq \nu \vnorm{\bgg}^{2}, \label{eq:null_space_grad_1} \\
		\vnorm{\UU_{\perp}^{\intercal}\bgg}^2 &\leq (1 - \nu) \vnorm{\bgg}^{2}. \label{eq:null_space_grad_2}
		\end{align}
	\end{subequations}  
\end{lemma}

\section{Inexact Hessian and Matrix Perturbation Theory}
\label{sec:pert}
Typically, at the heart of convergence analysis of various Newton-type methods with inexact Hessian information lies matrix perturbations of the form 
\begin{subequations}
	\label{eq:perturb}	
	\begin{align}
	\label{eq:E}
	\tHH = \HH + \EE,
	\end{align}
	for some symmetric matrix $ \EE $. The goal is then to obtain a bound 
	\begin{align}
	\label{eq:epsilon}
	\vnorm{\EE} \leq \varepsilon,
	\end{align}
\end{subequations}
such that the algorithm with inexact Hessian has desirable per-iteration costs and yet maintains the iteration complexity of the original exact algorithm. 
Once such a bound on $ \varepsilon $ is established, a variety of approaches including deterministic, e.g., finite difference, and stochastic, e.g., sub-sampling, can be used to form $ \tHH $ in \cref{eq:perturb}. For example, for finite-sum minimization problem \cref{eq:obj_sum}, \cite[Lemma 16]{xuNonconvexTheoretical2017} established that if $|\mathcal{S}| \in \mathcal{O}\left(\varepsilon^{-2} \log \left( 2d/ \delta\right) \right)$ and samples are drawn uniformly at random, then we have $ \Pr \left(\| \HH - \tHH \| \leq \varepsilon \right) \geq 1 - \delta$. Bounds using non-uniform sampling have also been given, e.g., \cite{xuNonconvexTheoretical2017,xu2016sub}.

In strongly-convex settings, establishing the convergence of the classical Newton's method and its Newton-CG variant relies on Hessian perturbations that are \emph{approximately spectrum preserving}, e.g., $ \varepsilon $ in \cref{eq:perturb} must be such that for the perturbed Hessian, we have
\begin{align}
\label{eq:spectrum}
(1-\tilde{\varepsilon}_{1})\HH \preceq \tHH \preceq (1+\tilde{\varepsilon}_{1})\HH,
\end{align}
where $ \tilde{\varepsilon}_{1} \in \mathcal{O}(\varepsilon) $ , e.g., \cite{erdogdu2015convergence,pilanci2015newton,ssn2018,bollapragada2016exact}. Under \cref{eq:spectrum}, the perturbed matrix is not only required to be full-rank, but also it must remain positive definite. In particular, \cref{eq:spectrum} implies that
\begin{align*}
(1-\tilde{\varepsilon}_{2})\HH^{-1} \preceq \tHH^{-1} \preceq (1+\tilde{\varepsilon}_{1}) \HH^{-1},
\end{align*}
for some $ \tilde{\varepsilon}_{2} \in \mathcal{O}(\varepsilon) $, which in turn gives
\begin{align}
\label{eq:inverse}
\vnorm{\HH^{-1} - \tHH^{-1}} \leq \tilde{\varepsilon}_{3},
\end{align}
for some $ \tilde{\varepsilon}_{3} \in \mathcal{O}(\varepsilon) $ \cite[Theorem 2.5]{stewart1990matrix}, i.e., the inverse of the perturbed Hessian is itself a small perturbation of the inverse of the true Hessian.
%\begin{align*}
%\frac{1}{(1+\tilde{\varepsilon})} \dotprod{\bgg, \HH^{-1}\bgg} \preceq \dotprod{\bgg, \tHH^{-1}\bgg} \preceq \frac{1}{(1-\tilde{\varepsilon})} \dotprod{\bgg, \HH^{-1}\bgg},
%\end{align*}
%i.e., the Newton decrement \cite{boyd2004convex}, which plays important role in the analysis of Newton's method and can also serve as an approximate upper bound on the sub-optimality gap, is well preserved.

In non-convex settings, however, where the true Hessian might be indefinite and/or rank deficient, requiring such conditions is simply infeasible. Indeed, when $ \HH $ is indefinite, the inequality \cref{eq:spectrum} ceases to be meaningful, i.e., no value of $ \tilde{\varepsilon} > 0 $ can give $ (1-\tilde{\varepsilon})\HH \preceq (1+\tilde{\varepsilon})\HH $. Further, when $ \HH $ has a zero eigenvalue, i.e., it is singular, it is practically impossible to assume that the corresponding eigenvalue of $ \tHH $ is also zero. In other words, in non-convex settings, no amount of perturbation in \cref{eq:perturb} will be guaranteed to be spectrum preserving.
In such settings, where the Hessian can simply fail to be invertible, one might be tempted to find a similar bound as in \cref{eq:inverse} but in terms of matrix pseudo-inverse, i.e., to find values of $ \varepsilon $ in \cref{eq:perturb} such that 
\begin{align}
\label{eq:pseudo_inverse}
\vnorm{\HHd - \tHHd} \leq  \tilde{\varepsilon}_{3}.
\end{align}
However, requiring \cref{eq:pseudo_inverse} is also extremely restrictive. In fact, it is known that the necessary and sufficient condition for $ \tHHd \rightarrow \HHd $ as $ \tHH \rightarrow \HH $, i.e., as $ \varepsilon \downarrow 0 $, is that $ \tHH $ is an \emph{acute perturbation} of $ \HH $, i.e., $ \rank(\tHH) = \rank(\HH) $ \cite[p.\ 146]{stewart1990matrix}. 
A cornerstone in the theory of matrix perturbations is establishing conditions on $ \tHH $ and $ \varepsilon $ in \cref{eq:perturb} that can give results in the same spirit as \cref{eq:pseudo_inverse} for more general perturbations, e.g., \cite{deng2010perturbation, meng2010optimal, stewart1990matrix}. 
For example, from \cite[Theorem 3.8]{stewart1990matrix} and \cref{eq:perturb}, we have 
\begin{align}
\label{eq:stewart}
\vnorm{\HHd - \tHHd} \leq \left(\frac{1 + \sqrt{5}}{2}\right) \max \left\{\vnorm{\HHd}^2, \vnorm{\tHHd}^2\right\} \varepsilon.
\end{align}
Employing \cref{eq:stewart} to guarantee \cref{eq:pseudo_inverse} necessarily relies on assuming $ \|\tHHd\| \in o(1/\sqrt{\varepsilon}) $, where ``$ o(.) $'' denotes the ``Little-O Notation'', i.e., $\|\tHHd\|$ must grow at a slower rate than $ 1/\sqrt{\varepsilon} $. However, as demonstrated by the following examples, this assumption is easily violated in many situations. 
%Inspired by \cite{stewart1979note,stewart1984rank,stewart1984second,stewart1990matrix,stewart1998perturbation,stewart1977perturbation,bernstein2009matrix}, we present examples demonstrating this. 

\begin{example}[Deterministic Perturbations]
	\label{example:det}
	\vspace*{-2mm}
	Suppose $ \vnorm{\EE} = \varepsilon $ and that the ratio of the largest over the smallest non-zero singular values of $ \EE $ is bounded by some constant, say $ C $. This, in turn, implies $ \sigma_{r_{\EE}}(\EE) \geq \varepsilon/C $, where $ \sigma_{r_{\EE}}(\EE) $ is the smallest non-zero singular value of $ \EE $. Further, suppose that $ r \triangleq \rank(\HH) \leq \rank(\tHH) \triangleq  \tr$ (cf.\ \cref{lemma:rank}), and $r_{\EE} = \tr + r $. By \cite[Proposition 9.6.8]{bernstein2009matrix}, we have
	\begin{align*}
	\sigma_{\tr}(\tHH) \geq \sigma_{r_{\EE}}(\EE) - \sigma_{r+1}(\HH) \geq {\varepsilon}/{C},
	\end{align*}
	which gives $\|\tHHd\| \in \mathcal{O}(1/\varepsilon)$.
\end{example}

\begin{example}[Random Perturbations]
	\label{example:rand}
	\vspace*{-2mm}
	Suppose $ [\EE]_{ij} \sim \mathcal{N}(0, \varepsilon^{2})$, where $ i,j = 1, \ldots, d $ and $ \mathcal{N}(0, \varepsilon^{2}) $ denotes the standard normal distribution with mean zero and standard deviation $ \varepsilon $. 
	%From the main result in \cite{latala2005some}, it follows that 
	%\begin{align*}
	%\Ex \left[\vnorm{\EE}\right] &\leq C_1\left(\max_i \sqrt{\sum_{j} \mathbb{E} \left[\EE^2_{ij}\right]} + \max_j \sqrt{\sum_{i} \mathbb{E} \left[\EE^2_{ij}\right]} + \sqrt[4]{\sum_{ij} \mathbb{E} \left[\EE^4_{ij}\right]}\right) \\
	%&= 2 C_1 \sqrt{d \varepsilon^2} + C_1 \sqrt[4]{3 d^2 \varepsilon^4}\\
	%&= C \varepsilon,
	%\end{align*}
	%where $ C_1$ and $C $ are some universal constants independent of $ \varepsilon $.
	We have $ \sigma_{i}(\HH) = 0, \; d \geq i > r \defeq \rank(\HH)$. Suppose further that the non-zero singular values of $ \HH $ are well separated from zero, e.g., $ \sigma_{i}(\HH) > 5 \varepsilon, \; i = 1,\ldots,r$. Using results similar to \cite[p.\ 411]{stewart1984rank}, one can show that the diagonal entries of $ \tSigma_{2} $ in \cref{eq:svd}, i.e., the non-zero singular values of $ \tHH $ corresponding to zero singular values of $ \HH $, will satisfy
	\begin{align*}
	\mathbb{E}\left[ \sigma^2_{i}(\tHH) \right] = ( d - r ) \varepsilon^{2} \quad  \text{ and } \quad \sigma_{i}(\tHH) \leq \sqrt{2} \vnorm{\EE}, \; i = r+1,\ldots,d.
	\end{align*}
	With probability $ 1-2 \exp(-2d) $, we have $ \vnorm{\EE} \leq  4 \sqrt{d} \varepsilon$ \cite[Eqn (2.3), p.\ 1582]{rudelson2010non}. 
	The latter two inequalities alone indicate that assuming $ \sigma_{i}(\tHH) \in \Omega(\sqrt{\varepsilon})$, and hence the stronger condition $ \|\tHHd\| \in o(1/\sqrt{\varepsilon}) $, is rather quite unreasonable.
	Now, on this latter event, for any $ C > 1 $, the reverse Markov inequality gives
	%\begin{align*}
	%\Pr\left( \sigma_{i}(\tHH) > \frac{\varepsilon}{C} \right) &= \Pr\left( \sigma^{2}_{i}(\tHH) > \frac{\varepsilon^{2}}{C^{2}} \right) \\
	%&\geq \mathbb{E}\left[ \frac{\sigma^{2}_{i}(\tHH) - \varepsilon^{2}/C^{2}}{\sigma^{2}_{i}(\tHH)} \right] \\
	%&\geq \mathbb{E}\left[ \frac{\sigma^{2}_{i}(\tHH) - \varepsilon^{2}/C^{2}}{32 d \varepsilon^{2}} \right] \\
	%&\geq \frac{( d - r ) \varepsilon^{2} - \varepsilon^{2}/C^{2}}{32 d \varepsilon^{2}}\\
	%&\geq \frac{d - r  - 1/C^{2}}{32 d}.
	%\end{align*}
	\begin{align*}
	\Pr\left( \sigma_{i}(\tHH) \leq \frac{\varepsilon}{C} \right) \leq \frac{\mathbb{E}\left[32 d \varepsilon^{2} - \sigma^{2}_{i}(\tHH) \right]}{32 d \varepsilon^{2} - \varepsilon^{2}/C^{2}} \leq \frac{32 d - ( d - r ) }{32 d - 1/C^{2}},
	\end{align*}
	which implies
	\begin{align*}
	\Pr\left( \sigma_{i}(\tHH) > \frac{\varepsilon}{C} \right) & \geq \frac{( d - r ) - 1/C^{2} }{32 d - 1/C^{2}}.
	\end{align*}
	In other words, with a positive probability that is independent of $ \varepsilon $, we have $ \|\tHHd\| \in \mathcal{O}(1/\varepsilon)$.
\end{example}
In light of \cref{example:det,example:rand}, a more sensible noise model is the one which allows for $\|\tHHd\|$ to grow at the same rate as $ 1/\varepsilon $.
\begin{assumption}[Perturbation Model]
	\label{assmpt:perturb}
	\vspace*{-2mm}
	For the perturbation \cref{eq:perturb}, we have 
	\begin{align}
	\label{eq:epsilon_02} 
	\vnorm{\tHHd} = \vnorm{\left[\HH + \EE\right]^{\dagger}} \leq \frac{C}{\varepsilon},
	\end{align}
	where $ \varepsilon $ is as in \cref{eq:perturb} and $ C \geq 1 $ is some universal constant. 
\end{assumption}
Under \cref{assmpt:perturb}, unless the perturbations are acute, i.e., rank preserving, obtaining \cref{eq:pseudo_inverse} is simply hopeless. In this light, instead of considering the distance between $ \HHd $ and $ \tHHd $ viewed as linear operators, one can perhaps consider the distance between the subspaces spanned by them respectively. More specifically, instead of \cref{eq:pseudo_inverse},  one could attempt at finding conditions in \cref{eq:perturb} such that 
\begin{align}
\label{eq:projection}
\vnorm{\UU \UUt - \tUU \tUUt} \leq  \tilde{\varepsilon}_{3},
\end{align}
where $ \UU $ and $ \tUU $ are orthonormal bases for $ \range(\tHH) $ and $ \range(\tHH) $, respectively. In other words, at first sight, the quantity of interest could be the distance  between  the two subspaces, namely $ \range(\HH) $ and its perturbation $ \range(\tHH) $ \cite[Section 2.5.3]{golub2012matrix}. More recent and improved results in bounding such distance are given in \cite{o2018random}. For simplicity we include the statement of \cite[Theorem 19]{o2018random}, but only slightly modified to fit the settings that we consider here.
\begin{theorem}[{Modified Davis-Kahan-Wedin Sine Theorem \cite[Theorem 19]{o2018random}}]
	\label{thm:proj}
	\vspace*{-2mm}
	Consider a symmetric matrix $ \AA \in \reals^{d \times d} $ of rank $ r $, and let $ \tAA $ be its symmetric perturbation. For an integer $ 1 \leq j \leq r $, let $ \UU_j \defeq \left[\uu_1,\dots,\uu_j\right] \in \reals^{d \times j} $ and $ \tUU_j \defeq \left[ \tuu_1,\dots,\tuu_j \right]  \in \reals^{d \times j} $, where $ \uu_{i} \in \reals^{d} $ and $ \tuu_{i} \in \reals^{d} $ are, respectively, $ i\th $ eigenvectors of matrices $ \AA $ and $ \tAA $. The principal angle between $ \range(\UU_j) $ and $ \range(\tUU_j) $ is given by 
	\begin{align}
	\label{eq:davis_khan}
	\sin \angle \left( \range(\UU), \range(\tUU)\right) = \vnorm{\UU_j \UU_{j}^{\intercal} - \tUU_j \tUU_{j}^{\intercal}} \leq \frac{2 \|\AA - \tAA\|}{\sigma_{j}(\AA) - \sigma_{j+1}(\AA)}.
	\end{align}
\end{theorem}
%Recall Davis-Kahan theorem
%	\begin{align*}
%	\vnorm{\UU \UUt - \tUU \tUUt} \leq \frac{\|\AA - \tAA\|}{\sigma_{j} - \sigma_{j+1}}
%	\end{align*}
%	Then by \cref{eq:sigma_tU1},
%	\begin{align*}
%	\vnorm{\UU \UUt - \tUU \tUUt} \leq \frac{\varepsilon}{\gamma - \varepsilon}
%	\end{align*}
%	When $\varepsilon < \gamma / 2$, 
%	\begin{align*}
%	\frac{\varepsilon}{\gamma - \varepsilon} \leq \frac{2 \varepsilon}{\gamma}.
%	\end{align*}

As a result, if $ \tHH $ is an acute perturbation of $ \HH $, i.e., $ \rank(\tHH) = \rank(\HH) = r $, we can appeal to \cref{thm:proj} and obtain a bound as in \cref{eq:projection}. Indeed, since $ \sigma_{r+1}(\HH) = 0 $ in this case, we get 
\begin{align*}
\vnorm{\UU \UUt - \tUU \tUUt} \leq \frac{2\vnorm{\HH - \tHH}}{\sigma_{r}(\HH)} = 2 \vnorm{\HH - \tHH} \vnorm{\HHd} \leq 2 \vnorm{\HHd}  \varepsilon.
\end{align*} 
However, from \cite[Facts 5.12.17(iv) and 9.9.29]{bernstein2009matrix}, it simply follows that 
\begin{align*}
\rank(\HH) \neq \rank(\tHH) \Longrightarrow \vnorm{\tUU \tUUt - \UU \UUt} = 1.
\end{align*}
Again as before, requiring \cref{eq:projection} in non-convex settings is indeed far too stringent. 

What comes to the rescue is the observation that instead of obtaining a bound as in \cref{eq:projection}, which implies a bounded distance between $ \UU \UUt $ and $ \tUU \tUUt $ along \emph{every} direction, for Newton-MR, we only need such distance to be bounded along a \emph{specific} direction, i.e., that of the gradient $ \bgg $. Indeed, instead of \cref{eq:projection}, which implies 
\begin{align*}
\vnorm{\left( \HH \HHd - \tHH \tHHd \right) \vv} \leq \teps \vnorm{\vv}, \quad \forall \; \vv \in \reals^{d},
\end{align*}
by only considering $ \vv = \bgg $ in the above, we seek to bound only the projection of the gradient on the range space of the Hessian matrices as
\begin{align}
\label{eq:pseudo_inverse_g}
\vnorm{\left( \HH \HHd - \tHH \tHHd \right) \bgg} \leq \teps \vnorm{\bgg}.
\end{align}

We now set out to obtain conditions that can guarantee \cref{eq:pseudo_inverse_g}. Our result relies on the following lemma (\cref{lemma:rank}), which establishes a relationship between $ \varepsilon $ in \cref{eq:perturb} and $ \rank(\tHH) $.
Although assuming rank-preserving perturbation, i.e., $ \rank(\HHk) = \rank(\tHHk) $, is too stringent to be of any practical use, under certain conditions, we can ensure that the perturbed matrix has a rank at least as large as the original matrix, i.e., perturbation is such that the rank is, at least, not reduced.
\begin{lemma}[Rank of Perturbed Hessian]
	\label{lemma:rank}
	\vspace*{-2mm}
	Under \cref{assmpt:pseudo_regularity}, if $ \varepsilon < \gamma $ in \cref{eq:perturb}, then $ \rank(\tHH) \geq \rank(\HH) $.
\end{lemma}
\begin{proof} 
	\label{pf:rank}
	First note that from \cref{eq:perturb}, it follows that $\lambda_{\min}(\tHH - \HH) \geq -\varepsilon$, and $\lambda_{\max}(\tHH - \HH) \leq \varepsilon$. Let $ r = \rank(\HH) $ and $ \tr = \rank(\tHH) $. By \cite[Theorem 8.4.11]{bernstein2009matrix}, for any $ 1 \leq j \leq r $, we have
	\begin{align*}
	\lambda_j(\HH) + \lambda_{\min}(\tHH - \HH) \leq \lambda_j(\tHH) \leq \lambda_j(\HH) + \lambda_{\max}(\tHH - \HH), 
	\end{align*}
	which implies
	\begin{align}
	\label{eq:wyle}
	\lambda_j(\HH) - \varepsilon &\leq \lambda_j(\tHH) \leq \lambda_j(\HH) + \varepsilon.
	\end{align}
	Now \cref{assmpt:pseudo_regularity} implies that $ \abs{\lambda_j(\HH)} \geq \gamma$, $1 \leq j \leq r $. Hence, we get 
	\begin{align*}
	\lambda_j(\tHH) &\leq \lambda_j(\HH) + \varepsilon \leq -\gamma + \varepsilon < 0, \quad \text{if} \quad \lambda_j(\HH) < 0, \\
	\lambda_j(\tHH) &\geq \lambda_j(\HH) - \varepsilon \geq \gamma  - \varepsilon > 0, \quad \text{if} \quad \lambda_j(\HH) > 0.
	\end{align*}
	Thus, it follows that $ \lambda_j(\tHH) \neq 0, \forall j = 1,\dots,r $, which implies that $ \rank(\tHH) \geq \rank(\HH) $.
\end{proof}

\Cref{lemma:rank} states that if $ \varepsilon < \gamma $, the perturbed Hessian is never of lower rank than the original Hessian. In other words, if $ \rank(\tHH) < \rank(\HH) $, then we must necessarily have that $ \varepsilon \geq \gamma $. %\Cref{lemma:rank} is also the reason of defining SVDs in \Cref{sec:notation}.
Also, from the proof of \cref{lemma:rank} and using the fact that $ \tHH $ is symmetric, we have
\begin{align}
\label{eq:sigma_tU1}
\sigma_{i} ( \tHH ) \geq \gamma - \varepsilon, \quad i = 1,\ldots, r,
\end{align}
where $ \gamma  $ is as in \cref{assmpt:pseudo_regularity}. Furthermore, if $ \varepsilon < \gamma $, we have $ r < \tr $, which by \cref{eq:wyle,eq:epsilon_02} yields
\begin{align}
\label{eq:sigma_tU2}
\frac{\varepsilon}{C} \leq \sigma_{i} ( \tHH ) \leq \varepsilon, \quad i = r+1,\ldots, \tr.
\end{align}

For the remainder of this paper, with $ r \leq \tr $, we let the singular value decomposition of $\HH$ and $\tHH$ be
\begin{align}
\label{eq:svd}
\HH &= \begin{bmatrix}
\UU &\UU_{\perp}
\end{bmatrix} \begin{bmatrix}
\Sigma & 0\\
0 & 0\\
\end{bmatrix} \VVt,\quad \text{and} \quad \tHH = \begin{bmatrix}
\undermat{\tUU}{\tUU_1 & \tUU_2} &\tUU_{\perp}
\end{bmatrix} \begin{bmatrix}
\tSigma_{1} & 0 & 0\\
0 & \tSigma_{2} & 0\\
0 & 0 & 0\\
\end{bmatrix} \tVVt,
\end{align}
where
\begin{align*}
&\UU \in \reals^{d \times r}, \Sigma \in \reals^{r \times r}, \VV \in \reals^{d \times d}, \\
&\tUU_{1} \in \reals^{d \times r}, \tUU_{2} \in \reals^{d \times (\tr-r)}, \tSigma_{1} \in \reals^{r \times r}, \tSigma_{2} \in \reals^{(\tr-r) \times (\tr-r)}, \tVV \in \reals^{d \times d}.
\end{align*}
Note that $ \tUU \in \reals^{d \times \tr} $ is divided into $ \tUU_{1} \in \reals^{d \times r} $ and $ \tUU_{2} \in \reals^{d \times (\tr-r)} $. Furthermore,  $ \UU $ and $\tUU_{1} $ have the same rank. Now, using \Cref{lemma:rank}, we can establish \cref{eq:pseudo_inverse_g}.
\begin{theorem}
	\label{thm:sub_proj}
	\vspace*{-2mm}
	Under \cref{assmpt:pseudo_regularity,assmpt:null_space}, and with $ \varepsilon < \gamma $ in \cref{eq:perturb}, we have \cref{eq:pseudo_inverse_g} with 
	\begin{align*}
	\teps \triangleq \frac{4 \varepsilon}{\gamma} + \sqrt{1 - \nu}.
	\end{align*}
	Furthermore, in special case of acute perturbation, i.e., $ \rank(\HH) = \rank(\tHH) $, we have
	\begin{align*}
	\teps \triangleq \frac{2 \varepsilon}{\gamma}.
	\end{align*}
	Here, $ \gamma $ and $ \nu $ are as in \cref{assmpt:pseudo_regularity,assmpt:null_space}, respectively.
\end{theorem}
\begin{proof} 
	By assumption on $ \varepsilon $, \cref{lemma:rank} gives $ r \triangleq \rank(\HH) \leq \rank(\tHH) \triangleq \tr $. By \cref{eq:pseudo_regularity_H}, \cref{eq:null_space_grad_2}, and  \cref{eq:davis_khan}, we will have
	\begin{align*}
	\vnorm{\HH \HHd \bgg - \tHH \tHHd \bgg} &= \vnorm{\UU \UUt \bgg - \tUU \tUUt \bgg} \leq \vnorm{\UU \UU^{\intercal} \bgg - \tUU_1 \tUU_1^{\intercal} \bgg} + \vnorm{\tUU_2 \tUU_{2}^{\intercal} \bgg}\\
	&\leq \vnorm{\UU \UU^{\intercal} \bgg - \tUU_1 \tUU_1^{\intercal} \bgg} + \vnorm{\tUU_2 \tUU_2^{\intercal} \bgg + \tUU_{\perp} \tUU_{\perp}^{\intercal} \bgg}\\
	&= \vnorm{\UU \UU^{\intercal} \bgg - \tUU_1 \tUU_1^{\intercal} \bgg} + \vnorm{\tUU_2 \tUU_2^{\intercal} \bgg + \tUU_{\perp} \tUU_{\perp}^{\intercal} \bgg - \UUp \UU^{\intercal}_{\perp} \bgg + \UUp \UU^{\intercal}_{\perp} \bgg}\\
	&= \vnorm{\UU \UU^{\intercal} \bgg - \tUU_1 \tUU_1^{\intercal} \bgg} + \vnorm{ (\eye - \UUp \UU^{\intercal}_{\perp}) \bgg - (\eye - \tUU_2 \tUU_2^{\intercal} - \tUU_{\perp} \tUU_{\perp}^{\intercal}) \bgg+ \UUp \UU^{\intercal}_{\perp} \bgg} \\
	&\leq \vnorm{\UU \UU^{\intercal} \bgg - \tUU_1 \tUU_1^{\intercal} \bgg} + \vnorm{ (\eye - \UUp \UU^{\intercal}_{\perp}) \bgg - (\eye - \tUU_2 \tUU_2^{\intercal} - \tUU_{\perp} \tUU_{\perp}^{\intercal}) \bgg} + \vnorm{\UUp \UU^{\intercal}_{\perp} \bgg} \\
	&\leq 2\vnorm{\UU \UU^{\intercal} \bgg - \tUU_1 \tUU_1^{\intercal} \bgg} + \vnorm{\UUp \UU^{\intercal}_{\perp} \bgg} \\
	&\leq 4\vnorm{\HH - \tHH}\vnorm{\HHd}\vnorm{\bgg} + \sqrt{1 - \nu} \vnorm{\bgg} \\
	&\leq \left(\frac{4 \varepsilon}{\gamma} + \sqrt{1 - \nu} \right) \vnorm{\bgg}, 
	\end{align*}
	where in the second inequality, we used the Pythagorean theorem as
	\begin{align*}
	\vnorm{\tUU_2 \tUU_2^{\intercal} \bgg + \tUU_{\perp} \tUU_{\perp}^{\intercal} \bgg}^{2} = \vnorm{\tUU_2 \tUU_2^{\intercal} \bgg}^{2} + \vnorm{\tUU_{\perp} \tUU_{\perp}^{\intercal} \bgg}^{2} \geq \vnorm{\tUU_2 \tUU_2^{\intercal} \bgg}^{2},
	\end{align*}
	and for the last equality we use the fact that $\UU \UU^{\intercal} + \UUp \UU^{\intercal}_{\perp} = \eye$, and $\tUU_1 \tUU_1^{\intercal} + \tUU_2 \tUU_2^{\intercal} + \tUU_{\perp} \tUU_{\perp}^{\intercal} = \eye$. The special case of acute perturbation follows simply from the above without the term involving $ \| \tUU_2 \tUU_{2}^{\intercal} \bgg \| $.
\end{proof}

\cref{thm:sub_proj} also allows us to obtain results similar in spirit to \cref{lemma:null_space_grad}.
\begin{lemma}
	\label{lemma:null_space_tH}
	\vspace*{-2mm}
	Under assumptions of \cref{thm:sub_proj}, we have
	\begin{subequations}
		\label{eq:null_space_grad_pert}  
		\begin{align}
		\vnorm{\tUUt \bgg}^{2} &\geq \tnu \vnorm{\bgg}^{2}, \label{eq:null_space_grad_pert_01}  \\
		\vnorm{\tUU_{\perp}^{\intercal} \bgg}^{2} &\leq (1- \tnu) \vnorm{\bgg}^{2}, \label{eq:null_space_grad_pert_02}  
		\end{align}
	\end{subequations}
	where 
	\begin{align*}
	\tnu &\triangleq 2 \nu - 1 - \frac{4 \varepsilon}{\gamma},
	\end{align*}
	and $ 0.5 < \nu \leq 1$, $\varepsilon < \gamma(2\nu - 1)/4 $.
	Furthermore, in the special case of acute perturbation, i.e., $ \rank(\HH) = \rank(\tHH) $, we have \cref{eq:null_space_grad_pert} with 
	\begin{align*}
	\tnu &\triangleq \nu - \frac{2 \varepsilon}{\gamma},
	\end{align*}
	where $ 0 < \nu \leq 1$, $\varepsilon < \gamma \nu /2 $.
	Here, $ \gamma $ is as in \cref{eq:pseudo_regularity},  and $ \tUU , \tUU_{\perp} $ are as in \cref{eq:svd}.
\end{lemma}

\begin{proof}
	We have.
	\begin{align*}
	\vnorm{\tUUt \bgg}^{2} = \bgg^{\intercal} \tUU \tUUt \bgg = \bgg^{\intercal} \UU \UUt \bgg - \bgg^{\intercal} \left( \UU \UUt - \tUU \tUUt \right) \bgg.
	\end{align*}
	Similarly to the proof of \cref{thm:sub_proj} and using \cref{lemma:null_space_grad}, we have
	\begin{align*}
	\abs{\bgg^{\intercal} \UU \UUt \bgg - \bgg^{\intercal} \tUU \tUUt \bgg} &\leq \abs{\bgg^{\intercal} \UU \UU^{\intercal} \bgg - \bgg^{\intercal} \tUU_1 \tUU_1^{\intercal} \bgg} + \bgg^{\intercal} \tUU_2 \tUU_{2}^{\intercal} \bgg \leq \frac{2 \varepsilon}{\gamma} + \bgg^{\intercal} \tUU_2 \tUU_{2}^{\intercal} \bgg \\
	&\leq \frac{2 \varepsilon}{\gamma} + \abs{\bgg^{\intercal} \UU \UU^{\intercal} \bgg - \bgg^{\intercal} \tUU_1 \tUU_1^{\intercal} \bgg} + \bgg^{\intercal} \UUp \UU^{\intercal}_{\perp} \bgg \leq \left( \frac{4 \varepsilon}{\gamma} + 1 - \nu \right) \vnorm{\bgg}^2. 
	\end{align*}
	Now, it follows that
	\begin{align*}
	\vnorm{\tUUt \bgg}^{2} \geq \bgg^{\intercal} \UU \UUt \bgg - \left(\frac{4 \varepsilon}{\gamma} + 1 - \nu \right) \vnorm{\bgg}^2 \geq \left(2\nu - 1 - \frac{4 \varepsilon}{\gamma} \right) \vnorm{\bgg}^2,
	\end{align*}
	which gives \cref{eq:null_space_grad_pert_01}. Now noting that $ \|\bgg\|^{2} = \|\tUUt \bgg\|^{2}+\|\tUU_{\perp}^{\intercal} \bgg \|^{2} $, we get \cref{eq:null_space_grad_pert_02}. Finally, for the case of acute perturbations, since $ \tUU \tUUt = \tUU_1 \tUU_1^{\intercal}$ (recall that in this case $ \tSigma_{2} = \zero $), it follows that 
	\begin{align*}
		\abs{\bgg^{\intercal} \UU \UUt \bgg - \bgg^{\intercal} \tUU \tUUt \bgg} &= \abs{\bgg^{\intercal} \UU \UU^{\intercal} \bgg - \bgg^{\intercal} \tUU_1 \tUU_1^{\intercal} \bgg} \leq \frac{2 \varepsilon}{\gamma} \vnorm{\bgg}^{2}.
	\end{align*}
	which gives the desired result.
\end{proof}

As it is evident from \cref{thm:sub_proj}, situations where either $ \nu = 1 $ or the perturbation is acute, exhibit a certain inherent stability. More specifically, define the operator $ \bm{\pi}: \reals^{d \times d} \rightarrow \reals^{d}$ as $\bm{\pi}(\EE) \defeq (\HH + \EE) (\HH + \EE)^{\dagger} \bgg = \tHH \tHHd \bgg $, i.e., the projection of the gradient onto the range space of $ \tHH = \HH + \EE $. When $ \nu = 1 $ or the perturbation is acute, such a mapping is continuous at $ \EE = \zero $, i.e., $ \lim_{\vnorm{\EE} \to 0} \bm{\pi}(\EE) = \bm{\pi}(\zero) = \HH \HHd \bgg$. This observation gives rise to the following definition.

\begin{definition}[Inherent Stability]
	\label{def:stable}
	\vspace*{-2mm}
	A perturbation remains inherently stable if one of the following conditions hold:
	\begin{itemize}
		\item $ \tHH $ is an acute perturbation of $ \HH $, i.e., $ \rank(\tHH) = \rank(\HH) $, or
		\item $ \nu = 1 $ with $ \nu $ as in \cref{assmpt:null_space}, i.e., $ \bgg \in \range(\HH)$.
	\end{itemize}
\end{definition}
In light of \cref{def:stable} and \cref{lemma:null_space_tH}, we end this section by specifically stating the perturbation regimes where we develop our convergence theory of \Cref{sec:stability}.
\begin{condition}
	\label{cond:stable}
	\vspace*{-2mm}
	We consider two perturbation regimes:
	\begin{itemize}
		\item For general perturbations, we consider \cref{eq:perturb} with $\varepsilon < \gamma(2\nu - 1)/4 $ for $ 0.5 < \nu \leq 1 $. 
		\item For inherently stable perturbations, we consider \cref{eq:perturb} with $\varepsilon < \gamma \nu /2 $ for $ 0 < \nu \leq 1 $.
	\end{itemize}
	Here,  $ \gamma $ and $ \nu $ are, respectively, as in \cref{assmpt:pseudo_regularity,assmpt:null_space}. 
\end{condition}

\section{Newton-MR with Hessian Approximations}
\label{sec:stability}
In this section, we provide convergence analysis of Newton-MR with inexact Hessian. To do so, we first briefly review Newton-MR in \Cref{sec:nt-MR} and, in its light, introduce the variant in which the Hessian is approximated (\cref{alg:Newton_invex_sub}). This is then followed by its convergence analysis in \Cref{sec:convergence}. 
For the special case of strongly-convex problems, the convergence results of \Cref{sec:convergence} bear a strong resemblance to those of the classical Newton's method (and Newton-CG variant) with inexact Hessian, e.g., \cite{ssn2018,bollapragada2016exact}. These comparisons are made in more details in \Cref{sec:comparison}. 
%Of course, \Cref{sec:convergence} established the theoretical result of \cref{alg:Newton_invex_sub} in more general settings.

\subsection{Newton-MR Algorithm: Review}
\label{sec:nt-MR}
We now briefly review Newton-MR as it was introduced in \cite{roosta2018newton}.
In non-convex settings, the Hessian matrix could be indefinite and possibly rank-deficient. In this light, at the $ k\th $ iteration, Newton-MR in its pure form involves the exact update direction of the form
\begin{align}
\label{eq:newton_mr_pseudo}
\ppk = -\HHdk \bggk.
\end{align}
The exact update direction \cref{eq:newton_mr_pseudo} can be equivalently written as the least norm solution to the least squares problem $ \| \bggk + \HHk \pp \| $, i.e.,
\begin{align}
\label{eq:least_norm_solution}
\min_{\pp \in \reals^{d}} ~~ \|\pp\| \quad \text{subject to} \quad \pp \in \Argmin_{\widehat{\pp} \in \mathbb{R}^{d}} \vnorm{\HHk \widehat{\pp} + \bggk}.
\end{align}
In practice, computing the Moore-Penrose generalized inverse can be computationally prohibitive, in which case the inexact variant of Newton-MR makes use of approximate update direction as
\begin{align}
\label{eq:least_norm_solution_relaxed}
\text{Find } \ppk \in \range(\HHk), \quad \text{subject to} \quad \dotprod{\ppk,\HHk \bggk} \leq -(1-\theta) \vnorm{\bggk}^{2},
\end{align}
where $ \theta < 1 $ is the inexactness tolerance. It is easy to see that \cref{eq:least_norm_solution_relaxed} is implied by \cref{eq:least_norm_solution}. When $ \bggk \in \range(\HHk) $, i.e., the linear system $ \HHk \pp = - \bggk $ is consistent, MINRES \cite{paige1975solution} can be used to obtain (approximate) pseudo-inverse solution. However, due to its many desirable properties, MINRES-QLP \cite{choi2011minres} has been advocated in \cite{roosta2018newton} for more general cases as the preferred solver for \cref{eq:least_norm_solution} or \cref{eq:least_norm_solution_relaxed}. When the Hessian is perturbed, even if initially $ \bggk \in \range(\HHk) $, it is generally most likely that $ \bggk \notin \range(\tHHk) $, and hence MINRES-QLP remains the method of choice for our setting here.

After computing the update direction, the next iterate is obtained by moving along $ \ppk $ by some appropriate step length, i.e., $ \xxkk = \xxk + \alphak \ppk $. Note that from both \cref{eq:least_norm_solution,eq:least_norm_solution_relaxed} it follows that $ \dotprod{\ppk, \HHk \bggk} \leq 0 $, i.e., $ \ppk $ is a descent direction for the norm of the gradient, $ \| \bgg \|^{2} $. As a result, the step-size, $ \alphak $, can be chosen by applying Armijo-type line-search \cite{nocedal2006numerical} such that for some $ 0 \leq \alpha_k \leq 1 $, we have 
\begin{align}
\label{eq:armijo_gen}
\vnorm{\bggkk}^{2} \leq \vnorm{\bggk}^{2} + 2 \rho \alpha_k \dotprod{\ppk, \HHk \bggk},
\end{align}
where $ 0< \rho < 1 $ is a given line-search parameter. Typically, back-tracking strategy \cite{nocedal2006numerical} is employed to approximately find such a step-size.

Modification of Newton-MR to include the perturbed matrix $ \tHH $ as in \cref{eq:epsilon} is rather straightforward. We simply replace $ \HHk $ with $ \tHHk $ in all of \cref{eq:least_norm_solution,eq:least_norm_solution_relaxed,eq:armijo_gen}. The modified variant is depicted in \cref{alg:Newton_invex_sub}. Note that, in this context, whenever we refer to \cref{eq:least_norm_solution}, \cref{eq:least_norm_solution_relaxed} and \cref{eq:armijo_gen}, it is implied that $ \tHH $ is used instead of $ \HH $.

\begin{algorithm}
	\caption{Newton-MR With Inexact Hessian Information \label{alg:Newton_invex_sub}}
	\begin{algorithmic}[1]
		\vspace{1mm}
		\STATE \textbf{Input:} $ \xx_{0} $, $ 0 < \tau < 1 $, $ 0 < \rho < 1 $
		\vspace{1mm}
		\FOR {$ k = 0,1,2, \ldots $ until $ \| \bggk \| \leq \tau $} 
		%		\vspace{1mm}
		%		\STATE Form $ \tHHk $ such that $ \vnorm{\HHk - \tHHk} \leq \varepsilon $
		\vspace{1mm}
		\STATE Solve \eqref{eq:least_norm_solution} (or \eqref{eq:least_norm_solution_relaxed} with MINRES-QLP) with $ \tHHk $ in place of $ \HHk $
		\vspace{1mm}
		\STATE Find $ \alphak $ such that \eqref{eq:armijo_gen} holds with $ \tHHk $ in place of $ \HHk $
		\vspace{1mm}
		\STATE Update $ \xx_{k+1}  =  \xxk + \alphak \ppk $
		\vspace{1mm}
		\ENDFOR
		\vspace{1mm}
		\STATE \textbf{Output:} $ \xx $ for which $ \| \bggk \| \leq \tau $
	\end{algorithmic}
\end{algorithm}

\subsection{Convergence Analysis}
\label{sec:convergence}
In this section, we give the convergence analysis of \cref{alg:Newton_invex_sub}. For this, in \Cref{sec:exact} we first consider the exact update \cref{eq:least_norm_solution}, and subsequently in \Cref{sec:inexact}, consider the inexact variant where the update direction is computed approximately using \cref{eq:least_norm_solution_relaxed}.

\subsubsection{Exact Updates}
\label{sec:exact}	

A major ingredient in establishing the convergence of \cref{alg:Newton_invex_sub} using \cref{eq:least_norm_solution} is to obtain an upper-bound on the norm of the exact updates, $ \pp = - \tHHd \bgg $. This indeed is crucial in light of \cref{eq:epsilon_02}, which implies that $ \tHHd $ can grow unbounded as $ \varepsilon \downarrow 0 $. However, as in \cref{thm:sub_proj}, we can obtain a bound, which fits squarely into the notion of inherently stable perturbations from \cref{def:stable}. 

\begin{lemma}[Stability of Pseudo-inverse of Perturbed Hessian]
	\label{lemma:pseudo_regularity_tH}
	\vspace*{-2mm}
	Under \cref{assmpt:perturb} as well as Assumptions of \cref{thm:sub_proj}, we have
	\begin{align*}
	\vnorm{\tHHd \bgg} \leq \frac{1}{\tgamma} \vnorm{\bgg},
	\end{align*}
	where 
	\begin{align*}
	%	\label{eq:tgamma}
	\tgamma &\triangleq \left( \frac{1}{\gamma - \varepsilon} + C \left( \frac{2}{\gamma} +  \frac{\sqrt{1 - \nu}}{\varepsilon}\right) \right)^{-1}.
	\end{align*}
	Furthermore, in special case of acute perturbation, i.e., $ \rank(\HH) = \rank(\tHH) $, we have
	\begin{align*}
	%	\label{eq:tgamma_acute}
	\tgamma \triangleq \gamma - \varepsilon.
	\end{align*}	
	Here, $ \gamma, \nu, \varepsilon $ and $ C $ are, respectively, as in \cref{eq:perturb,eq:epsilon_02,eq:pseudo_regularity,eq:null_space}. 
\end{lemma}
\begin{proof}
	Consider the SVD of $ \tHHd $ as in \cref{eq:svd}. Note that \cref{eq:sigma_tU1} implies that $ \| \tSigma_{1}^{-1} \| \leq 1/(\gamma - \varepsilon) $. Further, from \cref{eq:sigma_tU2}, it follows that $ \| \tSigma_{2}^{-1} \| \leq C/\varepsilon $. Now, it follows that
	\begin{align*}
	\vnorm{\tHHd \bgg}^{2} &= \vnorm{\tVV \begin{bmatrix}
		\tSigma_{1}^{-1} & 0 & 0\\
		0 & \tSigma_{2}^{-1} & 0\\
		0 & 0 & 0\\
		\end{bmatrix} \begin{bmatrix}
		\tUU_{1}^{\intercal} \\
		\tUU_{2}^{\intercal}\\
		\tUU_{\perp}^{\intercal}
		\end{bmatrix} \bgg}^{2} \\
	&= \vnorm{\tSigma_{1}^{-1} \tUU_{1}^{\intercal} \bgg}^{2} +  \vnorm{\tSigma_{2}^{-1} \tUU_{2}^{\intercal} \bgg}^{2} \\
	&\leq \frac{1}{(\gamma - \varepsilon)^{2}} \vnorm{\bgg}^{2} + C^{2} \left(\frac{2}{\gamma} + \frac{\sqrt{1 - \nu}}{\varepsilon}\right)^{2} \vnorm{\bgg}^{2},
	\end{align*}
	where  the last inequality is obtained using the bound on $ \vnorm{\tUU_{2}^{\intercal} \bgg} = \vnorm{\tUU_{2} \tUU_{2}^{\intercal} \bgg} $ as in the proof of \cref{thm:sub_proj}. The result follows from the inequality $ \sqrt{a^{2} + b^{2}} \leq a + b, \; \forall a,b \geq 0 $.
\end{proof}

\begin{figure}[!htbp]
\centering
\vspace*{-5mm}
\subfigure[$ \nu = 1 $]{\includegraphics[scale=0.16]{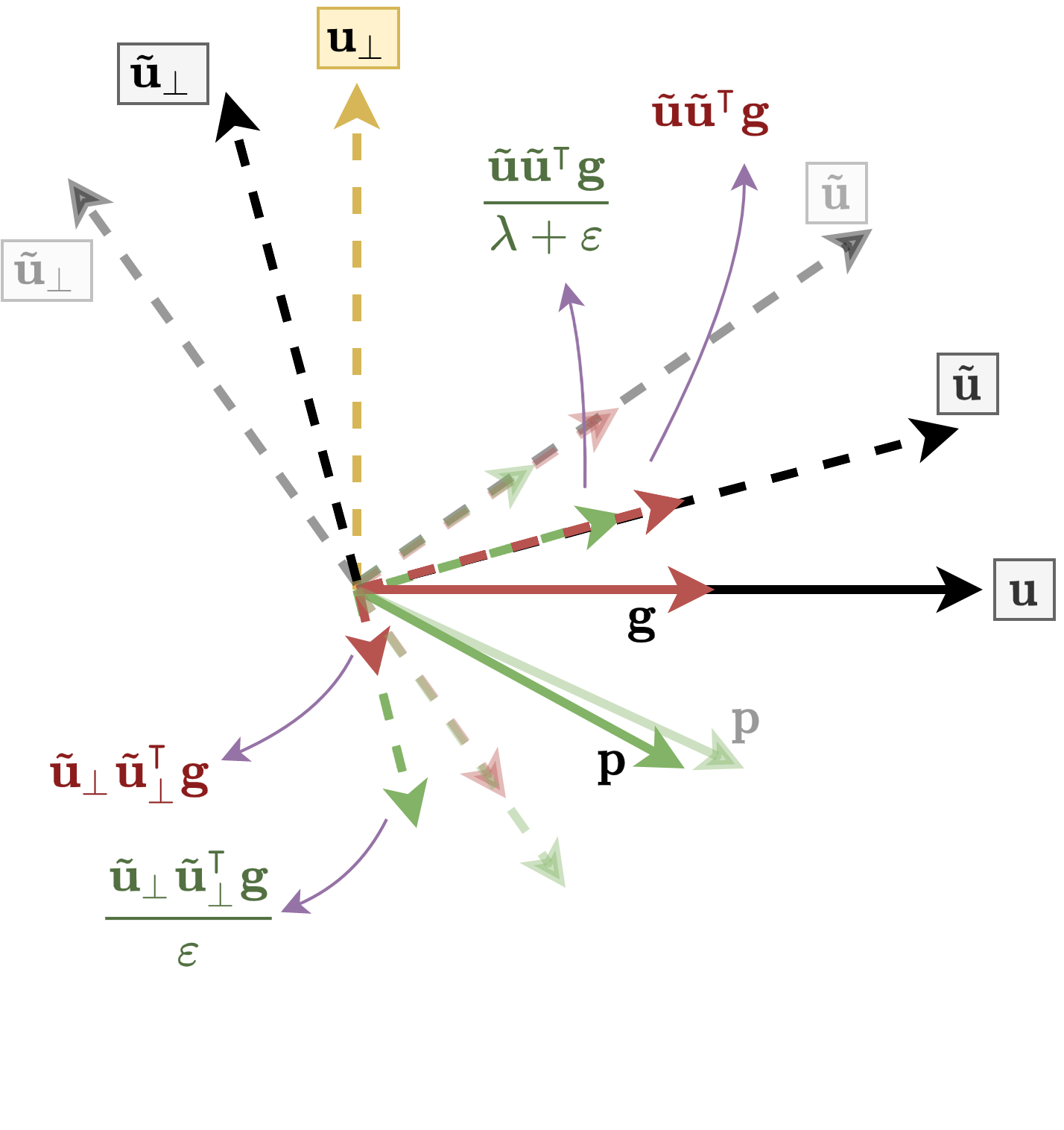}}\hspace{15mm}
\subfigure[$ \nu < 1 $]{\includegraphics[scale=0.16]{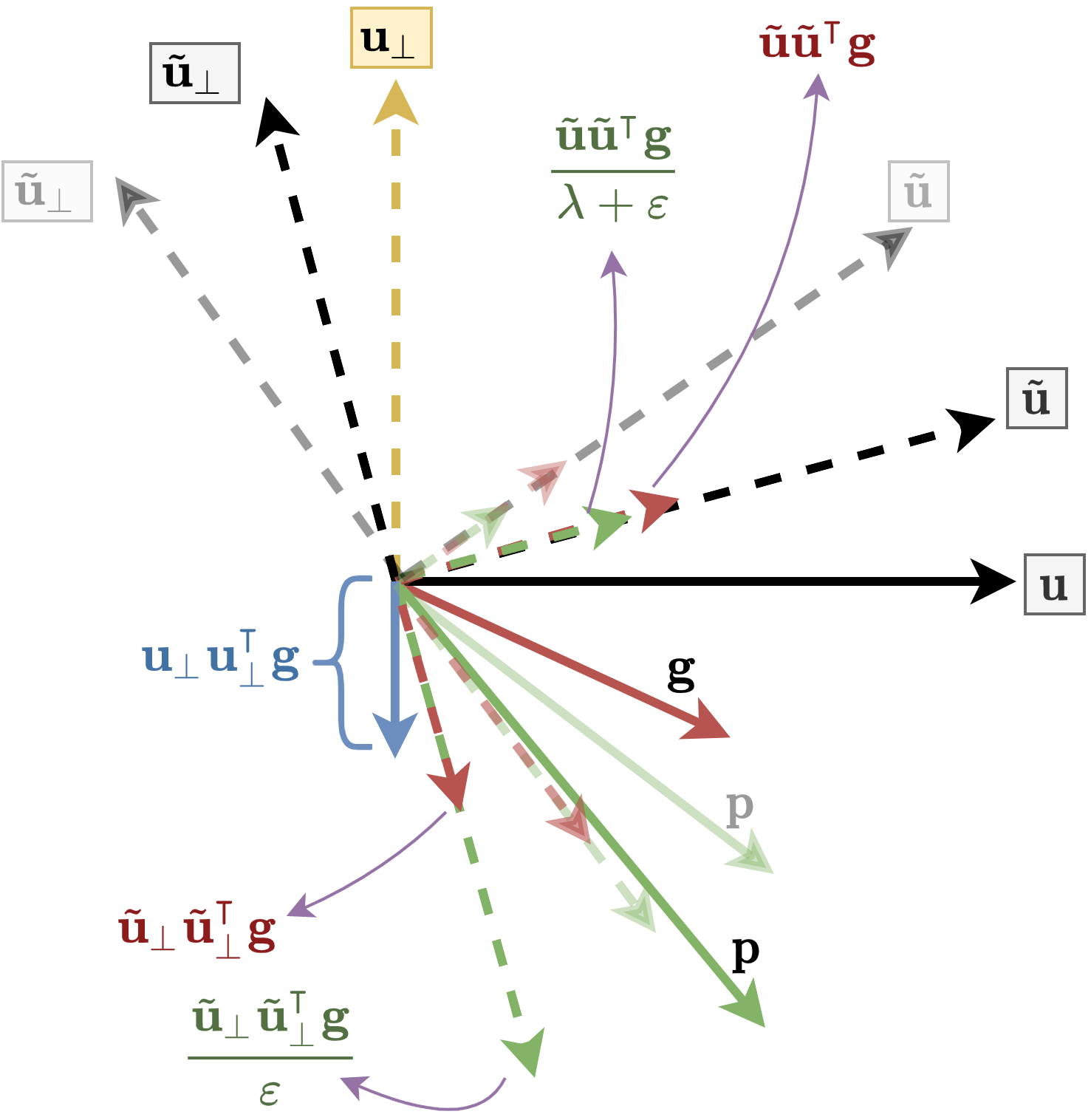}}
\caption{Illustration of \cref{lemma:pseudo_regularity_tH} for non-acute perturbations when $ \rank(\HH) = 1 $ and $ d = 2 $. Here $ \Span\{\uu\} = \range(\HH) $, $ \Span\{\uu_{\perp}\} = \Null(\HH) $, $ \lambda $ is the non-zero eigenvalue of $ \HH $, $ \{\tilde{\uu},\tilde{\uu}_{\perp}\} $ are the eigenvectors of $ \tHH $ and $ \rank(\tHH) = 2 $.  Also, $ \pp = \tHHd \bgg = \tilde{\uu} \tilde{\uu}^{\intercal} \bgg/\tilde{\lambda} + \tilde{\uu}_{\perp} \tilde{\uu}_{\perp}^{\intercal} \bgg/\varepsilon $, where $ \tilde{\lambda} \in \lambda \pm \bigO{\varepsilon} $ (here, we assume $ \lambda \geq 1$ and, for simplicity, $\tilde{\lambda} = \lambda +\varepsilon $). The transparent arrows depict the case for the larger $ \varepsilon $ while the opaque ones correspond to the smaller $ \varepsilon $. (a) When $ \nu = 1 $, $ \bgg $ lies entirely along $ \uu $. Hence, its projections on $ \tilde{\uu}_{\perp} $ will shrink to zero as $ \varepsilon \downarrow 0 $, and $ \lim_{\varepsilon \downarrow 0} \vnorm{\tilde{\uu}_{\perp} \tilde{\uu}_{\perp}^{\intercal} \bgg}/\varepsilon \in \bigO{1} $. (b) When $ \nu < 1 $, $ \bgg $ has non-zero components along $ \uu_{\perp} $. Hence, its projections on $ \tilde{\uu}_{\perp} $ will be bounded away from zero, i.e., $ \lim_{\varepsilon \downarrow 0} \tilde{\uu}_{\perp} \tilde{\uu}_{\perp}^{\intercal} \bgg = \uu_{\perp} \uu_{\perp}^{\intercal} \bgg $, and $ \vnorm{\tilde{\uu}_{\perp} \tilde{\uu}_{\perp}^{\intercal} \bgg}/\varepsilon \xrightarrow[\varepsilon \downarrow 0]{} \infty$. As a result, when $ \varepsilon \downarrow 0 $, this component of $ \pp $ grows unboundedly. \label{fig:pert}}
\end{figure}

\begin{remark}
	\vspace*{-2mm}
	As indicated in \Cref{sec:pert}, obtaining a bound similar to \cref{eq:pseudo_regularity_H} for $ \tHH $ is entirely dependent on having $ \rank(\HH) = \rank(\tHH) $, which is too stringent to require. \cref{lemma:pseudo_regularity_tH} indicates that, when $ \rank(\HH) \neq \rank(\tHH) $, despite the fact that $ \| \tHHd \| $ can grow unbounded as $ \varepsilon \downarrow 0 $, if $ \nu = 1 $, we are guaranteed that the action of $ \tHHd $ on $ \bgg $, i.e., the exact Newton-MR direction, indeed remains bounded. 
	If $ \nu < 1 $ and $ \rank(\HH) \neq \rank(\tHH) $, then $ \| \tHHd \bgg \| $ can become increasingly larger with smaller perturbations, resulting in an algorithm that might no longer converge; see the numerical examples of \Cref{sec:frac}. 
	An intuitive illustration of this phenomenon is also depicted in \cref{fig:pert}.
\end{remark}

For our convergence proofs, we frequently make use of the following result.
\begin{lemma}[{\!\!\cite[Lemma 10]{roosta2018newton}}]
	\label{lemma:aux}
	\vspace*{-2mm}
	Consider any $ \xx,\zz \in \mathbb{R}^{d} $, $ 0 \leq L < \infty $ and $ h: \mathbb{R}^{d} \rightarrow \mathbb{R} $. If it holds that $\vnorm{\nabla h(\yy) - \nabla h(\xx)} \leq L \vnorm{\yy - \xx}, \; \forall \yy \in \left[\xx, \zz\right]$, then, we have $h(\yy) \leq h(\xx) + \dotprod{\nabla h(\xx), \yy - \xx} + L \vnorm{\yy - \xx}^2/2, \; \forall \yy \in \left[\xx, \zz\right]$.
\end{lemma}
If we take $ h(\xx) = \| \bgg(\xx) \|^{2}/2 $, the constant $ L $ in \cref{lemma:aux} is $ L(\xx_{0}) $ in \cref{assmpt:lipschitz_special}.

We now establish a general structural result, which allows for obtaining sufficient conditions for convergence of \cref{alg:Newton_invex_sub}.
\begin{theorem}[\cref{alg:Newton_invex_sub} With Exact Updates]
	\label{thm:exact} 
	\vspace*{-2mm}
	Under \cref{assmpt:perturb,assmpt:moral_diff,assmpt:lipschitz_special,assmpt:pseudo_regularity,assmpt:null_space,cond:stable}, for the iterates of \cref{alg:Newton_invex_sub} with exact updates \cref{eq:least_norm_solution}, we have
	\begin{align*}
	\vnorm{\bggkk}^{2} \leq \left(1 - \eta\right) \vnorm{\bggk}^{2}, 
	\end{align*}
	where
	\begin{align*}
		\eta \defeq \max \left\{0,\frac{4 \rho \tnu \tgamma^{2}}{L(\xxo)}\left( (1 - \rho) \tnu - \frac{\varepsilon}{\tgamma}\right)\right\} \in [0,1],
	\end{align*}
	and $ \rho, L(\xx_{0}), \tnu$ and $\tgamma$ are, respectively, as in \cref{eq:armijo_gen}, \cref{assmpt:lipschitz_special}, and \cref{lemma:pseudo_regularity_tH,lemma:null_space_tH}.
\end{theorem}

\begin{proof} 
	\label{pf:exact}
	From \cref{lemma:aux} with $ \xx = \xxk $, $ \zz = \xxk + \ppk $, $ \yy = \xxk + \alpha \ppk $, and $ h(\xx) = \| \bgg(\xx) \|^{2}/2 $, we have 
	\begin{align}
	\label{eq:aux}
	\vnorm{\bgg_{k+1}}^{2} \leq \vnorm{\bggk}^{2} + 2 \alpha \dotprod{\ppk, \HHk \bggk} + \alpha^2 L(\xx_{0}) \vnorm{ \ppk }^2
	\end{align}   
	Now to obtain a sufficient condition on $ \alpha $ to satisfy \cref{eq:armijo_gen}, we consider the inequality 
	\begin{align*}
	2 \alpha \dotprod{\ppk, \HHk \bggk} + \alpha^2 L(\xx_{0}) \vnorm{ \ppk }^2 \leq 2 \rho \alpha \dotprod{\ppk, \tHHk \bggk},
	\end{align*}
	we used \cref{eq:aux} as upper bound on $ \vnorm{\bgg_{k+1}}^{2} $. 
	Rearranging gives
	\begin{align}
	\label{eq:alpha}
	\alpha \leq \frac{2 \left( \rho \dotprod{ \bggk, \tHHk \ppk} - \dotprod{\bggk, \HHk \ppk} \right)}{L(\xx_{0}) \vnorm{ \ppk }^2 }.
	\end{align}
	By \cref{eq:perturb,lemma:pseudo_regularity_tH,lemma:null_space_tH}, we have
	\begin{align*}
	\rho \dotprod{ \bggk, \tHHk \ppk} - \dotprod{\bggk, \HHk \ppk} &= -(1 - \rho)\dotprod{ \bggk, \tHHk \ppk} - \dotprod{\bggk, \left(\HHk - \tHHk\right) \ppk} \nonumber\\
	&\geq (1 - \rho) \tnu \vnorm{\bggk}^2 - \frac{\varepsilon}{\tgamma}\vnorm{\bggk}^2.
	\end{align*}
	If $ \varepsilon $ satisfies the inequality $ \varepsilon \leq (1 - \rho) \tnu \tgamma $, the lower bound on the step-size returned by line-search \cref{eq:armijo_gen} is $ \alphak \geq \alpha $ where
	\begin{align*}
		\alpha \defeq \frac{2 \tgamma^{2}}{L(\xx_{0})}\left( (1 - \rho) \tnu - \frac{\varepsilon}{\tgamma}\right).
	\end{align*}
	Otherwise, the lower bound is the trivial $ \alpha = 0 $.
	Now, from \cref{eq:armijo_gen} with the lower bound $ \alpha $, we get
	\begin{align*}
		\vnorm{\bggkk}^{2} &\leq \vnorm{\bggk}^{2} + 2 \rho \alphak \dotprod{\tHHk \ppk, \bggk} \leq \vnorm{\bggk}^{2} - 2 \rho \alphak \tnu \vnorm{\bggk}^{2} \leq \left(1 - 2 \rho \alpha \tnu\right) \vnorm{\bggk}^{2},
	\end{align*}
	which implies $ \eta = 2 \rho \alpha \tnu $. We finally note that, from $ \tnu \leq \nu $, $ \tgamma \leq \gamma $, we have
	\begin{align*}
		0 \leq \eta = \frac{4 \rho \tnu \tgamma^{2}}{L(\xx_{0})}\left( (1 - \rho) \tnu - \frac{\varepsilon}{\tgamma}\right) \leq \frac{4 \rho (1 - \rho) \nu^2 \gamma^{2}}{L(\xx_{0})} \leq 1,
	\end{align*}
	where the second inequality follows from \cite[Remark 5]{roosta2018newton}.
\end{proof}

From \cref{thm:exact}, it is clear that when $ (1 - \rho) \tnu \leq {\varepsilon}/{\tgamma} $, we have $ \eta = 0 $ and, as a result, a sufficient descent required for convergence cannot be established. This is, in fact, not a by-product of our analysis. Indeed, from \cref{lemma:pseudo_regularity_tH}, it follows that, as $ \varepsilon \downarrow 0 $, we have $ \tgamma \downarrow 0$, which implies that the least-norm solution can grow unboundedly. This is depicted in \cref{fig:pert}. As a consequence, the step-size from line-search may shrink to zero to counteract such unbounded growth. This phenomenon is also verified numerically in \Cref{sec:frac}. However, under certain conditions, we can indeed show that $ \eta > 0 $, which guarantees convergence.

\begin{corollary}[Convergence of \cref{alg:Newton_invex_sub} With Exact Updates Under General Perturbations]
	\label{coro:exact_nuis0}
	\vspace*{-2mm}
	Define 
	\begin{align*}
		a &\defeq C + 2 (1-\rho), \quad b \defeq (1-\rho)(2\nu - 1) - C \sqrt{1-\nu}, \\
		\delta(t) &\defeq \frac{\sqrt{\left(t^{2} + 4 (1-\rho)^{2}\right)^{2} - 16(1-\rho)^{4}} - \left(t^{2} - 4 (1-\rho)^{2} \right)}{8 (1-\rho)^{2}} < 1, \; \forall t \geq 1.
	\end{align*} 
	Under the assumptions of \cref{thm:exact}, if
	\begin{align*}
		\varepsilon < \frac{\gamma \left(2a + b + 1 - \sqrt{(2a + b + 1)^2 - 8 a b}\right)}{4a}, \quad \text{and} \quad \nu > \delta(C),
	\end{align*}
	we have $ \eta \in (0,1]$. 
	Here, $ \gamma, \nu, C $ and $ \rho $ are, respectively, as in \cref{eq:pseudo_regularity,eq:null_space,eq:epsilon_02,eq:armijo_gen}. 
\end{corollary}
The proof of \cref{coro:exact_nuis0} amounts to finding conditions for which $ (1 - \rho) \tnu > {\varepsilon}/{\tgamma} $, which we omit.
However, there is an interesting interplay between $ C $ and $ \nu $ in \cref{coro:exact_nuis0}. Indeed, since $ \delta(t) $ is increasing in $ t $, for perturbations with large $ C $ in \cref{eq:epsilon_02}, we can guarantee convergence as along as $ \nu  $ is close to one, i.e., the gradient contains a very small contribution from $ \Null(\HH) $. Furthermore, choosing smaller values for $ \rho $ eases some (though not all) of this restriction on $ \nu $. However, under inherently stable perturbations (cf.\ \cref{def:stable}), this restriction on $ \nu $ imposed by $ C $ can be entirely removed.

\begin{corollary}[Convergence of \cref{alg:Newton_invex_sub} With Exact Updates Under Inherent Stability]
	\label{coro:exact_nuis1}
	\vspace*{-2mm}
	Under the assumptions of \cref{thm:exact}, if the following conditions hold, we have $ \eta \in (0,1]$:
	\begin{itemize}
		\smallskip \item if the perturbation is acute and $ \varepsilon $ is small enough such that $ \varepsilon < (1 - \rho) (\gamma-\varepsilon) (\nu \gamma -2 \varepsilon)/\gamma $, then 
		\begin{align*}
			\eta \defeq \frac{4 \rho  (\nu \gamma- 2 \varepsilon) (\gamma-\varepsilon)^{2}}{\gamma^{2} L(\xx_{0})}\left( (1 - \rho) (\nu \gamma -2 \varepsilon) - \frac{\gamma \varepsilon}{\gamma-\varepsilon}\right) \in (0,1],
		\end{align*}
		\smallskip \item otherwise if $ \nu =1 $ and $ \varepsilon $ is small enough such that $ \varepsilon  < (1 - \rho) (\gamma-\varepsilon)(\gamma-4 \varepsilon)/\left((1+2C)\gamma - 2C\varepsilon \right)$, then	
		\begin{align*}
			\eta \defeq \frac{4 \rho (\gamma- 4 \varepsilon) (\gamma-\varepsilon)^{2}}{\left((1+2C)\gamma - 2C\varepsilon \right)^{2} L(\xx_{0})}\left( (1 - \rho) (\gamma-4 \varepsilon) - \frac{\left((1+2C)\gamma - 2C\varepsilon \right)\varepsilon}{\gamma-\varepsilon}\right) \in (0,1].
		\end{align*}
	\end{itemize}
	Here, $ L(\xx_{0}), \gamma, \nu, C $, and $ \rho $ are as in \cref{eq:armijo_gen,eq:epsilon_02,eq:pseudo_regularity,eq:lip_special,eq:null_space}, respectively.
\end{corollary}

\begin{remark}
	\label{rem:exact_actue}
	\vspace*{-2mm}
	For acute perturbations, we see from \cref{coro:exact_nuis1} that the convergence rate in the limit where $ \varepsilon \downarrow 0 $, matches that of unperturbed algorithm in \cite{roosta2018newton}.  For the case where $ \nu =1 $ but the perturbation is not acute, the limiting rate is worse than the unperturbed algorithm, and we believe that this is simply a byproduct of our analysis here.
\end{remark}

In the analysis of Newton's method, local convergence rate, i.e., convergence speed in the vicinity of a local solution, plays a critical role. There, by considering $ \alphak = 1 $, one can obtain fast problem-independent local convergence rates \cite{ssn2018}. Here, we aim to do the same for \cref{alg:Newton_invex_sub}. We note that the notion of ``$ \xx $ being local'' in the context of Newton-MR amounts to ``$\vnorm{\bgg(\xx)} $ being small enough'' \cite{roosta2018newton}. Obtaining a recursive behavior for $ \vnorm{\bgg} $ underpins our results here.
\begin{theorem}[\cref{alg:Newton_invex_sub} With $ \alphak = 1 $ and Exact Updates]
	\label{thm:exact_err}
	\vspace*{-2mm}
	Under the assumptions of \cref{thm:exact} with \cref{assmpt:lipschitz_special} replaced with \cref{eq:lip_usual_hessian}, for the iterates of \cref{alg:Newton_invex_sub} with $ \alphak = 1 $ and exact update, we have
	\begin{align*}
		\vnorm{ \bgg(\xx_{k+1}) } \leq c_1 \vnorm{\bgg(\xxk)}^{2}  + c_2 \vnorm{\bgg(\xxk)},
	\end{align*}
	where 
	\begin{align*}
		c_1 \defeq \frac{L_{\HH}}{2 \tgamma^{2}}, \quad \text{and} \quad c_2 \defeq \left( \frac{\varepsilon}{\tgamma} + \sqrt{1 - \tnu} \right) 
	\end{align*}
	and $ L_{\HH}, \tnu$ and $\tgamma$ are, respectively, as in \cref{eq:lip_usual_hessian,lemma:pseudo_regularity_tH,lemma:null_space_tH}.
\end{theorem}
\begin{proof}
	With $ \alphak = 1 $, we can apply \cref{lemma:null_space_tH} and the mean-value theorem \cite{ciarlet2013linear} for vector-valued functions to get
	\begin{align*}
	\vnorm{ \bgg(\xx_{k+1}) } &= \vnorm{ \bgg(\xxk + \ppk)} = \vnorm{\bgg(\xxk) + \int_{0}^{1} \HH \left(\xxk + t \ppk \right) \ppk  \df t}\\
	& =\vnorm{\left( \tUU \tUUt + \tUU_\perp \tUU^\intercal_\perp \right) \bgg(\xxk) + \int_{0}^{1} \HH \left(\xxk + t \ppk \right) \ppk \df t}\\
	& \leq \vnorm{\tHH(\xxk) \tHHd(\xxk) \bgg(\xxk) + \int_{0}^{1} \HH \left(\xxk + t \ppk \right) \ppk \df t} + \vnorm{\tUU_\perp \tUU^\intercal_\perp \bgg(\xxk)}\\
	& \leq \frac{1}{\tgamma} \vnorm{\bgg(\xxk)} \int_{0}^{1} \vnorm{\HH\left(\xxk + t \ppk \right) - \tHH(\xxk)} \df t + \sqrt{1 - \tnu} \vnorm{\bgg(\xxk)} \\
	& \leq \frac{1}{\tgamma} \vnorm{\bgg(\xxk)} \int_{0}^{1} \vnorm{\HH\left(\xxk + t \ppk \right) - \HH(\xxk)} \df t + \frac{\varepsilon}{\tgamma} \vnorm{\bgg(\xxk)} + \sqrt{1 - \tnu} \vnorm{\bgg(\xxk)} \\
	& \leq \frac{L_{\HH}}{2 \tgamma^{2}} \vnorm{\bgg(\xxk)}^{2}  + \left( \frac{\varepsilon}{\tgamma} + \sqrt{1 - \tnu} \right) \vnorm{\bgg(\xxk)}
	\end{align*}
\end{proof}

Here also as in \cref{thm:exact}, unless $ c_{2} < 1 $, one cannot establish local convergence using \cref{thm:exact_err}. We now show that for inherently stable perturbations, we can indeed guarantee this. For general perturbations, we note that a similar results as in \cref{coro:exact_nuis0} can also be obtained in the context of \cref{thm:exact_err}. However, for the sake of simplicity, we opt to omit them here.

\begin{corollary}[\cref{alg:Newton_invex_sub} With $ \alphak = 1 $ and Exact Updates Under Inherent Stability]
	\label{coro:exact_err}
	\vspace*{-2mm}
	Under the assumptions of \cref{thm:exact_err}, if the following conditions hold, we have $ c_2 < 1$:
	\begin{itemize}
		\smallskip \item if the perturbation is acute and $ \varepsilon $ is small enough such that 
		\begin{align*}
		\varepsilon < (\gamma-\varepsilon) \left( 1 - \sqrt{1 - \left(\nu-{2\varepsilon}/{\gamma}\right)}\right),
		\end{align*}
		then
		\begin{align*}
			c_{1} &= \frac{L_{\HH}}{2 (\gamma-\varepsilon)^{2}}, \quad c_{2} = \frac{\varepsilon}{(\gamma-\varepsilon)} + \sqrt{1 - \left(\nu-\frac{2\varepsilon}{\gamma}\right)} < 1,
		\end{align*}
		\smallskip \item otherwise if $ \nu =1 $ and $ \varepsilon $ is small enough such that 
		\begin{align*}
		\varepsilon < \left(\gamma-\varepsilon\right) \left( 1 -2 \sqrt{{\varepsilon}/{\gamma}} \right)/(1+2C),
		\end{align*}
		then	
		\begin{align*}
			c_{1} &= \frac{\left((1+2C)\gamma - 2C\varepsilon \right)^{2} L_{\HH}}{2 \gamma^{2}\left(\gamma-\varepsilon\right)^{2}}, \quad c_{2} = \frac{\left((1+2C)\gamma - 2C\varepsilon \right) \varepsilon}{\gamma\left(\gamma-\varepsilon\right)} + 2 \sqrt{\frac{\varepsilon}{\gamma}} < 1.
		\end{align*}
	\end{itemize}
	Here, $ L_{\HH}, \gamma , \nu $ and $ C $ are as in \cref{eq:lip_usual_hessian}, \cref{eq:epsilon_02,eq:pseudo_regularity_H,eq:null_space}, respectively.
\end{corollary}

\begin{remark}
	\label{rem:exact_local}
	\vspace*{-2mm}
	\Cref{coro:exact_err} shows that, under inherent stability and for small $ \varepsilon $, we obtain a problem-independent local linear convergence rate. For example, consider any $ \varepsilon $ small enough for which we get $ c_{2} < 1 $. Then for any $ c_{2} < c < 1  $, there exists a $ r > 0 $ for which if $ \vnorm{\bggk} \leq r $, we have $ \vnorm{\bgg_{k+1}} \leq c\vnorm{\bggk} $. 
	More generally, however, as $ \varepsilon \downarrow 0 $, since $ \tgamma \downarrow 0 $, we can get $ c_2 > 1 $ in \cref{thm:exact_err}, which can amount to divergence of the algorithm with constant step-size of $ \alphak = 1 $.
\end{remark}

\subsubsection{Inexact Updates}
\label{sec:inexact}

We now turn to convergence analysis of Newton-MR using inexact update \cref{eq:least_norm_solution_relaxed}. Clearly, the inexactness tolerance $ \theta $ has to be chosen with regard to \cref{lemma:null_space_tH}. Indeed, suppose the conditions of \cref{lemma:null_space_tH} are satisfied. With exact solution to \cref{eq:least_norm_solution}, we have
\begin{align*}
\dotprod{\bggk, \tHHk \ppk + \bggk} &= - \dotprod{\bggk, \tHHk \tHHdk \bggk} + \vnorm{\bggk}^{2} \leq (1-\tnu)\vnorm{\bggk}^{2},
\end{align*}
where $ \tnu $ is defined in \cref{lemma:null_space_tH}. This in turn implies that it is sufficient to choose $ \theta $ such that $ \theta \leq 1-\tnu$, giving rise to the following condition in inexactness tolerance. 
\begin{condition}[Inexactness Tolerance]
	\label{cond:theta}
	\vspace*{-2mm}
	The inexactness tolerance, $ \theta $, in \cref{eq:least_norm_solution_relaxed} is chosen such that $ 1 - \tnu \leq \theta < 1 $ where $ \tnu $ is defined in \cref{lemma:null_space_tH}.
\end{condition}

As advocated in \cite{roosta2018newton}, due to several desirable advantages, MINRES-QLP \cite{choi2011minres} is the method of choice for inexact variant of Newton-MR in which the search direction is computed from \cref{eq:least_norm_solution_relaxed}. Recall that, at the $ k\th $ iteration of \cref{alg:Newton_invex_sub}, the $ t\th $ iteration of MINRES-QLP can be represented as
\begin{align}
\label{eq:MINRES-QLP-SubProb}
\pp_{k}^{(t)} = \argmin \vnorm{\pp}^2, \quad \text{subject to} \quad \pp \in \Argmin_{\hat{\pp} \in \mathcal{K}_{t}} \vnorm{\tHHk \widehat{\pp} + \bggk}^2,
\end{align}
where $ \mathcal{K}_{t} = \mathcal{K}_{t}(\tHHk, \bggk) $ or $ \mathcal{K}_{t} = \mathcal{K}_{t}(\tHHk, \tHHk \bggk)  $.

Before delving deeper into the analysis of this section, we first give some simple properties of solutions to \cref{eq:least_norm_solution_relaxed} obtained from MINRES-QLP.

\begin{lemma}
	\vspace*{-2mm}
	\label{lemma:inexact_residual}
	For any solution to \cref{eq:least_norm_solution_relaxed} obtained from MINRES-QLP, we have 
	\begin{subequations}
		\label{eq:inexact_residual}
		\begin{align}
		\vnorm{\tHHk \ppk} &\leq \vnorm{\bggk}, \label{eq:inexact_residual_1}\\
		\vnorm{\tHHk \ppk + \bggk} &\leq \sqrt{\theta} \vnorm{\bggk}. \label{eq:inexact_residual_2}
		\end{align}
	\end{subequations}
\end{lemma}
\begin{proof} 
	It has been shown in \cite[Lemma 3.3 and Section 6.6]{choi2011minres} that for $ \pp_{k}^{(t)} $ as in \cref{eq:MINRES-QLP-SubProb}, $ \| \tHHk \pp_{k}^{(t)} \| $ is monotonically non-decreasing with $ t $. As a result, we obtain $ \| \tHHk \pp_{k} \| \leq \| \tHHk \tHHdk \bggk \| = \| \tUU \tUUt \bggk \| \leq \| \bggk \| $. Also, from \cref{eq:MINRES-QLP-SubProb}  and \cite[Lemma 3.3]{choi2011minres}, we always have $\dotprod{\pp_{k}, \tHHk \left( \tHHk \ppk + \bggk \right)} = 0$. Now, from \cref{eq:least_norm_solution_relaxed} we get \cref{eq:inexact_residual_2} as
	\begin{align*}
	    \theta \vnorm{\bggk}^{2} \geq \dotprod{\bggk, \tHHk \ppk + \bggk} = \dotprod{\bggk, \tHHk \ppk + \bggk} + \overbrace{\dotprod{\tHHk \ppk, \left( \tHHk \ppk + \bggk \right)}}^{= 0} = \vnorm{\tHHk \ppk + \bggk}^{2}.
	\end{align*}
\end{proof}

Here, as in \Cref{sec:exact}, establishing the convergence of \cref{alg:Newton_invex_sub} using \cref{eq:least_norm_solution_relaxed} hinges upon obtaining a bound similar to that in \cref{lemma:pseudo_regularity_tH}, but in terms of $ \ppk $ from \cref{eq:least_norm_solution_relaxed}.  A na\"{i}ve application of \cref{eq:inexact_residual_1,eq:epsilon_02} gives 
\begin{align*}
\vnorm{\ppk} \leq \frac{C}{\varepsilon} \vnorm{\tHHk \ppk} \leq \frac{C}{\varepsilon} \vnorm{\bggk},
\end{align*}
which implies the search direction can become unbounded as $ \varepsilon \downarrow 0 $.
Unfortunately, the norms of the iterates of MINRES-QLP are not necessarily monotonic; see \cite{calvetti2000curve,fong2012cg,choi2011minres}. As a result, although by \cref{lemma:pseudo_regularity_tH}, we have an upper bound on the final iterate, i.e., the exact solution \cref{eq:newton_mr_pseudo}, the intermediate iterates from \cref{eq:MINRES-QLP-SubProb} may have larger norms. Nonetheless, as part of the results of this section, we show that indeed all iterates of MINRES-QLP from \cref{eq:MINRES-QLP-SubProb} are bounded in the same way as in \cref{lemma:pseudo_regularity_tH}, which can be of independent interest; see \cref{lemma:solution_norm}.

To achieve this, we first show that \cref{eq:MINRES-QLP-SubProb} can be decoupled into two separate constrained least squares problems. We then show that the solution to each of these least squares problems is indeed bounded. 

\begin{lemma}
	\label{lemma:decoupling}
	\vspace*{-2mm}
	For any symmetric matrix $ \AA \in \reals^{d \times d}$ and $ \bb \in \reals^{d} $, consider the problem
	\begin{align}
	\label{eq:ols_full}
	\xxs =  \argmin \vnorm{\xx}^2 \quad \text{s.t.} \quad \xx \in \Argmin_{\widehat{\xx} \in \mathcal{K}_{t}} \vnorm{\AA \widehat{\xx} - \bb}^2,
	\end{align}
	where $ \mathcal{K}_{t} $ is any Krylov subspace.
	Let $ \PP_1$ and $\PP_2 $ be orthogonal projectors onto $\AA$-invariant subspaces. Further, assume that $ \PP_1 \PP_2 = \PP_2 \PP_1 = \bm{0} $ and $ \range(\AA) = \range (\PP_1)  \oplus \range(\PP_2 )$, where $ \oplus  $ denotes the direct sum. We have
	\begin{align}
	\label{eq:ols_decoupled}
	\xxs &=  \argmin_{ \xx_{1} \in \PP_{1} \cdot \mathcal{K}_{t}} \vnorm{\AA \xx_{1} - \PP_{1}\bb}^2 + \argmin_{\xx_{2} \in \PP_{2} \cdot \mathcal{K}_{t}} \vnorm{\AA \xx_{2} - \PP_{2} \bb}^2.
	\end{align}
	
\end{lemma}
\begin{proof}
	First note that since $ \AA $ is symmetric and $ \PP_{i},  i=1,2  $, are the orthogonal projectors onto invariant subspaces of $ \AA $, we have $ \PP_{i} \AA = \AA \PP_{i},  i=1,2  $. 
	Let $ \PP = \PP_{1} + \PP_{2} $. By Pythagoras theorem we have
	\begin{align*}
	\vnorm{\AA \xx - \bb}^2 = \vnorm{\PP \AA \xx - \PP \bb}^2 + \vnorm{\left(\eye - \PP\right) \bb}^2.
	\end{align*}
	Noting that $ \xxs \in \range(\AA) $ (see \cite[Lemma 7]{roosta2018newton}), we can rewrite \cref{eq:ols_full} as
	\begin{align*}
	\xxs &=  \argmin_{\xx \in \PP \cdot \mathcal{K}_{t}} \vnorm{\AA \PP \xx - \PP \bb}^2.
	\end{align*}
	
	Defining $ \xx_{i} = \PP_{i} \xx, i = 1,2 $, for any $ \xx \in \range(\AA) $, we can write $ \xx = \xx_{1} + \xx_{2}$. Noting that $ \PP_{1} $ and $ \PP_{2} $ are orthogonal projections onto orthogonal subspaces, we have 
	\begin{align*}
	\xxs &=  \argmin_{ \substack{\xx_{1} \in \PP_{1} \cdot \mathcal{K}_{t} \\ \xx_{2} \in \PP_{2} \cdot \mathcal{K}_{t}} } \vnorm{\AA \left( \PP_{1}\xx_{1} + \PP_{2}\xx_{2} \right) - \left(\PP_{1} + \PP_{2}\right)\bb}^2 \\
	&=  \argmin_{ \substack{\xx_{1} \in \PP_{1} \cdot \mathcal{K}_{t} \\ \xx_{2} \in \PP_{2} \cdot \mathcal{K}_{t}} } \vnorm{\PP_{1} \AA \xx_{1} + \PP_{2} \AA\xx_{2} - \left(\PP_{1} + \PP_{2}\right)\bb}^2 \\
	&=  \argmin_{ \substack{\xx_{1} \in \PP_{1} \cdot \mathcal{K}_{t} \\ \xx_{2} \in \PP_{2} \cdot \mathcal{K}_{t}} } \left( \vnorm{\PP_{1} \AA \xx_{1} - \PP_{1}\bb}^2 + \vnorm{\PP_{2} \AA \xx_{2} - \PP_{2} \bb}^2 \right) \\
	&=  \argmin_{ \xx_{1} \in \PP_{1} \cdot \mathcal{K}_{t}} \vnorm{\AA \xx_{1} - \PP_{1}\bb}^2 + \argmin_{\xx_{2} \in \PP_{2} \cdot \mathcal{K}_{t}} \vnorm{\AA \xx_{2} - \PP_{2} \bb}^2.
	\end{align*}
\end{proof}

The following lemma gives a bound on the solution of each of decoupled terms in \cref{eq:ols_decoupled}.
\begin{lemma}
	\label{lemma:lanczos}
	\vspace*{-2mm}
	For any symmetric matrix $ \AA \in \reals^{d \times d}$ and $ \bb \in \reals^{d} $, consider the problem
	\begin{align}	
	\label{eq:x_OLS_variant_MR}
	\xxs \defeq \argmin_{\xx \in \PP \cdot \mathcal{K}_{t}} \vnorm{ \AA \xx - \PP\bb},
	\end{align} 
	where $ \PP $ is the orthogonal projector onto a $\AA$-invariant subspace and $ \mathcal{K}_{t} $ is $ \mathcal{K}_{t}(\AA,\bb) $ or $ \mathcal{K}_{t}(\AA,\AA \bb) $, for $ t \in \left\{1,2,\ldots,\rank(\AA\PP)\right\} $. We have 
	\begin{align*}
	\vnorm{\xxs} \leq  \vnorm{\PP \bb}  \vnorm{\left[ \AA \PP \right]^{\dagger}}.
	\end{align*}
\end{lemma}

\begin{proof}
	Clearly, we can replace \cref{eq:x_OLS_variant_MR} with an equivalent formulation as
	\begin{align*}
	\xxs = \argmin_{\xx \in \PP \cdot \mathcal{K}_{t}} \vnorm{ \AA \PP \xx - \PP\bb}. 
	\end{align*}
	We prove the result for when $ \mathcal{K}_{t} = \mathcal{K}_{t}(\AA,\bb) $ as the case of $ \mathcal{K}_{t} = \mathcal{K}_{t}(\AA,\AA \bb) $ is proven similarly. As before, since $ \AA $ is symmetric and $ \PP $ is the orthogonal projector onto an invariant subspaces of $ \AA $, we have $ \PP \AA = \AA \PP$, hence $ \AA \PP $ is also symmetric. Consider applying Lanczos process to obtain the decomposition
	\begin{align*}
	\AA \PP \QQt = \QQ_{t+1} \TTt,
	\end{align*}
	where $ \TTt \in \reals^{(t+1) \times t} $ and $ \QQt = \left[ \qq_1, \qq_2, \ldots, \qq_{t} \right] $ is an orthonormal basis for the Krylov subspace $ \PP \cdot \mathcal{K}_{t} =  \mathcal{K}_{t}(\AA \PP,\PP\bb) $ and $ \QQ_{t+1} = \left[ \QQt \mid \qq_{t+1} \right] $ with $ \QQ_{t}^{\intercal} \qq_{t+1} = \bm{0} $. Recall that one can find $ \xxs = \QQt \yys $ where 
	\begin{align*}
	\yys \defeq \argmin_{\yy \in \reals^t} \vnorm{\TTt \yy - \QQ_{t+1}^{\intercal} \PP \bb}.
	\end{align*}
	It follows that
	\begin{align*}
	\vnorm{\xxs} = \vnorm{\QQt \yys} = \vnorm{\yys} = \vnorm{\TT_{t}^{\dagger} \QQ_{t+1}^{\intercal} \PP \bb}.
	\end{align*}
	Also, we have 
	\begin{align*}
	\vnorm{\TT_{t}^{\dagger}} = \vnorm{ \left[ \QQ_{t+1} \TTt \QQ_{t}^{\intercal} \right]^{\dagger}} \leq \vnorm{\left[ \QQ_{t+2} \TT_{t+1} \QQ_{t+1}^{\intercal}  \right]^{\dagger}} \leq \ldots \leq \vnorm{\left[ \AA \PP \right]^{\dagger}},
	\end{align*}
	where the first equality is obtained by noting that
	\begin{align*}
	\left[ \QQ_{t+1} \TTt \QQ_{t}^{\intercal} \right]^{\dagger} =  \left[\TTt \QQ_{t}^{\intercal}\right]^{\dagger} \left[\QQ_{t+1}\right]^{\intercal} = \QQt \TT_{t}^{\dagger} \QQ_{t+1}^{\intercal},
	\end{align*}
	and the series of inequalities follow from \cite[Proposition 2.1]{calvetti2002gmres}.
	So, we finally get
	\begin{align*}
	\vnorm{\xxs} = \vnorm{\TT_{t}^{\dagger} \QQ_{t+1}^{\intercal} \PP \bb} \leq \vnorm{\PP \bb}  \vnorm{\TT_{t}^{\dagger}} \leq \vnorm{\PP \bb}  \vnorm{\left[ \AA \PP \right]^{\dagger}}.
	\end{align*}
\end{proof}

We are now ready to prove a result similar to \cref{lemma:pseudo_regularity_tH} for the case of inexact updates.
\begin{lemma}
	\label{lemma:solution_norm}
	\vspace*{-2mm}
	Under Assumptions of \cref{lemma:pseudo_regularity_tH}, for the iterates of MINRES-QLP in \cref{eq:MINRES-QLP-SubProb}, we have
	\begin{align*}
	\vnorm{\pp_{k}^{(t)}} \leq \frac{1}{\tgamma} \vnorm{\bggk}, \quad t = 1,2,\ldots,\rank(\tHHk),
	\end{align*}
	where $ \tgamma $ is as in \cref{lemma:pseudo_regularity_tH}.
\end{lemma}

\begin{proof}
	For simplicity, we drop the dependence on $ k $ and $ t $. 
	Let $ \tPP_1, \tPP_2 $ and $ \tPP_{\perp} $ denote the projectors $ \tUU_1 \tUU_1^{\intercal}, \tUU_2 \tUU_2^{\intercal} $ and $ \tUU_{\perp} \tUU_{\perp}^{\intercal} $, respectively, where $ \tUU_1, \tUU_2$ and $  \tUU_{\perp} $ are defined in \cref{eq:svd}. Also, let $ \tPP = \tPP_1 + \tPP_2 $.
	Using \cref{lemma:decoupling}, we can write \cref{eq:MINRES-QLP-SubProb} as $ \pp = \pp_{1} + \pp_{2} $, where
	\begin{align*}
	\pp_{1} =  \argmin_{ \pp \in \tPP_{1} \cdot \mathcal{K}_{t}} \vnorm{\tHH \pp + \tPP_{1}\bgg}, \quad \pp_{2} =  \argmin_{ \pp \in \tPP_{2} \cdot \mathcal{K}_{t}} \vnorm{\tHH \pp + \tPP_{2}\bgg}.
	\end{align*}

	From \cref{lemma:lanczos,eq:sigma_tU1}, it follows that 
	\begin{align*}
	\vnorm{\pp_{1}} \leq \vnorm{\tPP_1 \bgg } \vnorm{\left[ \tHH \tPP_1 \right]^{\dagger}} = \vnorm{\tUU_1 \tUU_1^{\intercal} \bgg } \vnorm{\left[ \tUU_1 \tUU_1^{\intercal} \tHH \right]^{\dagger}} \leq \frac{1}{\gamma - \varepsilon} \vnorm{\bgg}.
	\end{align*}		
	Similarly, by \cref{eq:null_space_grad_2}, \cref{eq:sigma_tU2}, \cref{thm:proj}, and \cref{lemma:lanczos}, we have
	\begin{align*}
	\vnorm{\pp_{2}} &\leq \vnorm{\tPP_2 \bgg } \vnorm{\left[ \tHH \tPP_2 \right]^{\dagger}} \leq \vnorm{\left( \tUU_2 \tUU_2^{\intercal} + \tUU_{\perp} \tUU_{\perp}^{\intercal} \right)\bgg} \vnorm{\left[ \tUU_2 \tUU_2^{\intercal} \tHH \right]^{\dagger}} \\
	&\leq \frac{C}{\varepsilon} \vnorm{\left( \tUU_2 \tUU_2^{\intercal} + \tUU_{\perp} \tUU_{\perp}^{\intercal} - \UU_{\perp} \UU_{\perp}^{\intercal} + \UU_{\perp} \UU_{\perp}^{\intercal} \right)\bgg} \\
	&\leq \frac{C}{\varepsilon}\vnorm{\left( \UU \UU^{\intercal} - \tUU_{1} \tUU_{1}^{\intercal} + \UU_{\perp} \UU_{\perp}^{\intercal} \right) \bgg} \leq \frac{C}{\varepsilon}\left(\frac{2\varepsilon}{\gamma} + \sqrt{1 - \nu} \right) \vnorm{\bgg},
	\end{align*}	
	which gives us
	\begin{align*}
	\vnorm{\pp_{2}} \leq C \left(\frac{2}{\gamma} + \frac{\sqrt{1 - \nu}}{\varepsilon} \right) \vnorm{\bgg}.
	\end{align*}
	Finally, we obtain
	\begin{align*}
	\vnorm{\pp}^{2} = \vnorm{\pp_{1} + \pp_{2}}^{2} = \vnorm{\pp_{1}}^{2} + \vnorm{\pp_{2}}^{2} \leq \left( \left( \frac{1}{\gamma - \varepsilon} \right)^{2} + \left( C \left(\frac{2}{\gamma} + \frac{\sqrt{1 - \nu}}{\varepsilon} \right) \right)^{2} \right)\vnorm{\bgg}^{2}.
	\end{align*}
	The result follows from the inequality $ \sqrt{a^{2} + b^{2}} \leq a + b, \; \forall a,b \geq 0 $.
\end{proof}

The inexactness condition in \cref{eq:least_norm_solution_relaxed} involves two criteria for an approximate solution $ \ppk $, namely feasibility of $ \ppk $ in \cref{eq:least_norm_solution_relaxed} and that $ \ppk \in \range(\tHHk) $. When $ \bggk \in \range(\tHHk) $, the latter is enforced naturally as a result of MINRES-QLP's underlying Krylov subspace. However, in cases where $ \bggk \notin \range(\tHHk) $, one could simply modify the Krylov subspace as described in \cite{roosta2018newton}. To allow for unification of the results of this section, we define \emph{range-invariant} Krylov subspace, which encapsulate these variants. 

\begin{definition}[Range-invariant Krylov Subspace]
	\label{def:krylov}
	\vspace*{-2mm}
	For any symmetric matrix $ \AA $, the range-invariant Krylov subspace is defined as follows. 
	\begin{enumerate}[label = \textbf{(\roman*)}]
		\item If $\bb \in \range(\AA)$, we can consider the usual $ \mathcal{K}_{t}(\AA, \bb) $ , e.g., MINRES \cite{paige1975solution}.
		\item Otherwise, we can always employ $ \mathcal{K}_{t}(\AA, \AA \bb) $, e.g., MR-II \cite{hanke2017conjugate}.
	\end{enumerate}
	Here, $t = 1, \ldots, \rank(\AA)$.
\end{definition}
In the subsequent discussion, we always assume that MINRES-QLP used within \cref{alg:Newton_invex_sub} generates iterates from an appropriate range-invariant Krylov subspace, $ \mathcal{K}_{t}(\tHHk, \bggk) $ or $ \mathcal{K}_{t}(\tHHk, \tHHk \bggk) $ (cf.\ \cref{eq:MINRES-QLP-SubProb}). 

Now, similar with the proofs for exact updates \cref{thm:exact}, we can obtain the following results for \cref{alg:Newton_invex_sub} with inexact updates satisfying \cref{eq:least_norm_solution_relaxed}.

\begin{theorem}[\cref{alg:Newton_invex_sub} With Inexact Updates]
	\label{thm:inexact} 
	\vspace*{-2mm}
	Under \cref{assmpt:perturb,assmpt:moral_diff,assmpt:lipschitz_special,assmpt:pseudo_regularity,assmpt:null_space} and \cref{cond:stable,cond:theta}, for the iterates of \cref{alg:Newton_invex_sub} with inexact updates, we have
	\begin{align*}
	\vnorm{\bgg_{k+1}}^2 \leq \left(1 -  \eta \right) \vnorm{\bggk}^2
	\end{align*} 
	where
	\begin{align*}
		\eta \defeq \max \left\{0,\frac{4 \rho  \tgamma^2 (1 - \theta)}{L(\xx_{0})}\left( (1 - \rho) (1 - \theta) -  \frac{\varepsilon}{\tgamma} \right) \right\} \in [0,1],
	\end{align*}
	and $ \rho, L(\xx_{0}), \tnu, \tgamma$ and $ \theta $ are, respectively, as in \cref{eq:armijo_gen}, \cref{assmpt:lipschitz_special}, \cref{lemma:pseudo_regularity_tH,lemma:null_space_tH}, and \cref{cond:theta}.
\end{theorem}
\begin{proof}  
	\label{pf:inexact}
	Similar with the proof for \cref{thm:exact}, by \cref{eq:perturb,eq:least_norm_solution_relaxed,lemma:solution_norm}, we have
	\begin{align*}
	\rho \dotprod{ \bggk, \tHHk \ppk} - \dotprod{\bggk, \HHk \ppk} &= -(1 - \rho)\dotprod{ \bggk, \tHHk \ppk} + \dotprod{\bggk, \left(\tHHk - \HHk\right) \ppk} \nonumber\\
	&\geq (1 - \rho)(1 - \theta)\vnorm{\bggk}^2 - \frac{\varepsilon}{\tgamma} \vnorm{\bggk}^2.
	\end{align*}
	If $ \varepsilon $ satisfies the inequality $ \varepsilon \leq (1 - \rho) (1-\theta) \tgamma $, the lower bound on the step-size returned by line-search \cref{eq:armijo_gen} is $ \alphak \geq \alpha $ where
	\begin{align*}
		\alpha \defeq \frac{2 \tgamma^2 }{L(\xx_{0})}\left( (1 - \rho) (1 - \theta) -  \frac{\varepsilon}{\tgamma} \right).
	\end{align*}	
	Otherwise, the lower bound is the trivial $ \alpha = 0 $.
	Now, from \cref{eq:armijo_gen} with the lower bound $ \alpha $, we get
	\begin{align*}
		\vnorm{\bggkk}^{2} &\leq \vnorm{\bggk}^{2} + 2 \rho \alphak \dotprod{\tHHk \ppk, \bggk} \leq \vnorm{\bggk}^{2} - 2 \rho \alphak (1 - \theta) \vnorm{\bggk}^{2} \leq \left(1 - 2 \rho \alpha (1 - \theta)\right) \vnorm{\bggk}^{2},
	\end{align*}
	which implies $ \eta = 2 \rho \alpha (1 - \theta) $. We finally note since $ (1 - \theta) \leq \tnu $, similarly to the line of reasoning at the end of the proof of \cref{thm:exact}, we can deduce that $ \eta \in [0,1] $.
\end{proof}

\begin{remark}
	\vspace*{-2mm}
	Note that when $ \theta = 1-\tnu $, \cref{thm:exact,thm:inexact} exactly coincide. 
\end{remark}

From \cref{thm:inexact}, similar to \cref{thm:exact}, one cannot establish sufficient descent required for convergence unless $ \eta > 0 $. However, similar results as those of \cref{coro:exact_nuis0,coro:exact_nuis1} corresponding to \cref{thm:inexact} can also be easily established here. We omit those results for the sake of brevity. Nonetheless, the interesting interplay between $ \varepsilon $ and $ \theta $ that arises as a result of \cref{thm:inexact} should be highlighted. For example, suppose $ \nu = 1 $. By inspecting the condition $ \eta > 0 $, i.e., $  \varepsilon < \tgamma (1 - \rho) (1 - \theta) $, one can see that the smaller values of $ \epsilon $, i.e., more accurate estimations of $ \HH $, allow for larger values of $ \theta $, which, in turn, amount to cruder approximations to the exact least-norm solution. In other words, Hessian approximation and sub-problem accuracy in the form of least-squares residual (cf.\ \cref{eq:inexact_residual_2}) are inversely related.
 
As in \cref{thm:exact_err}, we can obtain a recursive behavior of $ \vnorm{\bggk} $ for the case where $ \alphak = 1 $, which can then be used to deduce a local problem-independent convergence rate similar to that described in \cref{rem:exact_local}.

\begin{theorem}[\cref{alg:Newton_invex_sub} With $ \alphak = 1 $ and Inexact Updates]
	\label{thm:inexact_err} 
	\vspace*{-2mm}
	Under the assumptions of \cref{thm:inexact} with \cref{assmpt:lipschitz_special} replaced with \cref{eq:lip_usual_hessian}, 
	for the iterates of \cref{alg:Newton_invex_sub} with $ \alphak = 1 $ and inexact updates \cref{eq:least_norm_solution_relaxed}, we have
	\begin{align*}
	\vnorm{\bgg(\xx_{k+1})} \leq \frac{L_{\HH}}{2 \tgamma^2} \vnorm{\bgg(\xxk)}^2 + \left( \frac{\varepsilon}{\tgamma} + \sqrt{\theta} \right) \vnorm{\bgg(\xxk)}. 
	\end{align*}
	where $ L_{\HH}, \tgamma$ and $\theta$ are, respectively, as in \cref{eq:lip_usual_hessian}, \cref{lemma:pseudo_regularity_tH}, and \cref{cond:theta}.
\end{theorem}
\begin{proof}
	\label{pf:inexact_err}
	Similarly to the proof of \cref{thm:exact_err}, using \cref{eq:inexact_residual_2}, we have
	\begin{align*}
	\vnorm{ \bgg(\xx_{k+1}) } &= \vnorm{ \bgg(\xxk + \ppk)} = \vnorm{\bgg(\xxk) + \int_{0}^{1} \HH(\xxk + t \ppk) \ppk \df t}\\
	& =\vnorm{\bgg(\xxk) + \tHH(\xxk) \ppk - \tHH(\xxk) \ppk + \int_{0}^{1} \HH(\xxk + t \ppk) \ppk \df t}\\
	& \leq \vnorm{- \tHH(\xxk) \ppk + \int_{0}^{1} \HH(\xxk + t \ppk) \ppk \df t} + \vnorm{\bgg(\xxk) + \tHH(\xxk) \ppk}\\
	& \leq \vnorm{\ppk} \int_{0}^{1} \vnorm{\HH(\xxk + t \ppk) - \HH(\xxk)} \df t + \varepsilon \vnorm{\ppk} + \sqrt{\theta} \vnorm{\bgg(\xxk)} \\
	& \leq \frac{L_{\HH}}{2 \tgamma^2} \vnorm{\bgg(\xxk)}^2 + \left( \frac{\varepsilon}{\tgamma} + \sqrt{\theta} \right) \vnorm{\bgg(\xxk)}.
	\end{align*}  
\end{proof}

\begin{remark}
	\vspace*{-2mm}
	Here also, when $ \theta = 1-\tnu $, \cref{thm:exact_err,thm:inexact_err} exactly coincide. 
\end{remark}

Just as in \Cref{sec:exact}, with some simple algebraic manipulations, we can easily derive sufficient conditions on $ \varepsilon $ such that $ {\varepsilon}/{\tgamma} + \sqrt{\theta} < 1 $ in \cref{thm:inexact_err} (for example a similar result to \cref{coro:exact_err} in the special case of inherently-stable perturbations). We omit those results for the sake of brevity.

\subsection{Comparison with Sub-sampled Newton Method}
\label{sec:comparison}
As mentioned in \Cref{sec:Intro}, even though Newton-MR can be readily applied, beyond strongly-convex settings, to the more general class of invex problems, its iterations bear a strong resemblance to those of the classical Newton's algorithm. Hence, it is illuminating to have a renewed look at the results of this paper in light of the existing results on Newton's method (and it Newton-CG variant) in the contexts of inexact Hessian and strong convexity. To do this, we consider the setting of finite-sum minimization problem \cref{eq:obj_sum} and compare the present results with those of \cite{ssn2018}. For concreteness, we consider \cite[Theorem 13]{ssn2018}, which gives the global convergence of sub-sampled Newton method with problem-independent local convergence rate. To create a level playing field here, we make the same assumptions as those of \cite[Theorem 13]{ssn2018}, namely each $f_{i}$ is twice-differentiable, smooth and convex, i.e., $ 0 \leq \lambda_{\min}(\HH_{i}(\xx)) \leq \lambda_{\max}(\HH_{i}(\xx)) \leq L_{i}, \; \forall \xx \in \reals^{d}$ where $ \HH_{i}(\xx) \defeq \nabla^{2} f_{i}(\xx)$, $ f $ is $ \gamma$-strongly convex with Lipschitz gradient and Hessian, i.e., it satisfies \cref{eq:lip_usual}. Note that by \cite[Lemma 1]{roosta2018newton}, we have that 
$ L(\xxo) = L_{\bgg}^{2} + L_{\HH} \vnorm{\nabla f(\xxo)} $ where $ L(\xxo) $ is as in \cref{eq:lip_special}.
As in \cite[Section 1.5]{ssn2018}, we define $ \kappa \defeq L_{\bgg}/\gamma $ and $ \kappa_{\max} \defeq \max_{i} L_{i}/\gamma $ as the problem and sub-sampling condition numbers, respectively. Note that $ \kappa_{\max} \geq \kappa $. We also define $ \kappa_{0} \defeq \sqrt{L(\xxo)} / \gamma $. 
Now, by the assumption on $ \varepsilon $ in \cite[Theorem 13]{ssn2018}, we have $ \varepsilon \in \bigO{1/\sqrt{\kappa_{\max}}} $, which in light of \cite[Lemma 2]{ssn2018} implies that a sample size of $ |\mathcal{S}| \in \bigOt{\kappa_{\max}^{2}}$ guarantees \cref{eq:spectrum}. 
%This, in turn, gives $ \tgamma \geq  c \gamma $ for some constant $ c < 1 $ (alternatively, we can say $ \tgamma \in \Omega(\gamma) $ where the constant is smaller than one). 

Specialized to strongly-convex problems, from \cref{lemma:rank}, we have that $ \rank(\HH) = \rank(\tHH) = d $, which implies $ \tnu = 1 $ and also gives $ \tgamma = \gamma - \varepsilon $ in \cref{lemma:pseudo_regularity_tH}. Now, from \cref{thm:exact}, we get 
\begin{align*}
\eta = \frac{4 \rho (\gamma - \varepsilon )^{2}}{L_{\bgg}^{2} + L_{\HH} \vnorm{\nabla f(\xxo)} }\left( (1 - \rho) - \frac{\varepsilon}{\gamma - \varepsilon }\right).
\end{align*}
Choosing $ \varepsilon \leq (1-\rho) \gamma /(2-\rho) $ implies $\eta \geq {8 \rho (1-\rho)}/{(3-\rho)^2 \kappa^{2}_{0}}$. 
With this $ \varepsilon $, \cite[Lemma 16]{xuNonconvexTheoretical2017} implies that a sample size of $ |\mathcal{S}| \in \bigOt{\kappa_{\max}^{2}}$ is required to form $ \tHH $, which is of the same order as that for sub-sampled Newton's method above. Similarly, with this $ \varepsilon $, the local convergence result of \cref{thm:exact_err} can be stated with $ c_{2} \leq (1-\rho)/2 $. For inexact update in \Cref{sec:inexact}, we can also derive similar results in the present context. Here, we emphasize that since $ \tnu = 1 $, the inexactness tolerance can be set to any value $ \theta \in [0,1) $ in \cref{cond:theta}. This is in sharp contrast to sub-sampled Newton-CG in which the inexactness tolerance is of the order $\theta \in \bigO{{1}/{\sqrt{\kappa_{\max}}}}$, which is rather restrictive; see \cite{ssn2018} for further details on inexactness tolerance for sub-sampled Newton-CG.

We put all this together in \cref{table:comparison}. We also convert the convergence results of this paper in $ \vnorm{\bgg} $ to those of \cite{ssn2018} which are in terms of $ f - f^{\star} $ and $ \vnorm{\xx - \xxs} $ for global and local convergence regimes, respectively. For this we use the well-known facts about strong convexity \cite{nesterov2004introductory} that
\begin{align*}
\vnorm{\bgg(\xx)} &\geq \gamma \vnorm{\xx - \xxs}, \quad \text{and} \quad \vnorm{\bgg(\xx)}^{2} \geq 2 \gamma \left( f(\xx) - f(\xxs) \right).
\end{align*}
In evaluating the complexities, we have assumed that the cost of one Hessian-vector product is of the same order as evaluating a gradient, e.g,~\cite{xu2016sub,pearlmutter1994fast,ssn2018,griewank1993some}.
From \cref{table:comparison}, the overall worst-case running-time of an algorithm to achieve the prescribed sub-optimality is estimated as \big( $ n d $ + Column \#2 $ \times $ Column \#3 \big) $ \times $ \big(Column \#4 or Column \#5\big), the first term $ n d $ is the cost of evaluating the full gradient.

\begin{table}[!htbp]
	\caption{Complexity comparison of variants of Newton's method and Newton-MR methods for \eqref{eq:obj_sum}. The notation $ \tilde{\mathcal{O}} $ implies hidden logarithmic factors, e.g., $  \ln(\kappa), \ln(\kappa_{\max}), \ln(d)$. Constants $ \gamma, \kappa, \kappa_{\max}, \kappa_{0}$ are defined in \Cref{sec:comparison}. Fourth column gathers iteration complexity to achieve sub-optimality $ f(\xxk) - f(\xx^{\star}) \leq \varsigma $ for some $ \varsigma \leq 1 $. Fifth column reflect the corresponding complexity to achieve $ \|\xxk - \xxs\| \leq \varsigma $ for some $ \varsigma \leq 1 $, assuming $ \xxo $ is close enough to $ \xxs$. \label{table:comparison}}
	\centering
	\vspace*{2mm}
	\hspace*{-3mm}
	\scalebox{0.70}{
		\begin{tabular}{|c | c | c | c | c | c |} 
			\hline
			\multicolumn{1}{|c|}{\centering Method} & \multicolumn{1}{m{2.5cm}|}{\centering Evaluating Hessian-Vector Product, $ \HH \vv $} & \multicolumn{1}{m{2.3cm}|}{\centering \# of Iterations of MINRES/CG} & \multicolumn{1}{m{2.9cm}|}{\centering Global Iteration Complexity} & \multicolumn{1}{m{2.9cm}|}{\centering Local Iteration Complexity} & Reference \\ [0.5ex] 
			\hline  & & & & & \\[-1ex]
			Newton & $ \mathcal{O}(nd) $ & $\mathcal{O}(d) $  & $ \mathcal{O}(\kappa^{2} \ln\frac{1}{\varsigma}) $ & $ \mathcal{O}(\ln \ln\frac{1}{\varsigma}) $ & Folklore\\[1ex]
			\hdashline & & & & &\\[-1ex]
			Newton-CG & $ \mathcal{O}(nd) $ & $ \tilde{\mathcal{O}}(\sqrt{\kappa}) $ & $ \mathcal{O}(\kappa^{2}\ln\frac{1}{\varsigma})$ & $ \mathcal{O}(\ln \frac{1}{\varsigma}) $ & Folklore \\[1ex]
			\hdashline & & & & &\\[-1ex]
			Sub-sampled Newton  & $ \tilde{\mathcal{O}}(d \kappa_{\max}^{2}) $ & $ \mathcal{O}(d) $ & $ \mathcal{O}(\kappa \kappa_{\max}\ln\frac{1}{\varsigma}) $ & $ \mathcal{O}(\ln \frac{1}{\varsigma}) $ & \cite[Theorem 13]{ssn2018}\\[1ex]
			\hdashline & & & & &\\[-1ex]
			Sub-sampled Newton-CG  & $ \tilde{\mathcal{O}}(d \kappa_{\max}^{2}) $ & $ \tilde{\mathcal{O}}(\sqrt{\kappa_{\max}}) $ & $ \mathcal{O}(\kappa \kappa_{\max}\ln\frac{1}{\varsigma}) $ & $ \mathcal{O}(\ln \frac{1}{\varsigma}) $ & \cite[Theorem 13]{ssn2018}\\[1ex]
			\hdashline & & & & &\\[-1ex]
			Newton-MR (Exact Update)  & $ \tilde{\mathcal{O}}(nd) $ & $ \mathcal{O}(d) $ & $ \mathcal{O}(\kappa^{2}_{0}\ln\frac{1}{\varsigma}) $ & $ \mathcal{O}(\ln \ln \frac{1}{\varsigma}) $ & \cite[Corollary 1, Theorem 2]{roosta2018newton}\\[1ex]
			\hdashline & & & & &\\[-1ex]
			Newton-MR (Inexact Update)  & $ \tilde{\mathcal{O}}(n d) $ & $ \tilde{\mathcal{O}}(\sqrt{\kappa}) $ & $ \mathcal{O}(\kappa^{2}_{0}\ln\frac{1}{\varsigma}) $ & $ \mathcal{O}(\ln \frac{1}{\varsigma}) $ & \cite[Corollary 2, Theorem 4]{roosta2018newton}\\[1ex]
			\hdashline & & & & &\\[-1ex]
			\cref{alg:Newton_invex_sub} with \cref{eq:least_norm_solution}  & $ \tilde{\mathcal{O}}(d \kappa_{\max}^{2}) $ & $ \mathcal{O}(d) $ & $ \mathcal{O}(\kappa^{2}_{0}\ln\frac{1}{\varsigma}) $ & $ \mathcal{O}(\ln \frac{1}{\varsigma}) $ & \cref{thm:exact,thm:exact_err}\\[1ex]
			\hdashline & & & & &\\[-1ex]
			\cref{alg:Newton_invex_sub} with \cref{eq:least_norm_solution_relaxed}  & $ \tilde{\mathcal{O}}(d \kappa_{\max}^{2}) $ & $ \tilde{\mathcal{O}}(\sqrt{\kappa_{\max}}) $ & $ \mathcal{O}(\kappa^{2}_{0}\ln\frac{1}{\varsigma}) $ & $ \mathcal{O}(\ln \frac{1}{\varsigma}) $ & \cref{thm:inexact,thm:inexact_err}\\[1ex]
			\hline
	\end{tabular}}
\end{table}

\cref{table:comparison} gives complexities involved in various algorithms for achieving sub-optimality in objective value, i.e.,  $ f(\xxk) - f(\xx^{\star}) \leq \varsigma $ for some $ \varsigma \leq 1 $ and the corresponding complexity to achieve $ \|\xxk - \xxs\| \leq \varsigma $ for some $ \varsigma \leq 1 $, assuming $ \xxo $ is in the vicinity of the solution $ \xxs$.  
We note that the complexities given in \cref{table:comparison} are , not only, for worst-case analysis, but also they are pessimistic. For example, from the worst-case complexity of the algorithms with Hessian sub-sampling, it appears that they are advantageous only in some marginal cases. However, this is unfortunately a side-effect of our analysis and not an inherent property of the sub-sampled algorithm. In this light, any conclusions from these tables should be made with great care.  

%Naturally, it is advisable to perform Hessian sub-sampling only when $ n\gg1 $; \cref{table:comparison}, very pessimistically, suggests that Hessian sub-sampling offer computational savings for Newton's method and Newton-MR when $n \geq \kappa^{2}_{\max}$ and $n \geq \kappa^{2}_{0}$, respectively. 

In the strongly-convex setting with the above smoothness assumptions, since $ \kappa_{0} \geq \kappa $, the global worst-case iteration complexity of Newton-MR is worse than that of Newton-CG (of course, Newton-MR applies to the larger class of invex objectives, which are also allowed to be less smooth than what is assumed to generate \cref{table:comparison}; see \cite{roosta2018newton} for a detailed discussion.) However, for sub-sampled variants of these algorithms, the comparison is not as straightforward. Indeed, the interplay between $ \xxo $, $ L_{\HH} $, and $ \max_{i} L_{i} $ determines the relationship between $ \kappa\kappa_{\max} $ and $ \kappa^{2}_{0} $. For examples, if $ \xxo $ is chosen such that $ \vnorm{\nabla f(\xxo)} \ll 1 $, then one expects to see $ \kappa^{2}_{0} \leq \kappa \kappa_{\max} $, which implies \cref{alg:Newton_invex_sub} should perform better than sub-sampled Newton methods in \cite{ssn2018}. Similarly, if $ L_{i}$'s are very non-uniform, then noting that $ L_{\bgg} \leq \sum_{i = 1}^{n} L_{i}/n $, we can also expect $ \kappa_{\max} \gg \kappa $, which could imply $ \kappa^{2}_{0} \leq \kappa \kappa_{\max} $. 

Although not reflected in \cref{table:comparison}, it has been shown that for linear systems involving positive definite systems, MINRES can achieve a given relative residual tolerance in far fewer iterations than CG; see \cite{fong2012cg} for a detailed discussion. This observation
%, coupled with the fact that $ \theta $ in \cref{cond:theta} can be set to any value $ \theta \in [0,1) $ as opposed to that required for sub-sampled Newton-CG method which is $\theta \in \bigO{{1}/{\sqrt{\kappa_{\max}}}}$,  
indicates that \cref{alg:Newton_invex_sub} with \cref{eq:least_norm_solution_relaxed} should typically converge faster than sub-sampled Newton-CG method of \cite{ssn2018}. This is indeed confirmed by the numerical experiments of \Cref{sec:experiments}. Finally, from \cref{table:comparison}, we can also see that in the absence of a good preconditioner, if $ \kappa_{\max} \geq d^{2} $, solving \cref{eq:least_norm_solution} exactly can be potentially more efficient than resorting to an inexact method.

\section{Numerical Experiments}
\label{sec:experiments}
In this section, we empirically verify the theoretical results of this paper and also evaluate the performance of Newton-MR as compared with several optimization methods. In particular, we first study the effect of unstable perturbations with $ \nu < 1 $ in \Cref{sec:frac} and show that, somewhat unintuitively, reducing the perturbations indeed results in worsening of the performance. In \Cref{sec:stable}, we then turn our attention to two class of problems where $ \nu = 1 $ and demonstrate that such inherent stability allows for the design of a highly efficient variant of Newton-MR in which the Hessian is approximated. The code for the experiments is available at \href{https://github.com/syangliu/Newton-MR}{https://github.com/syangliu/Newton-MR}.

\begin{figure}[!htbp]
	\centering
	% \vspace*{-3mm}
	\subfigure[$ f(\xxk) $ vs.\ Iterations]
	{\includegraphics[scale=0.35]{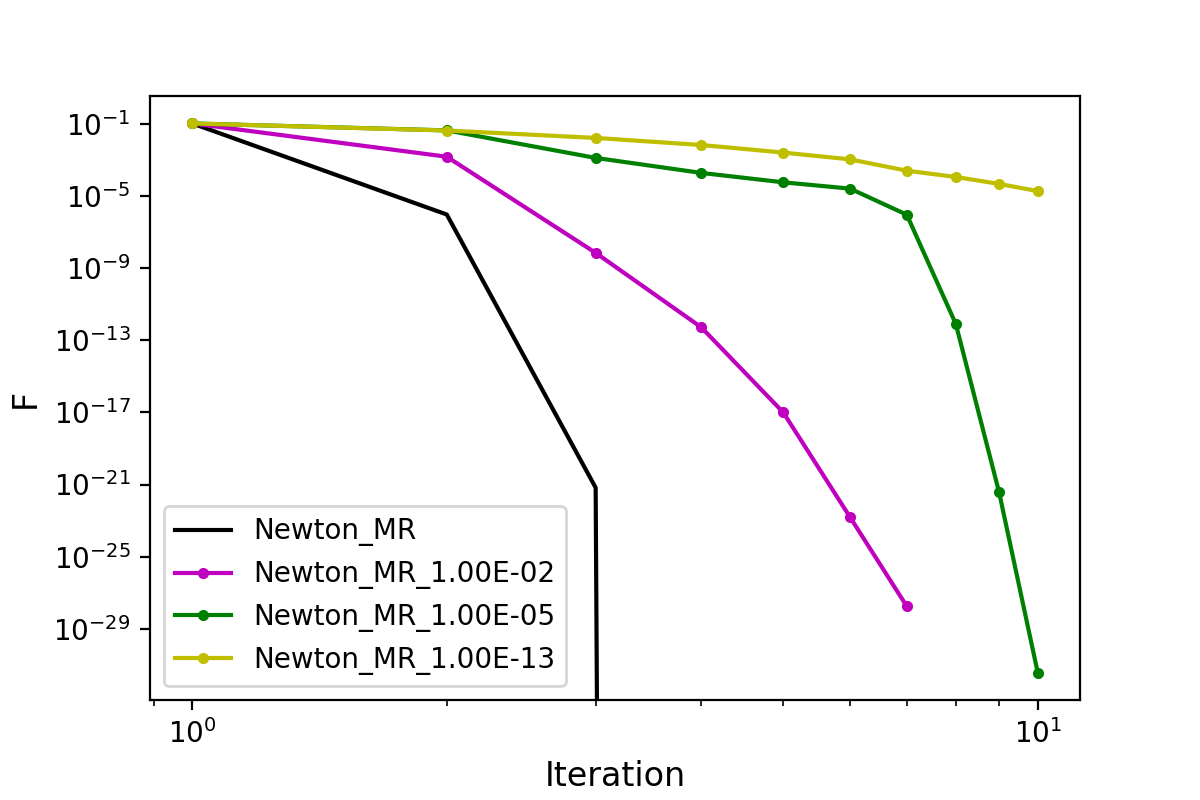}}
	\subfigure[$ \vnorm{\nabla f(\xxk)} $ vs.\ Iterations]
	{\includegraphics[scale=0.35]{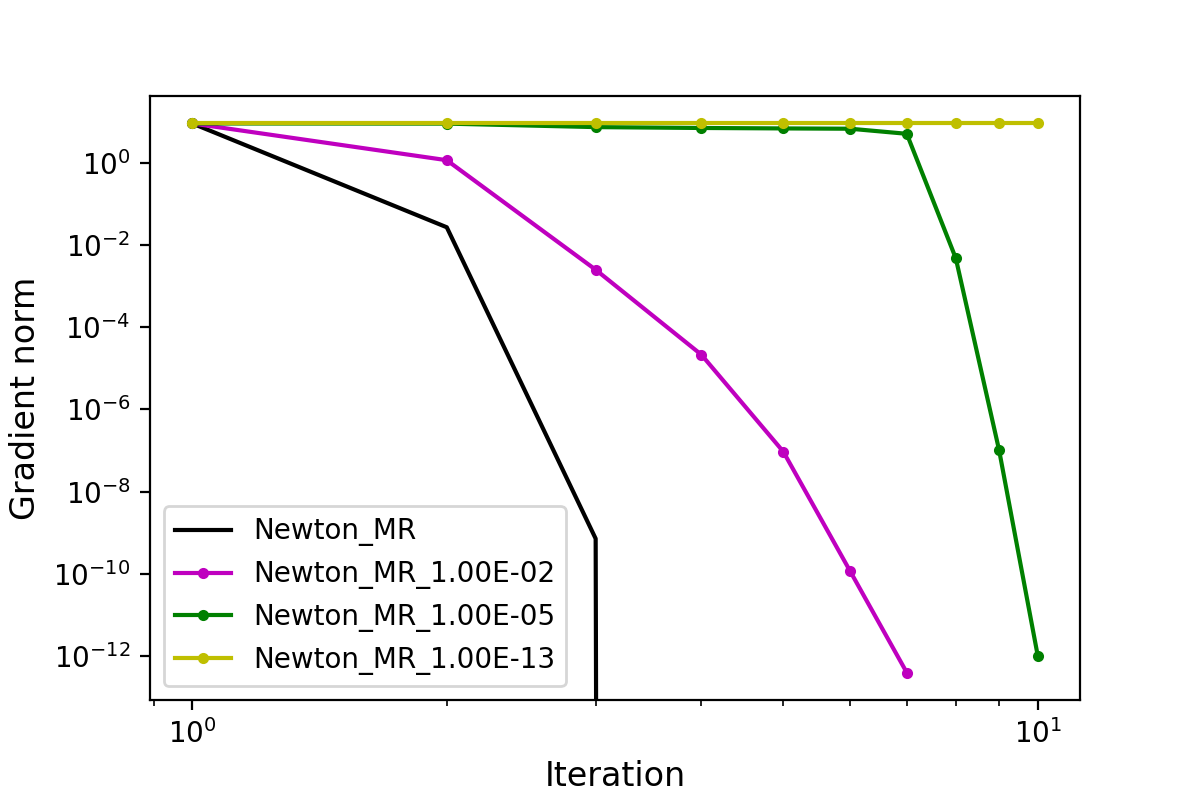}}
	\subfigure[Step-size vs.\ Iterations]
	{\includegraphics[scale=0.35]{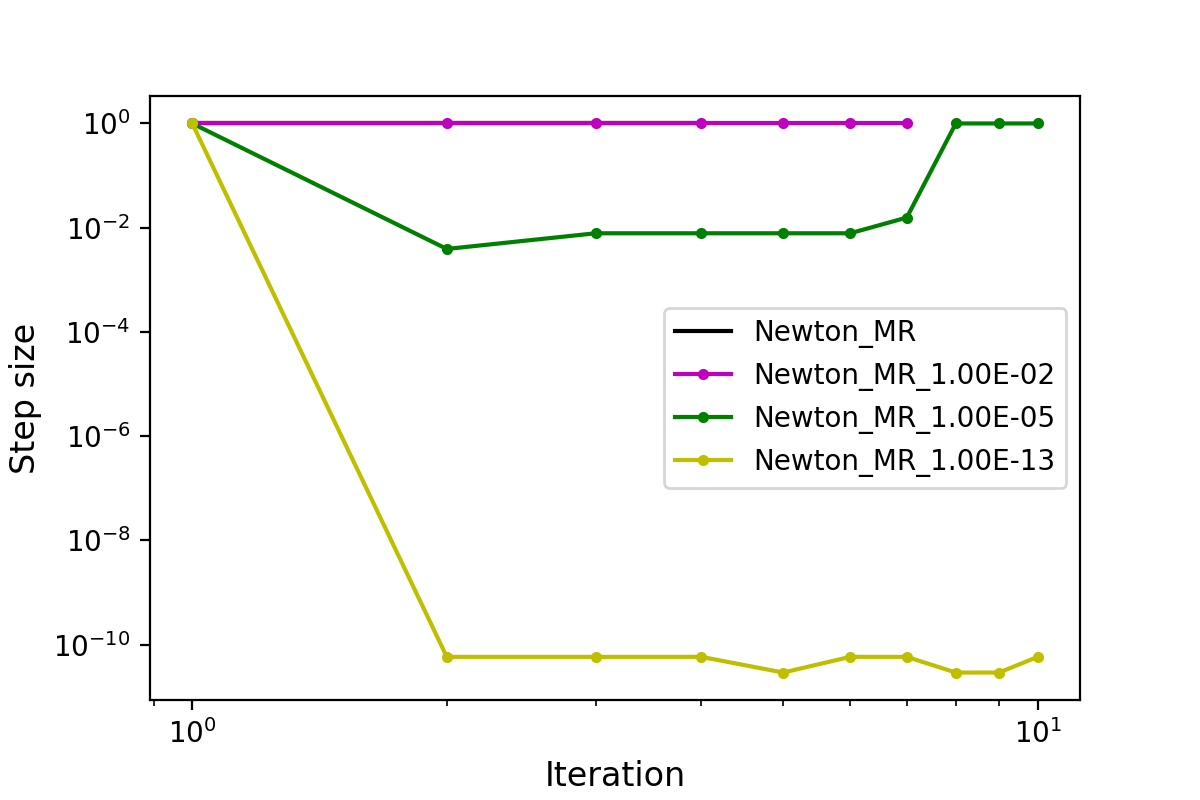}}
	\caption{Performance of Newton-MR under an unstable perturbation for $\varepsilon = 10^{-2}$, $10^{-5}$, and  $10^{-13}$ as in \Cref{sec:frac}, e.g., ``Newton-MR\_1.00E-02'' refers to the perturbation with $ \varepsilon = 10^{-2} $. ``Newton-MR'' refers to unperturbed algorithm. \label{fig:fraction}}
\end{figure}

\subsection{Unstable Perturbations}
\label{sec:frac}
We now verify the theoretical results of this paper in case of unstable perturbations where $ \nu < 1 $ and the perturbation is not acute. For this, we consider a simple two dimensional function, taken from \cite[Example 5]{roosta2018newton}, as
\begin{align}
\label{eq:fraction}
f(x_{1},x_{2}) = \frac{a x_{1}^2}{b - x_{2}},
\end{align}
where $ \textnormal{dom}(f) = \big\{(x_{1}, x_{2}) \;|\; x_{1} \in \mathbb{R}, \; x_{2} \in (-\infty,b) \cup (b, \infty)\big\} $. Clearly, $ f $ is unbounded below and, admittedly, this example is of little interest in optimization. In fact, applying \cref{alg:Newton_invex_sub} to \cref{eq:fraction} amounts to finding its stationary points, which are of the form $ (0,x_2) \in \textnormal{dom}(f) $. Nonetheless, \cref{eq:fraction} serves our purpose of demonstrating the effects of unstable perturbations in the performance of \cref{alg:Newton_invex_sub}. 

In \cite[Example 5]{roosta2018newton}, it has been shown that $ \nu = 8/9 $. 
%The gradient and the Hessian of \cref{eq:fraction} can be written as 
%\begin{align*}
%\bgg  = \left(\begin{array}{c}
%\frac{-2 a x_{1}}{x_{2}-b} \\ \\
%\frac{a x_{1}^{2}}{(x_{2}-b)^{2}}
%\end{array} \right),\quad 	
%\HH = \left(\begin{array}{cc}
%\frac{- 2 a }{x_{2} - b} & \frac{2 a x_{1}}{(x_{2}-b)^{2}} \\ \\
%\frac{2 a x_{1}}{(x_{2}-b)^{2}} & \frac{- 2 a x_{1}^{2}}{(x_{2}-b)^{3}}
%\end{array} \right).
%\end{align*}
Here, we consider $ a = 100, b = 1 $ and $ \xx_0 $ is chosen randomly from standard normal distribution. We draw a symmetric random matrix $ \EE $ from the Gaussian orthogonal ensemble and form the perturbed Hessian as $ \tHH = \HH + \varepsilon \EE/\vnorm{\EE} $. We consider \cref{alg:Newton_invex_sub} with exact updates for unperturbed as well as  perturbed Hessian with $ \varepsilon = 10^{-2}$, $10^{-5}$, and $10^{-13}$, respectively. As seen in \cref{fig:fraction}, for such an unstable perturbation, better approximations to the true Hessian, perhaps unintuitively, do not necessarily help with the convergence.  In fact, smaller values of $ \varepsilon $ amount to search directions that grow unboundedly larger, which result in the step-size shrinking to zero to counteract such unbounded growth; see \cref{fig:pert} for a depiction of this phenomenon. These numerical observations reaffirm the theoretical results of \Cref{sec:convergence}. 

\subsection{Stable Perturbations}
\label{sec:stable}
In this section, we demonstrate the efficiency of \cref{alg:Newton_invex_sub} under inherently stable perturbations. Our empirical evaluations of this section are done in the context of finite-sum minimization \cref{eq:obj_sum}. We first make comparisons among several Newton-type methods. In particular, we consider Newton-MR, Newton-CG, as well as their stochastic variants in which the Hessian matrix is sub-sampled, while the function and its gradient are computed exactly. We also consider the classical Gauss-Newton as well as L-BFGS. We then turn our attention to comparison among sub-sampled Newton-MR and several first-order alternatives, namely SGD with and without momentum \cite{sutskever2013importance}, Adagrad \cite{duchi2011adaptive}, RMSProp \cite{tijmen2012rmsprop}, Adam \cite{kingma2014adam}, and Adadelta \cite{zeiler2012adadelta}. We consider both deterministic and stochastic variants of the first-order algorithms where the gradient is, respectively, computed exactly and estimated using sub-samples. All first-order algorithms in this section use constant step-sizes, which are carefully fine-tuned for each experiment to give the best performance in terms of reducing the objective value.

\paragraph{Complexity Measure}
In all of our experiments, in addition to ``wall-clock'' time, we consider total number of oracle calls of function, gradient and Hessian-vector product as a complexity measure for evaluating the performance of each algorithm. Similar to \cite[Section 4]{roosta2018newton}, this is a judicious decision as measuring ``wall-clock'' time can be highly affected by particular implementation details. More specifically, for each $ i $ in \cref{eq:obj_sum}, after computing $ f_{i}(\xx) $, computing $ \nabla f_{i}(\xx) $ is equivalent to one additional function evaluation. In our implementations, we merely require Hessian-vector products $ \nabla^{2}f_{i}(\xx) \vv $, instead of forming the explicit Hessian, which amounts to two additional function evaluations, as compared with gradient evaluation. The number of such oracle calls for all algorithms considered here is given in \cref{table:oracle_calls}. 
\begin{table}	
	\centering
	\hspace*{-5.5mm}
	\scalebox{0.92}{
\begin{tabular}{|c|*{6}{c}|} 
	\hline 
	\multirow{2}{1.8cm}{\centering 2nd-order Methods} & Newton-MR  & Newton-CG & Gauss-Newton & ssNewton-MR & ssNewton-CG &  L-BFGS \\ [0.5ex]
	\cline{2-7}
	& $2 (t + \ell + 1) $ & $ 2 t + \ell + 2 $  & $ 2 t  + \ell + 2 $  & $ 2 t s / n + 2 (\ell + 1) $  & $ 2 t s / n + \ell + 2 $ & $ 2(\ell +1) $ \\
	\hline
	\hline
	\multirow{2}{1.8cm}{\centering 1st-order Methods} & Momentum & Adagrad & Adadelta & RMSprop & Adam &  SGD \\ [0.5ex]
	\cline {2-7}
	& $2 b / n $ & $ 2 b / n $  & $ 2 b / n $ & $ 2 b / n $  & $ 2 b / n $ & $ 2 b / n $  \\  
	\hline 
\end{tabular}
	}
	\caption{Complexity measure for each iteration of the algorithms for a finite-sum minimization problem involving $ n $ functions. Sub-sampled variants of Newton-MR and Newton-CG are referred to, respectively as ``ssNewton-MR'' and ``ssNewton-CG''. We also use $ t $ and $ \ell $ to denote, respectively, the total number of iterations for the corresponding inner solver and the line-search. The sample size for estimating the Hessian is denoted by $ s $, whereas $ b $ refers the mini-batch size for first-order methods. \label{table:oracle_calls}}
\end{table}

\paragraph{Parameters}
Throughout this section, we set the maximum iterations of the underlying inner solver, e.g., MINRES-QLP or CG, to $ 200 $ with an inexact tolerance of $ \theta=10^{-2} $. In \cref{alg:Newton_invex_sub}, for the termination criterion and the Armijo line-search parameter, respectively, we set $ \tau = 10^{-10} $ and $ \rho=10^{-4} $. Both Newton-CG and Gauss-Newton use the standard Armijo line-search whose parameter is also set to $ \rho=10^{-4} $. The parameter of the strong Wolfe curvature condition, used for L-BFGS, is $ 0.4 $. The history size of L-BFGS will be kept at $ 20 $ past iterations.
In the rest of this section, all methods are always initialized at $ \xx_{0} = \bm{0} $. For Newton-type methods, the initial trial step-size in line-search is always taken to be one. 
%The step-size of first-order methods can be found in \cref{table:learning_rate}. 
%\begin{table}	
%	\centering
%	\begin{tabular}{lllllll}
%		\toprule
%		Problem & Momentum & Adagrad & Adadelta & RMSprop & Adam & SGD   \\
%		\midrule
%		Softmax (MNIST) & $ 10^{-5} $ & $ 10^{-4} $ & $ 10^{-1} $ & $ 10^{-5} $ & $ 10^{-5} $ & $ 10^{-5} $ \\ 
%		\midrule
%		Softmax (Cifar10) & $ 10^{-7} $ & $ 10^{-6} $ & $ 10^{-4} $ & $ 10^{-6} $ & $ 10^{-7} $ & $ 10^{-7} $ \\ 
%		\midrule
%		GMM & $ 10^{-8} $ & $ 10^{-1} $ & $ 10 $ & $ 10^{-2} $ & $ 10^{-2} $ & $ 10^{-8} $ \\ 
%		\bottomrule
%	\end{tabular}
%	\caption{Step-sizes used for various first-order methods. For all problems, the step-size for each method was chosen after careful fine-tuning to obtain the best performance. \label{table:learning_rate}}
%\end{table}

\subsubsection{Softmax regression}
\label{sec:softmax}

Here, we consider the softmax cross-entropy minimization problem without regularization. More specifically, we have
\begin{align}
\label{eq:softmax}
f(\xx) \triangleq \mathcal{L}(\xx_{1},\xx_{2},\ldots,\xx_{C-1})  = \sum_{i=1}^{n} \left(\log \left(1+\sum_{c' = 1 }^{C-1} e^{\dotprod{\aa_{i}, \xx_{c'}}}\right)  - \sum_{c = 1}^{C-1}\mathbf{1}(b_{i} = c) \dotprod{\aa_{i},\xx_{c}}\right),
\end{align}
where $\{\aa_i, b_i\}_{i=1}^{n}$ with $\aa_i \in \reals^{p}$, $b_i \in \{0, 1, \dots, C\}$ denote the 
training data, $ C $ is the total number of classes for each input data $ \aa_{i} $ and $\xx = (\xx_{1},\xx_{2},\dots,\xx_{C-1})$. Note that, in this case, we have $ d = (C-1) \times p $. It can be shown that, depending on the data, \cref{eq:softmax} is either strictly-convex or merely weakly-convex.
In either case, however, it has been shown in \cite{roosta2018newton} that $ \nu = 1 $, i.e.,  $\nabla f(\xxk) \in \text{Range}\left(\nabla^{2} f(\xxk)\right)$. 

\Cref{fig:softmax_2nd_mnist}, \ref{fig:softmax_1st_cifar10} and \ref{fig:softmax_1st_cifar10_fullg} depict, respectively, the performance of variants of Newton-MR as compared with other Newton-type methods and several (stochastic) first-order methods. As it can be seen, all variants of Newton-MR are not only highly efficient in terms of oracle calls, but also they are very competitive in terms of ``wall-clock'' time. In fact, we can see that sub-sampled Newton-MR converges faster than all first-order methods. This can be attributed to moderate per-iteration cost of sub-sampled Newton-MR, which is coupled with far fewer overall iterations.

\begin{figure}[htb]
	\centering
	\subfigure[$ f(\xxk) $ vs.\ Oracle calls]
	{\includegraphics[scale=0.35]{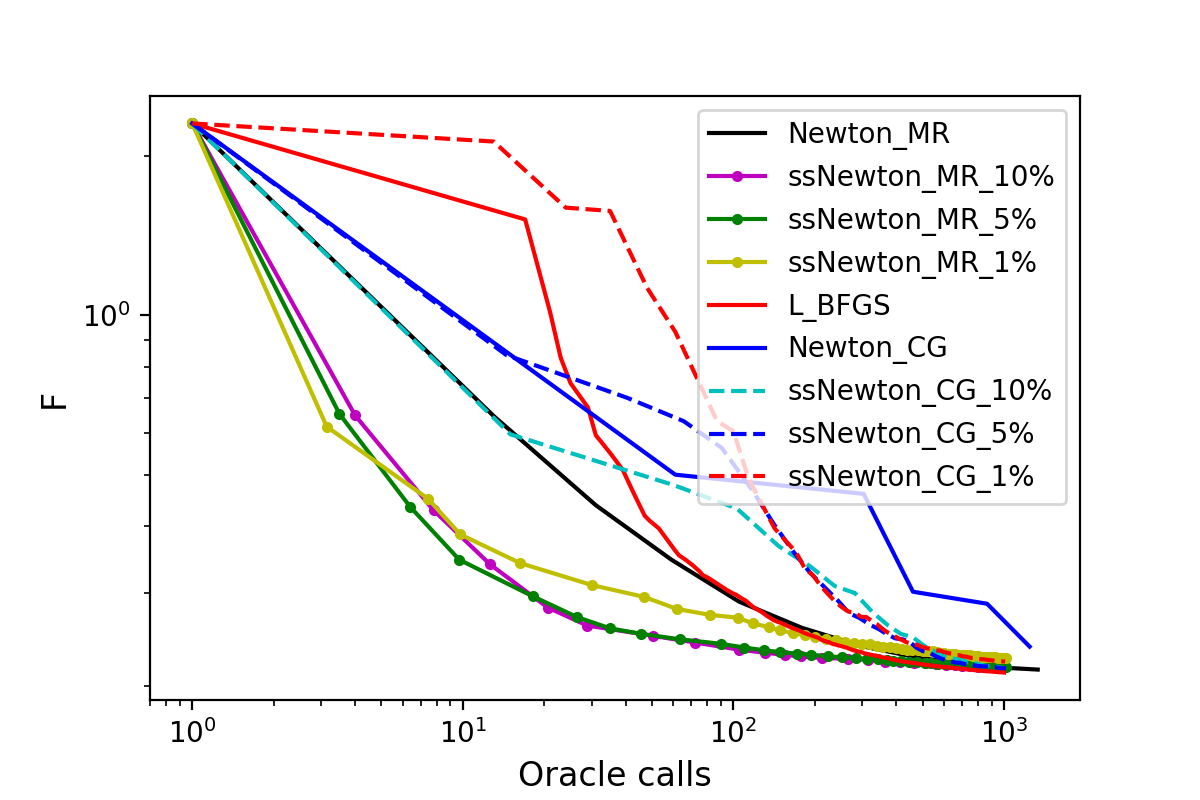}}
	\subfigure[$ \vnorm{\nabla f(\xxk)} $ vs.\ Oracle calls]
	{\includegraphics[scale=0.35]{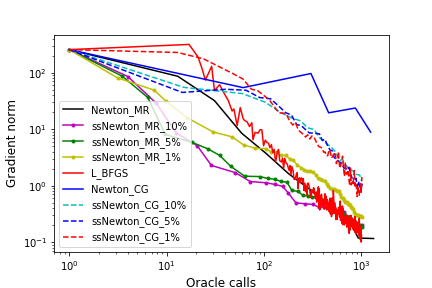}}
	\subfigure[$ f(\xxk) $ vs.\ Time (sec)]
	{\includegraphics[scale=0.35]{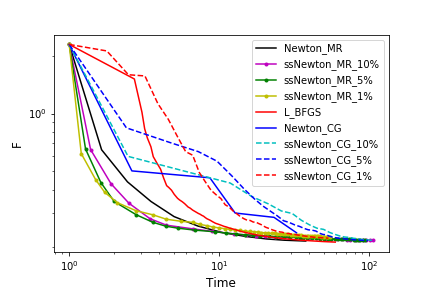}}
	\caption{Comparison among Newton-type methods on \cref{eq:softmax} using MNIST dataset. Here, sample sizes are chosen as $ s = 0.1n, 0.05n$ and $0.01n$, e.g., ``ssNewton-MR\_10\%'' uses $ s = 0.1n $. \label{fig:softmax_2nd_mnist}}
\end{figure}

\begin{figure}[htb]
	\centering
	\subfigure[$ f(\xxk) $ vs.\ Oracle calls]
	{\includegraphics[scale=0.35]{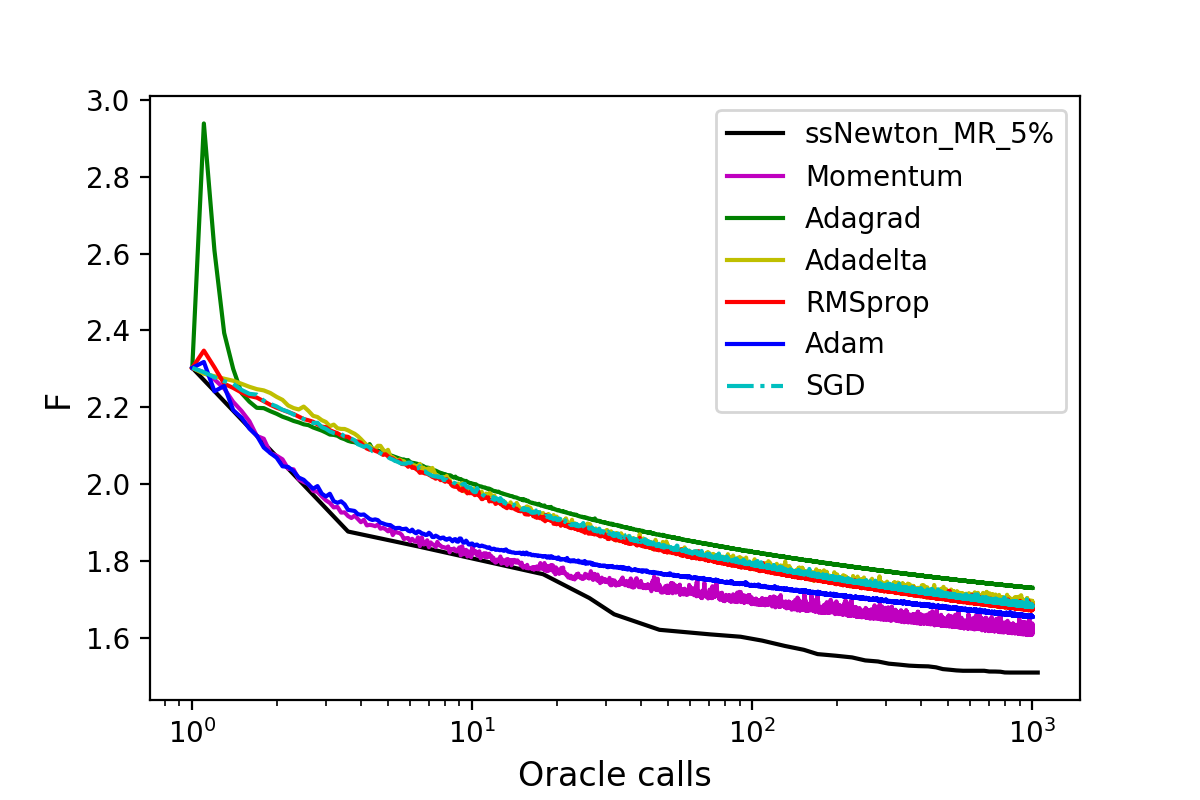}}
	\subfigure[$ \vnorm{\nabla f(\xxk)} $ vs.\ Oracle calls]
	{\includegraphics[scale=0.35]{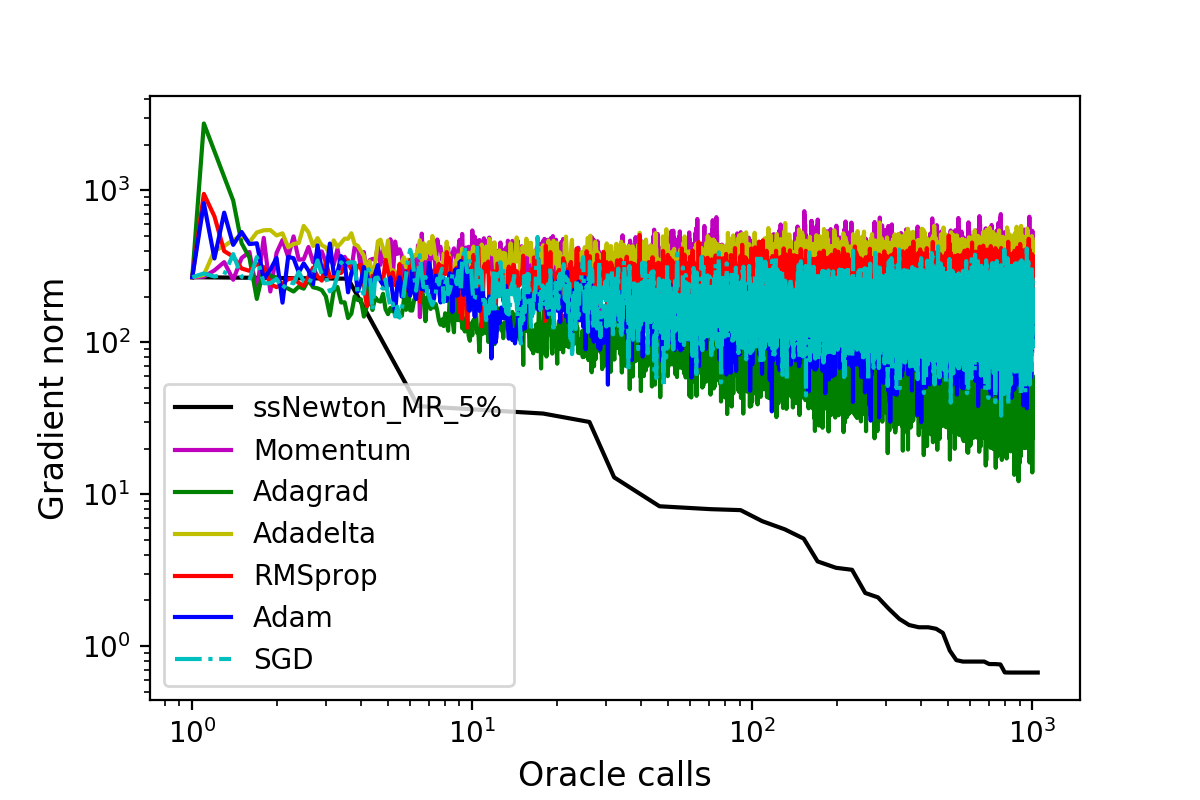}}
	\subfigure[$ f(\xxk) $ vs.\ Time (sec)]
	{\includegraphics[scale=0.35]{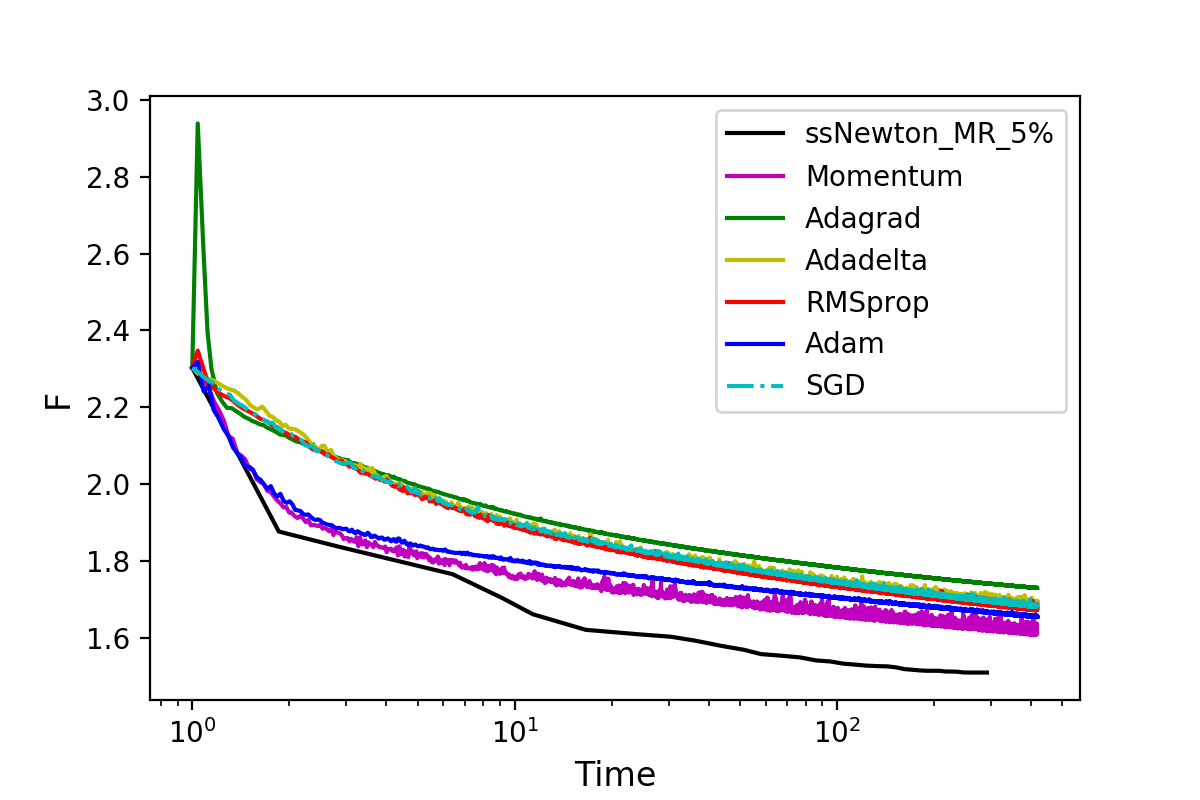}}
	\caption{Comparison among sub-sampled Newton-MR and several first-order methods on \cref{eq:softmax} using Cifar10 dataset. Here, sample/mini-batch sizes are $ s = b = 0.05n$. \label{fig:softmax_1st_cifar10}}
\end{figure}

\begin{figure}[htb]
	\centering
	\subfigure[$ f(\xxk) $ vs.\ Oracle calls]
	{\includegraphics[scale=0.35]{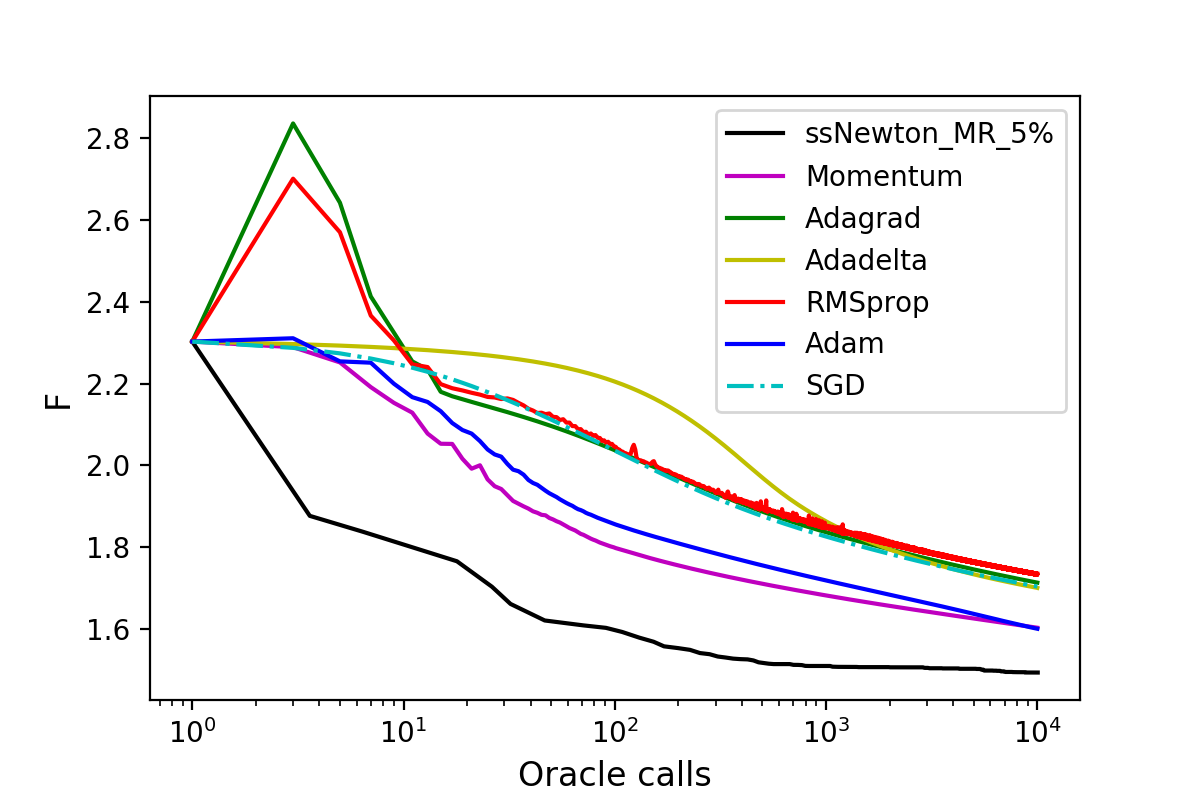}}
	\subfigure[$ \vnorm{\nabla f(\xxk)} $ vs.\ Oracle calls]
	{\includegraphics[scale=0.35]{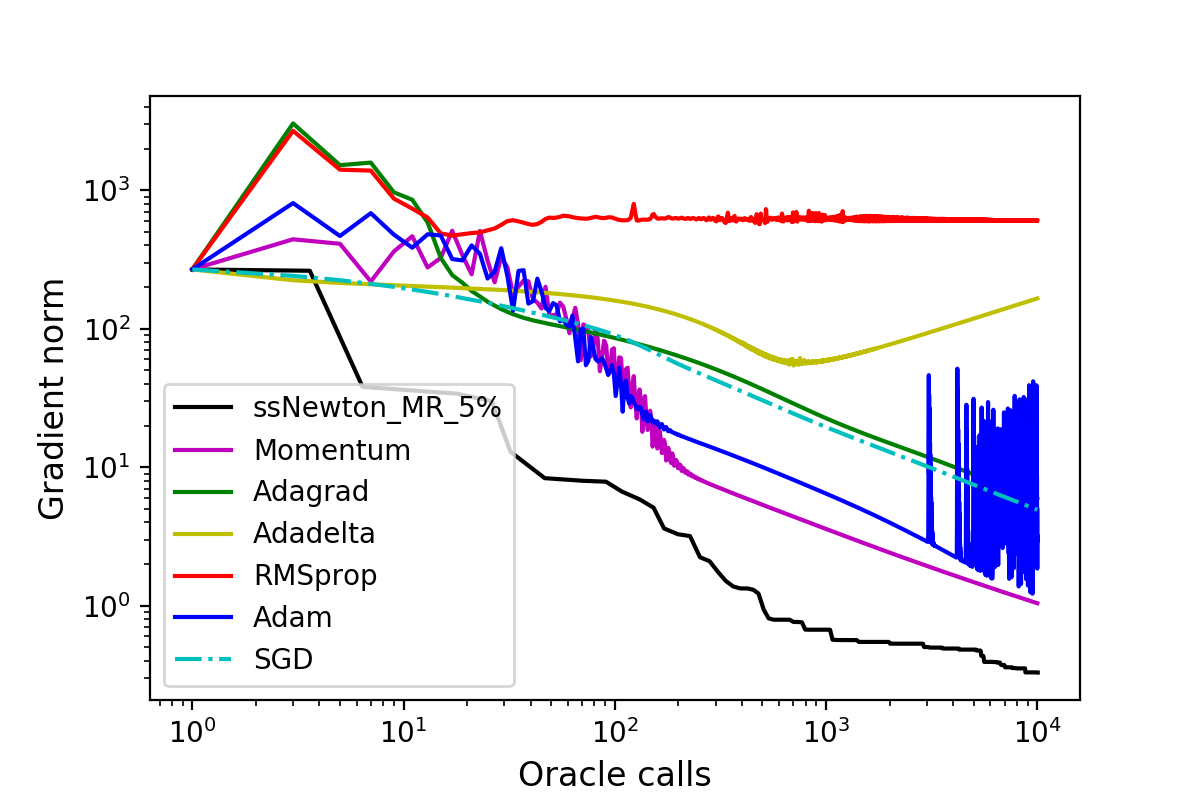}}
	\subfigure[$ f(\xxk) $ vs.\ Time (sec)]
	{\includegraphics[scale=0.35]{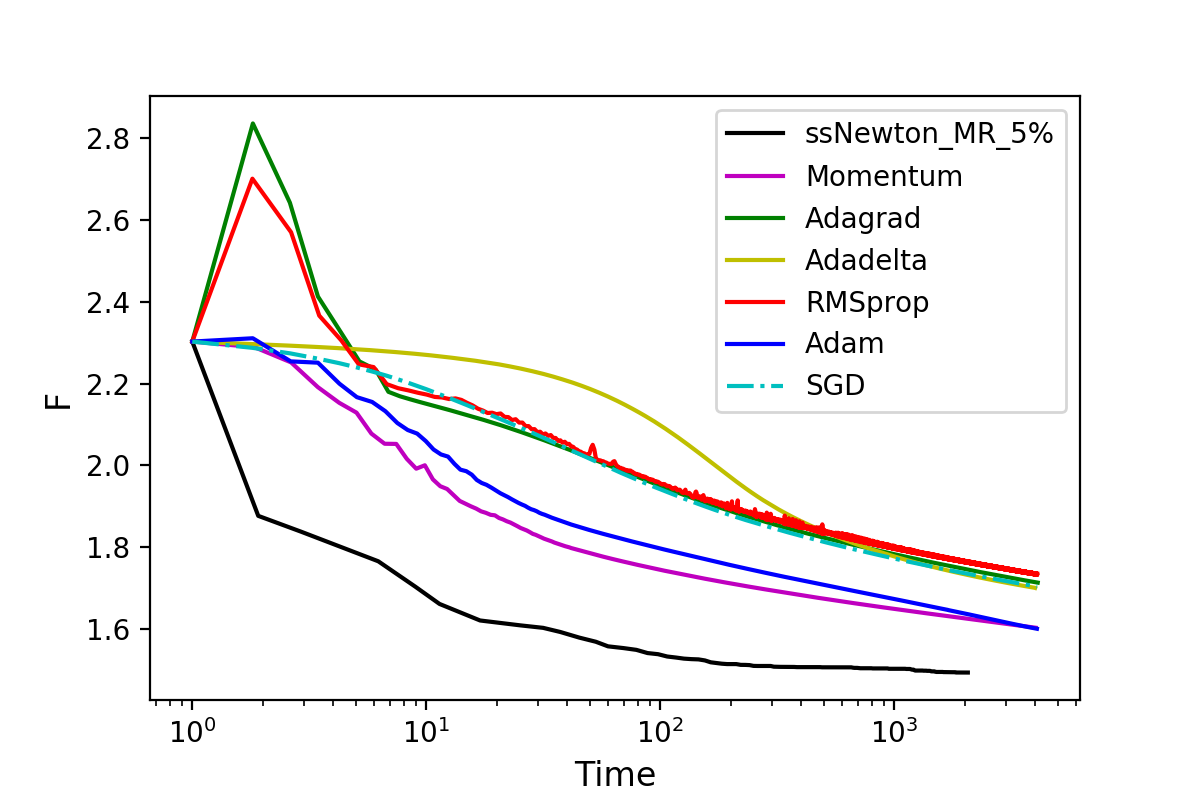}}
	\caption{Comparison among sub-sampled Newton-MR and several first-order methods on \cref{eq:softmax} using Cifar10 dataset. Here, sample/mini-batch sizes are $ s = 0.05n$, $ b = n $. \label{fig:softmax_1st_cifar10_fullg}}
\end{figure}

We then compare the performance of Newton-MR and Newton-CG as it relates to sensitivity to Hessian perturbations. We consider full and sub-sampled variants of both algorithms for a range of sample-sizes. \cref{fig:mr_vs_cg_mnist,fig:mr_vs_cg_hapt} clearly demonstrate that Newton-MR exhibits a great deal of robustness to Hessian perturbations, which amount to better performance for crude Hessian approximations. This is in sharp contrast to Newton-CG, which requires more accurate Hessian estimations to perform comparatively. Note the large variability in the performance of sub-sampled Newton-CG as compared with rather uniform performance of sub-sampled Newton-MR.

\begin{figure}[htb]
	\centering
	\subfigure[$ f(\xxk) $ vs.\ Iterations]
	{\includegraphics[scale=0.4]{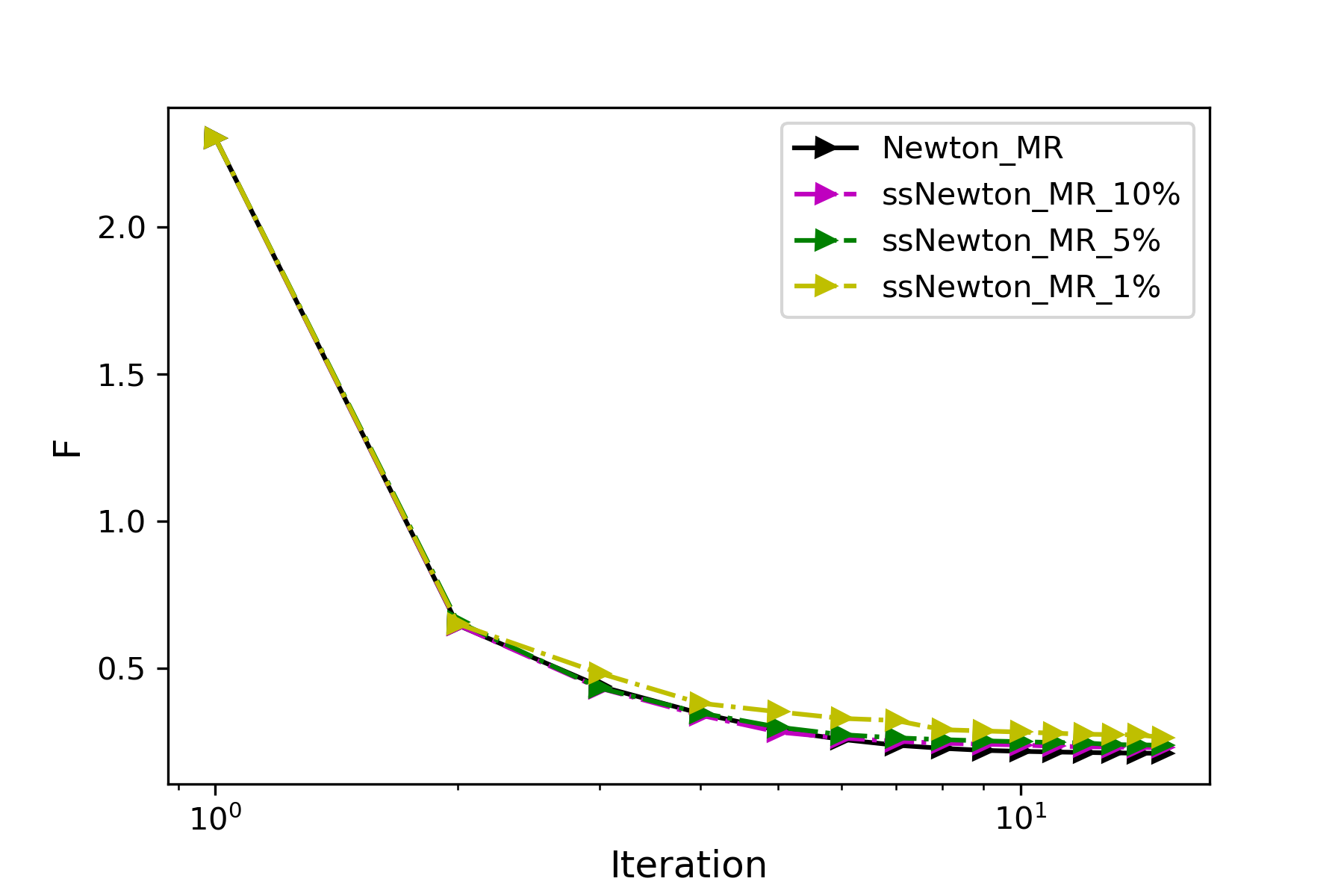}}
	\subfigure[$ \| \nabla f(\xxk) \| $ vs.\ Iterations]
	{\includegraphics[scale=0.4]{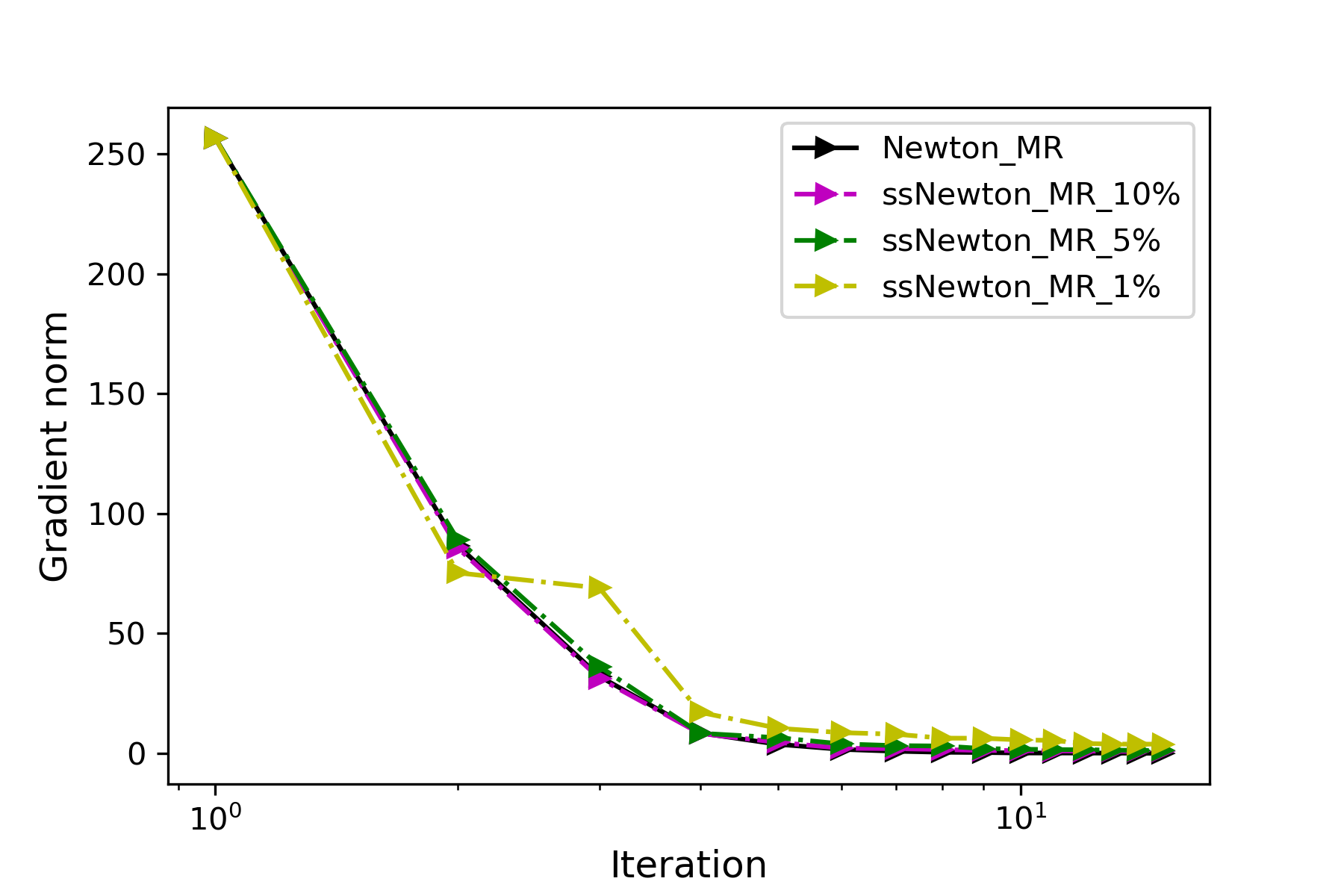}} \\
	\subfigure[$ f(\xxk) $ vs.\ Iterations]
	{\includegraphics[scale=0.4]{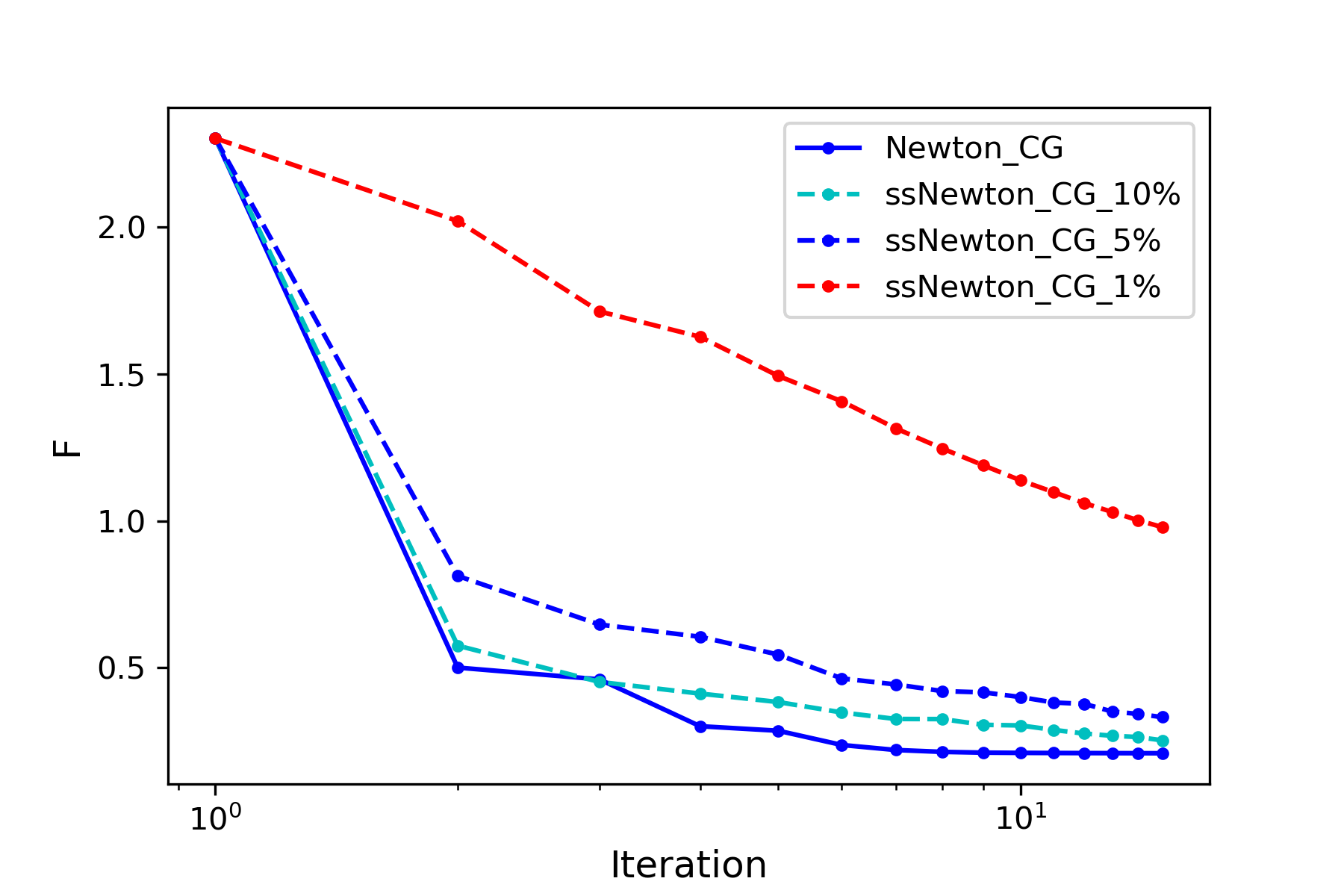}}
	\subfigure[$ \| \nabla f(\xxk) \| $ vs.\ Iterations]
	{\includegraphics[scale=0.4]{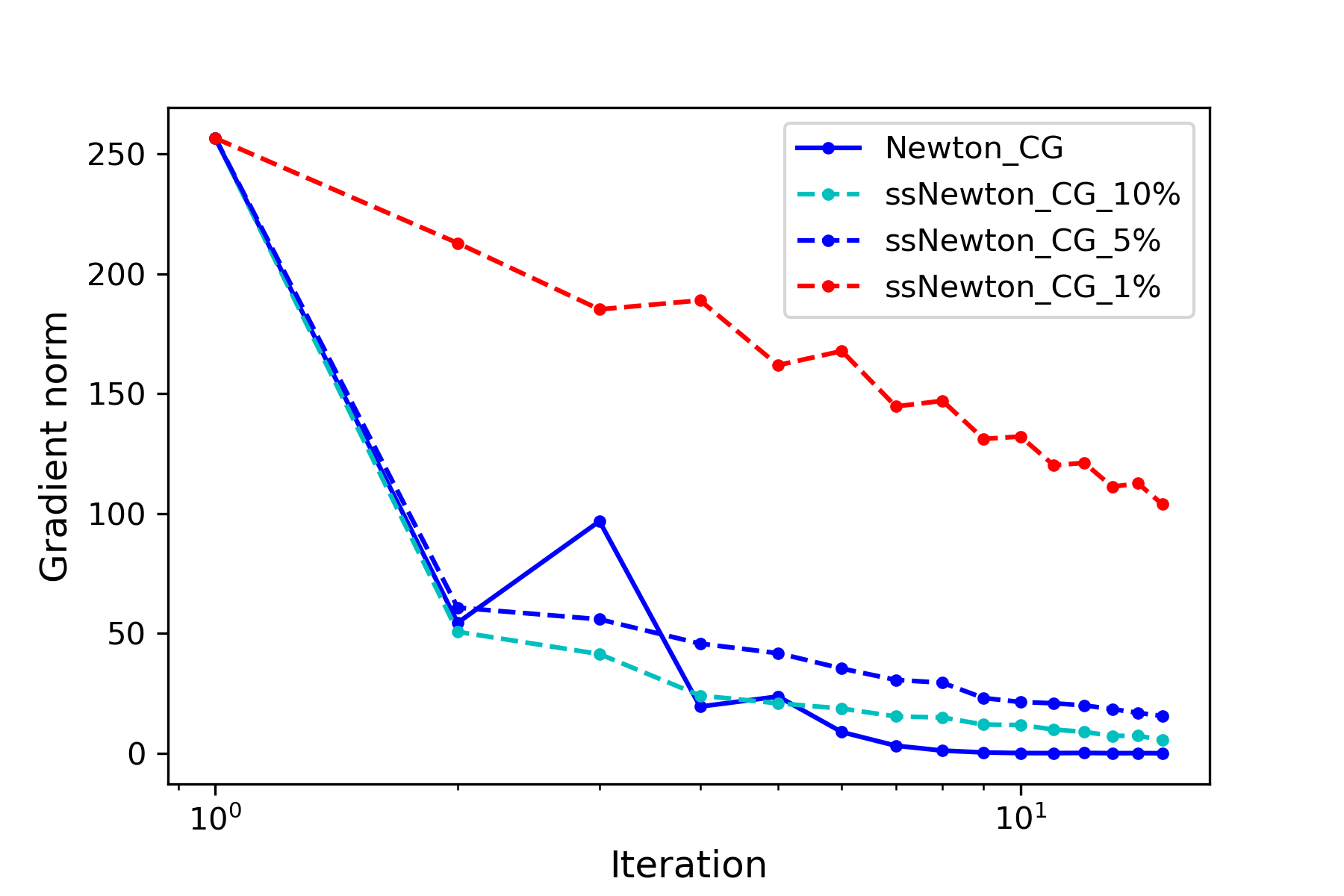}}
	\caption{Stability comparison between full and sub-sampled variants of Newton-MR and Newton-CG using $ s = 0.1n, 0.05n, 0.01n$  in \cref{table:oracle_calls} on \cref{eq:softmax} with MNIST dataset. \label{fig:mr_vs_cg_mnist}}
\end{figure}

\begin{figure}[htb]
\centering
\subfigure[$ f(\xxk) $ vs.\ Iterations]
{\includegraphics[scale=0.4]{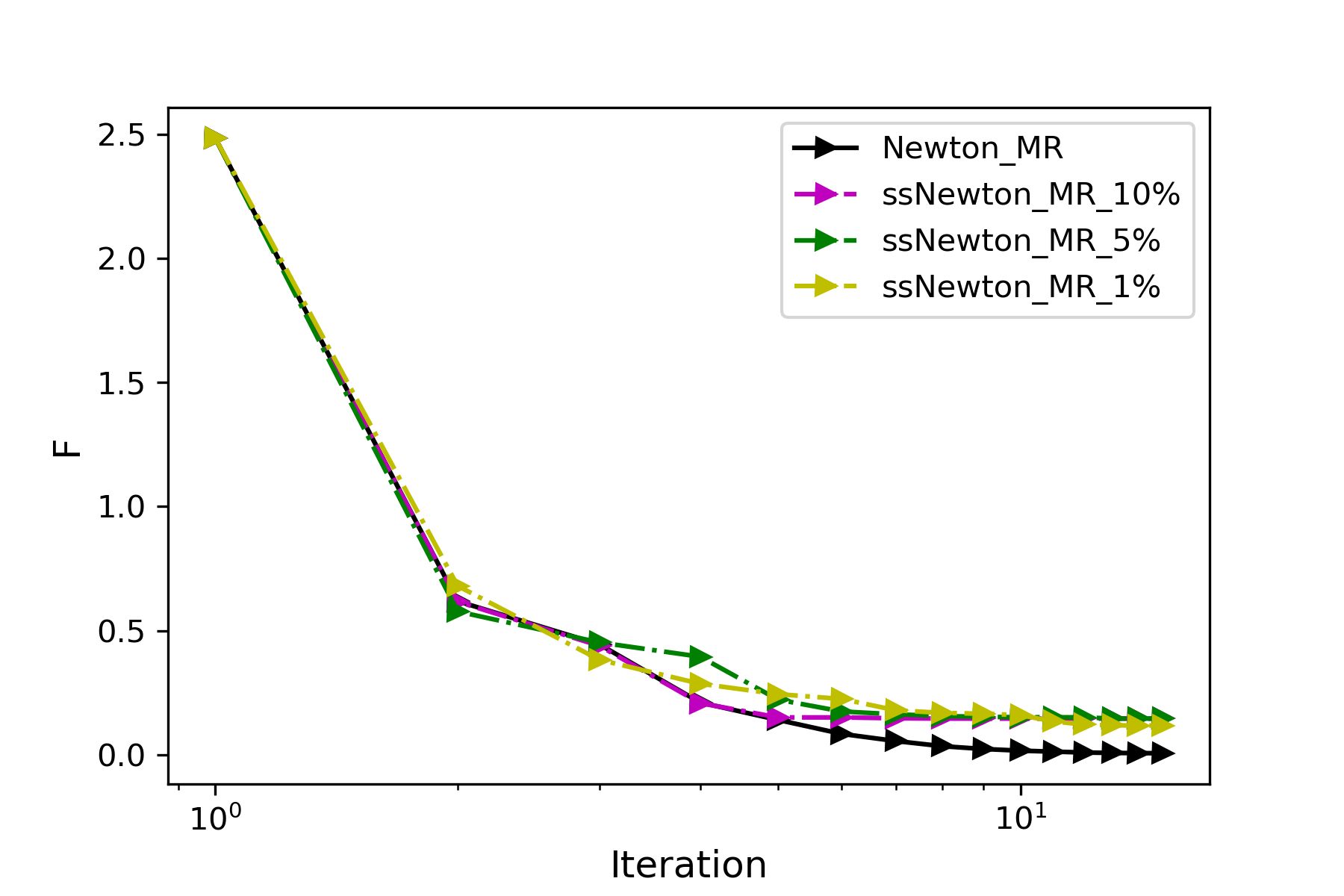}}
\subfigure[$ \| \nabla f(\xxk) \| $ vs.\ Iterations]
{\includegraphics[scale=0.4]{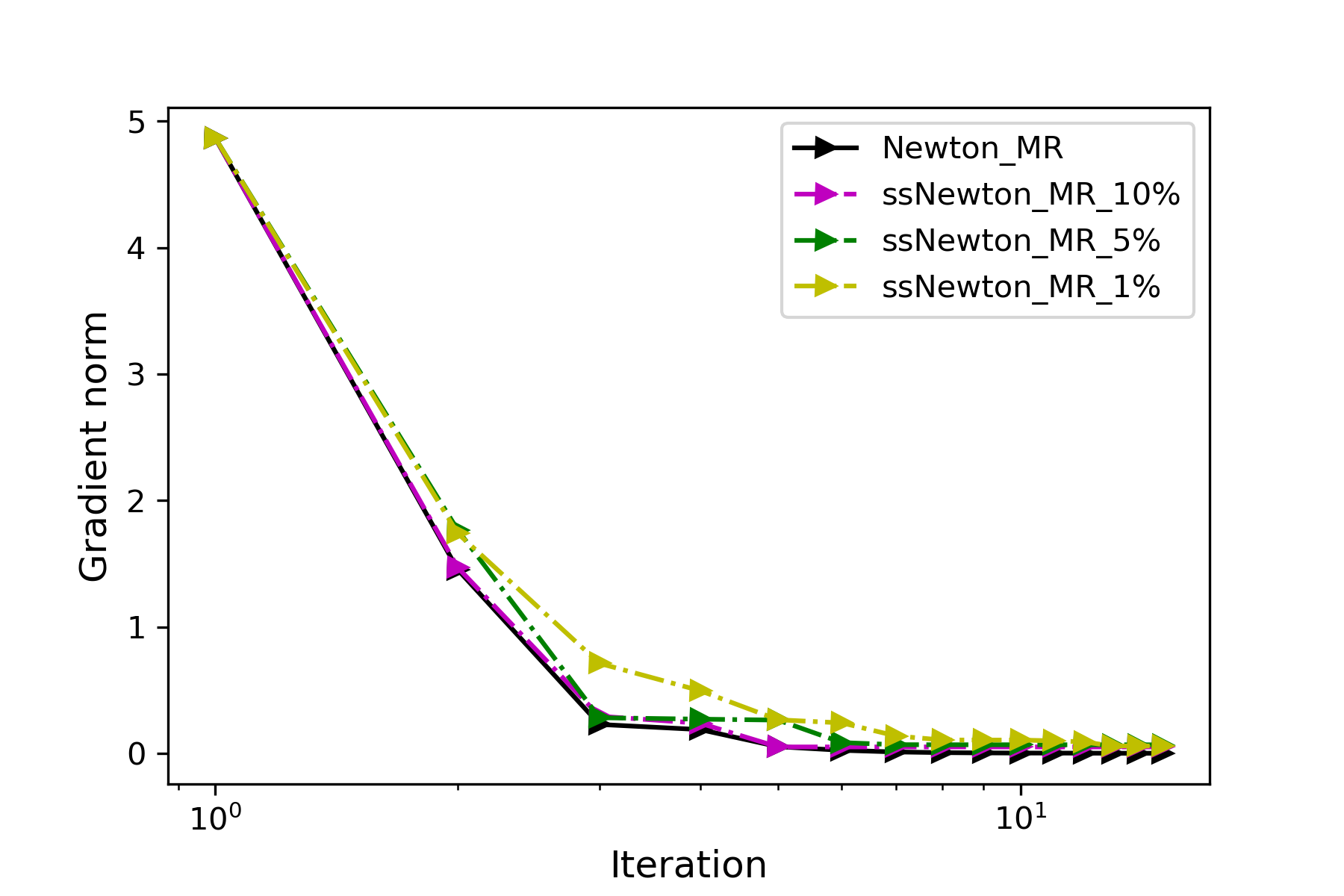}} \\
\subfigure[$ f(\xxk) $ vs.\ Iterations]
{\includegraphics[scale=0.4]{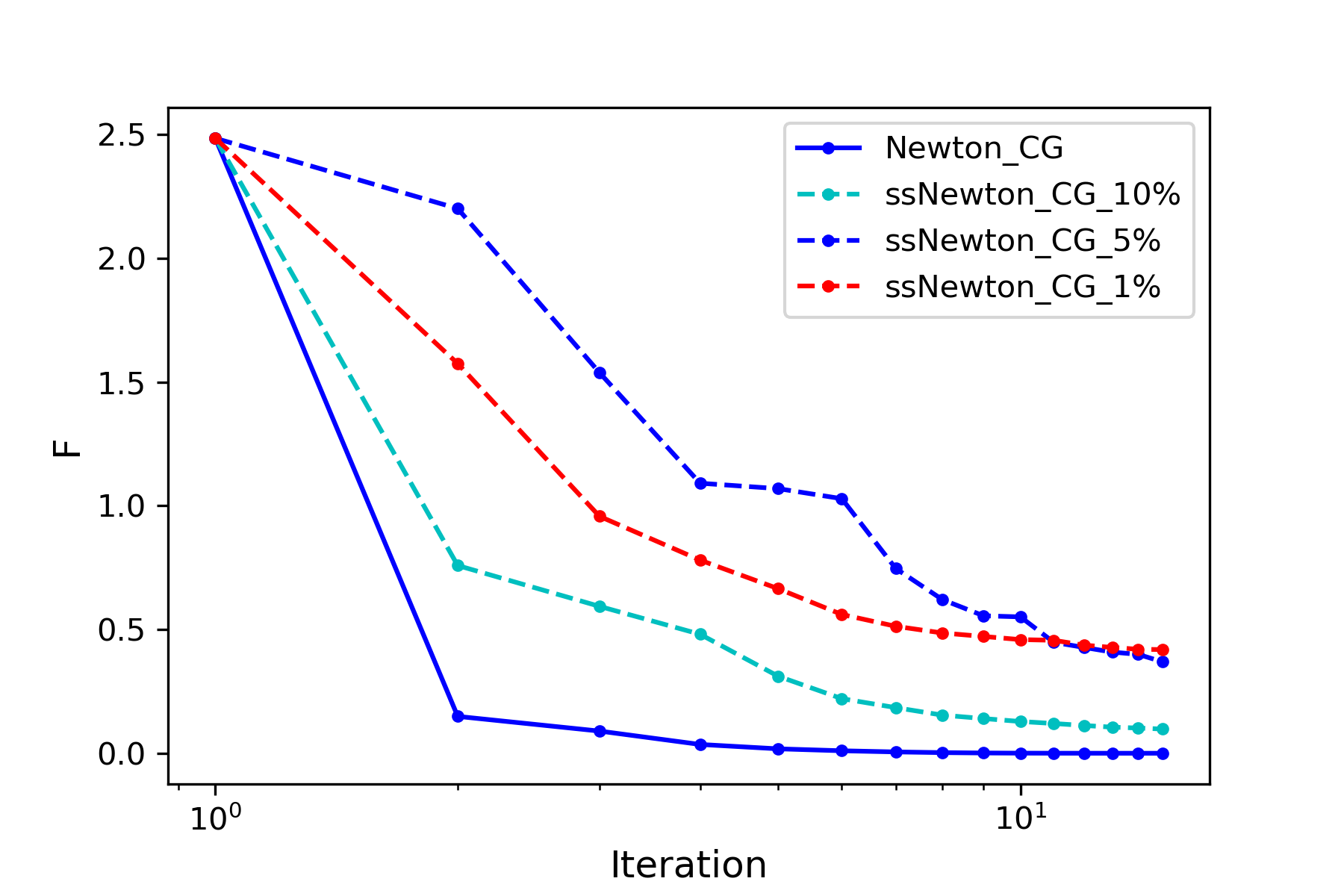}}
\subfigure[$ \| \nabla f(\xxk) \| $ vs.\ Iterations]
{\includegraphics[scale=0.4]{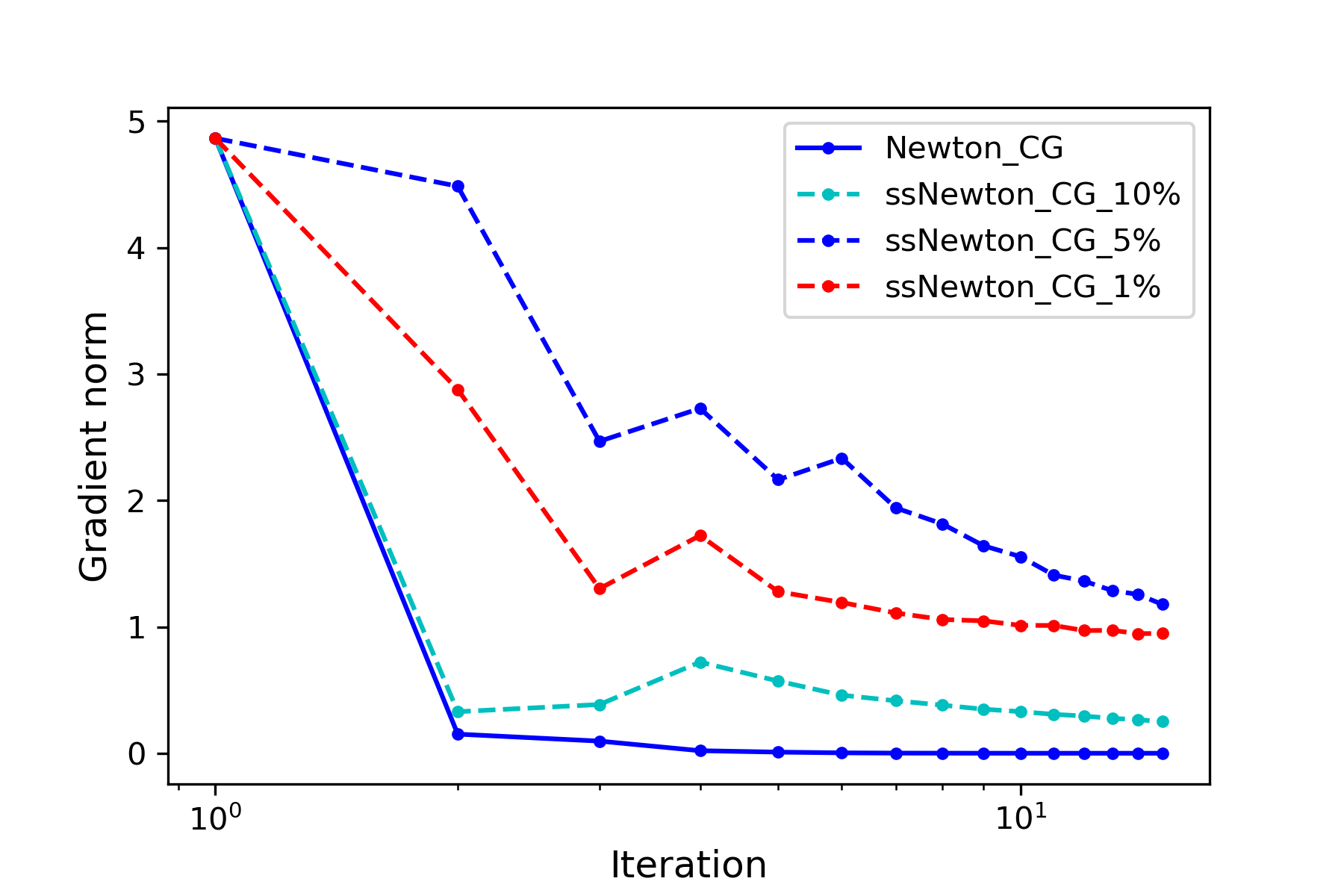}}
\caption{Stability comparison between full and sub-sampled variants of Newton-MR and Newton-CG using $ s = 0.1n, 0.05n, 0.01n$  in \cref{table:oracle_calls} on \cref{eq:softmax} with HAPT dataset. \label{fig:mr_vs_cg_hapt}}
\end{figure}

\subsubsection{Gaussian Mixture Model}
\label{sec:gmm}

Here, we consider an example involving a mixture of Gaussian densities. Although this problem is non-invex, it exhibits features that are close to being invex, e.g, small regions of saddle points and large regions containing global minimum \cite{mei2016landscape}. For simplicity, we consider a mixture model with two Gaussian components as 
\begin{align}
\label{eq:gmm}     
f(\xx) \triangleq \mathcal{L}(w, \uu, \vv) =  -\sum_{i=1}^n \log\left( \zeta(w) \Phi \left(\aa_i; \uu, \bm{\Sigma}_{1}\right) + \left( 1-\zeta(w) \right) \Phi \left(\aa_i; \vv,\bm{\Sigma}_{2}\right) \right), 
\end{align}
where $ \Phi $ denotes the density of the p-dimensional standard normal distribution, $ \aa_i \in \reals^{p} $ are the data points, $ \uu, \vv \in \reals^{p}, \bm{\Sigma}_{1}, \bm{\Sigma}_{2} \in \reals^{p \times p} $ are the corresponding mean vectors and the covariance matrices of two Gaussian distributions, $ w \in \reals $ and $ \zeta(t) = 1/(1+e^{-t}) $ is to ensure that the mixing weight lies within $ [0,1] $. Here, one can show that $ \nu = 1 $.
Note that, here, $ \xx \triangleq [w; \uu; \vv] \in \reals^{2p+1}$. 
In each run, we generate $ 1,000$ random data points, generated from the mixture distribution \cref{eq:gmm} with $ p = 100 $, and ground truth parameters as $ w^{\star} \sim \bm{\mathcal{N}}[0,1], \uu^{\star} \sim \bm{\mathcal{N}}[-1,1],  \vv^{\star} \sim \bm{\mathcal{U}}[3,4] $. Covariance matrices are constructed randomly, with controlled condition number, such that they are not axis-aligned. To establish this, we first randomly generate two $ p \times p $ matrices whose elements are i.i.d.\ drawn from standard normal distribution and uniform distribution, respectively. We then find the corresponding orthogonal bases, $ \QQ_{1}, \QQ_{2} $, using QR factorization. We then set $ \bm{\Sigma}_{i} = \QQ_{i}^{\intercal} \DD^{-1} \QQ_{i} $ where $ \DD $ is a diagonal matrix whose diagonal entries are chosen equidistantly from the interval $ [0,10^8] $. This way the condition number of each $ \bm{\Sigma}_{i} $ is $ 10^8 $. In all the figures,
\begin{align*}
\text{Estimation error at $ k\th $ iteration} \triangleq \hf \left( \frac{|w_{k} - w^{\star}|}{w^{\star}} + \frac{\|[\uu_{k}; \vv_{k}] - [\uu^{\star};\vv^{\star}]\|}{\|[\uu^{\star};\vv^{\star}]\|} \right).
\end{align*}

\begin{figure}[htb]
	\centering
	% \vspace*{-3mm}
	\subfigure[$ f(\xx) $]
	{\includegraphics[scale=0.35]{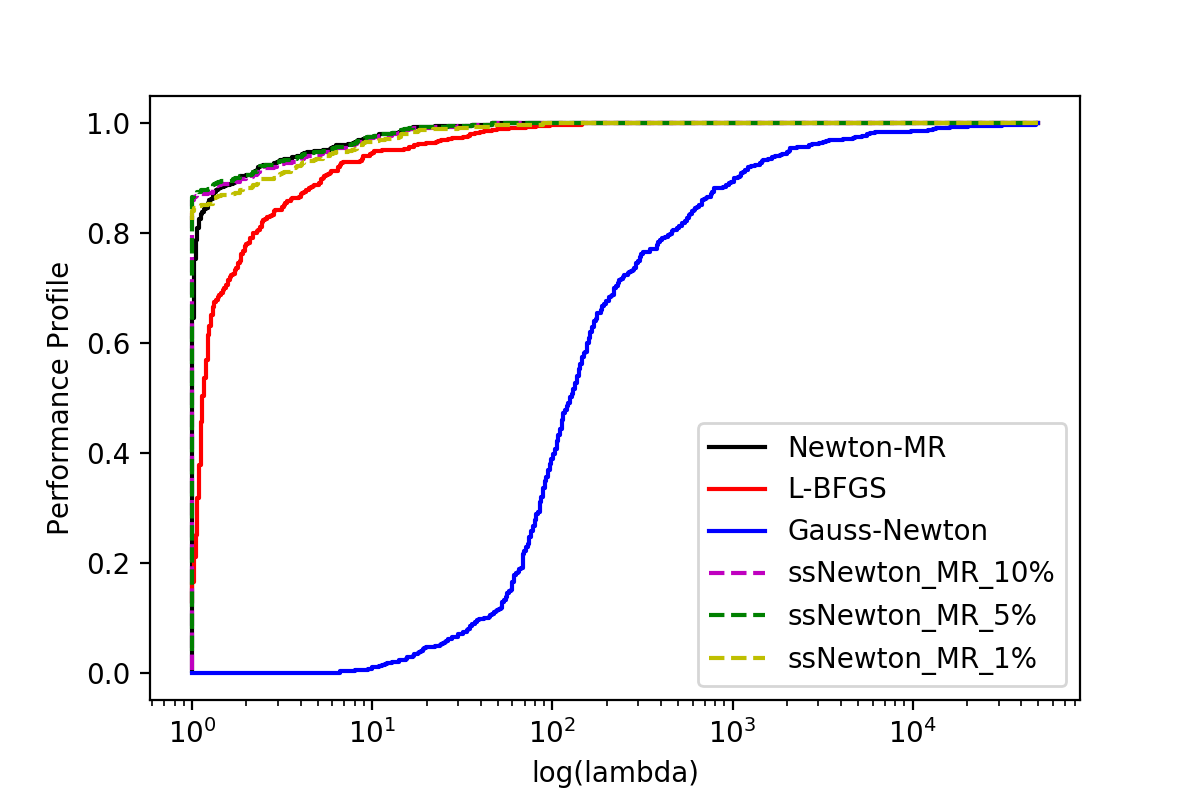}}
	\subfigure[$ \vnorm{\nabla f(\xx)} $]
	{\includegraphics[scale=0.35]{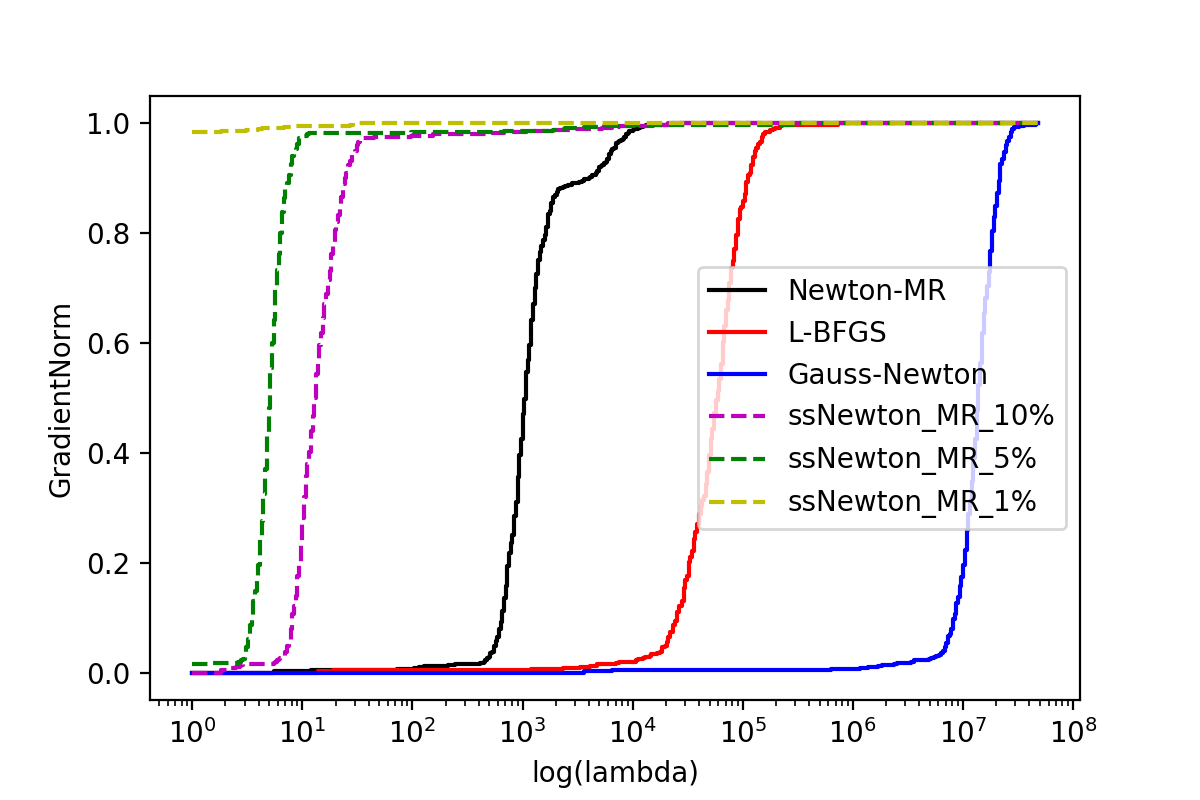}}
	\subfigure[Estimation error]
	{\includegraphics[scale=0.35]{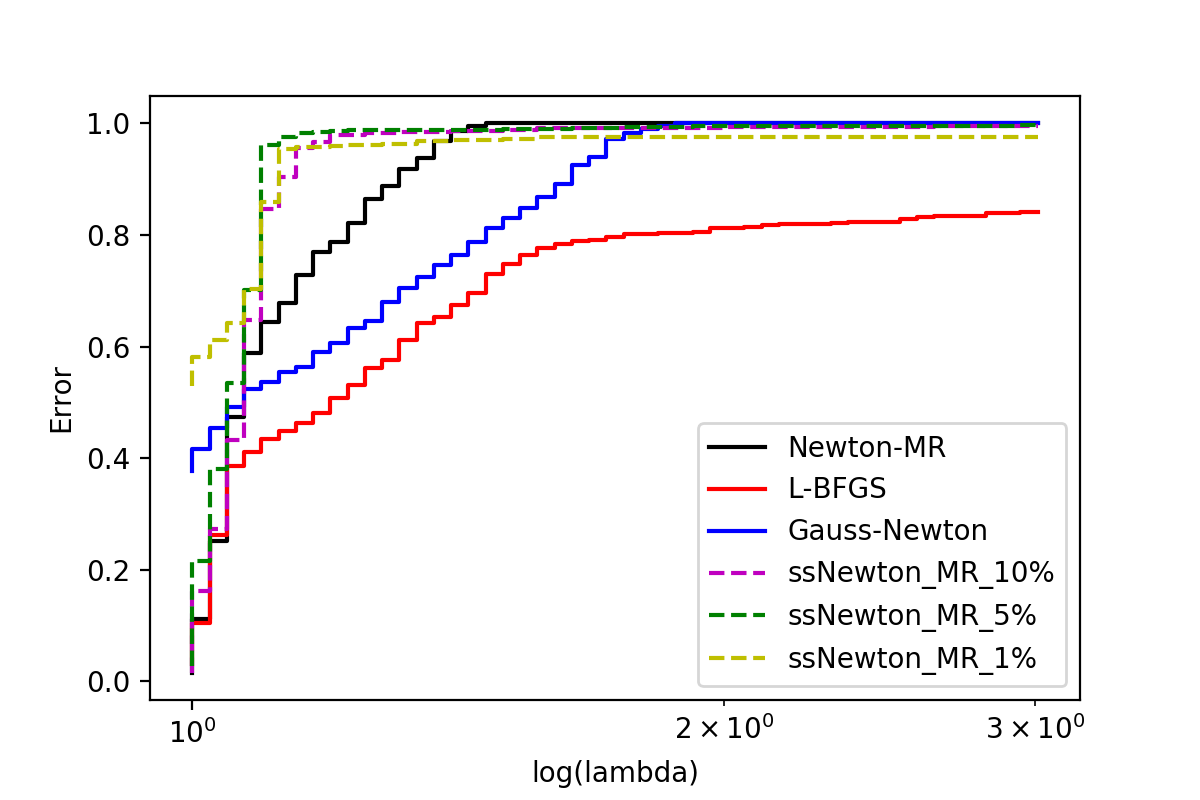}}
	\caption{Performance profile for 500 runs of various Newton-type methods for solving \eqref{eq:gmm} as detailed
		in \Cref{sec:gmm}. \label{fig:gmm_2nd}}
\end{figure}

\begin{figure}[htb]
	\centering
	% \vspace*{-3mm}
	\subfigure[$ f(\xx) $]
	{\includegraphics[scale=0.35]{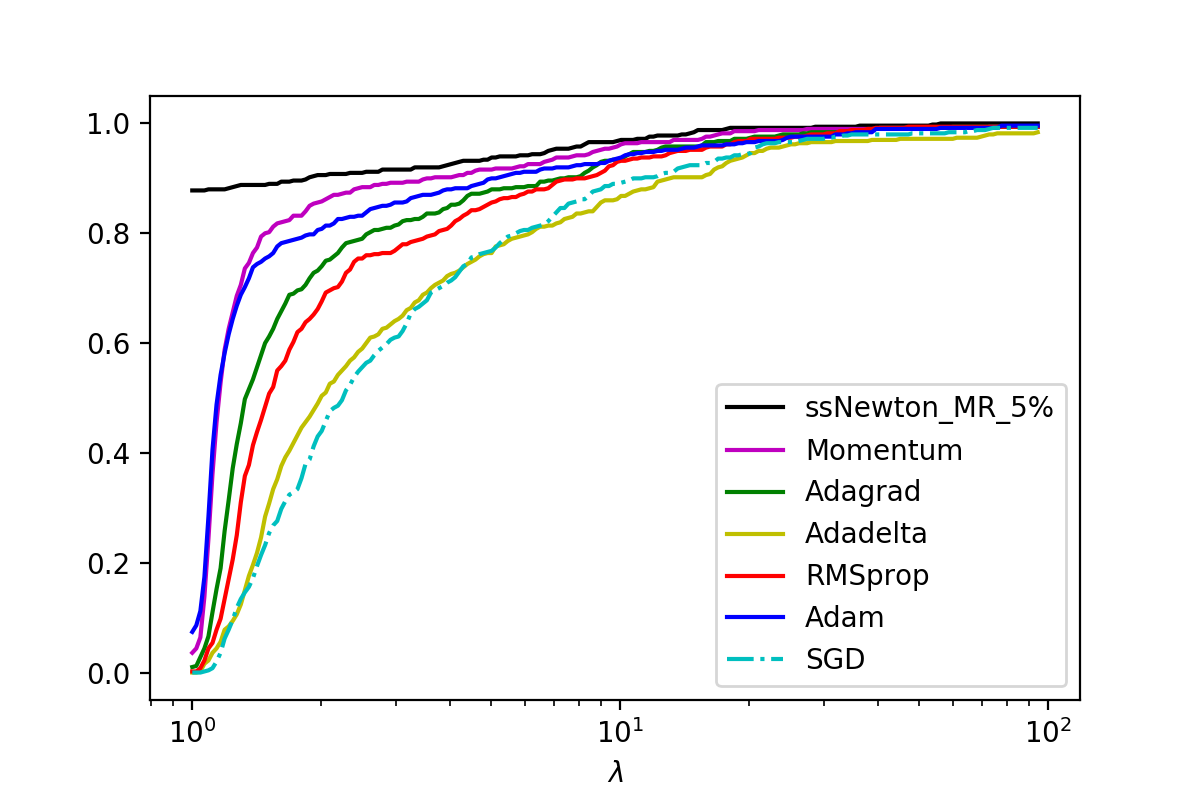}}
	\subfigure[$ \vnorm{\nabla f(\xx)} $]
	{\includegraphics[scale=0.35]{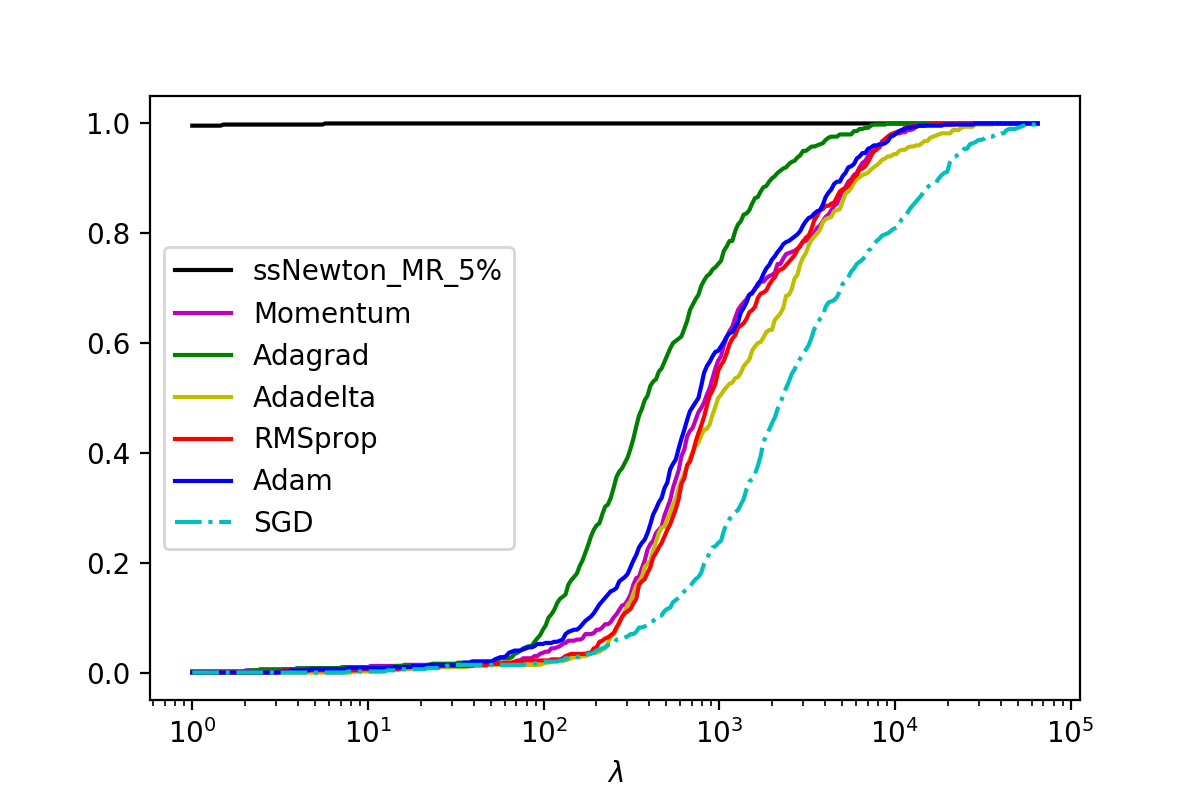}}
	\subfigure[Estimation error]
	{\includegraphics[scale=0.35]{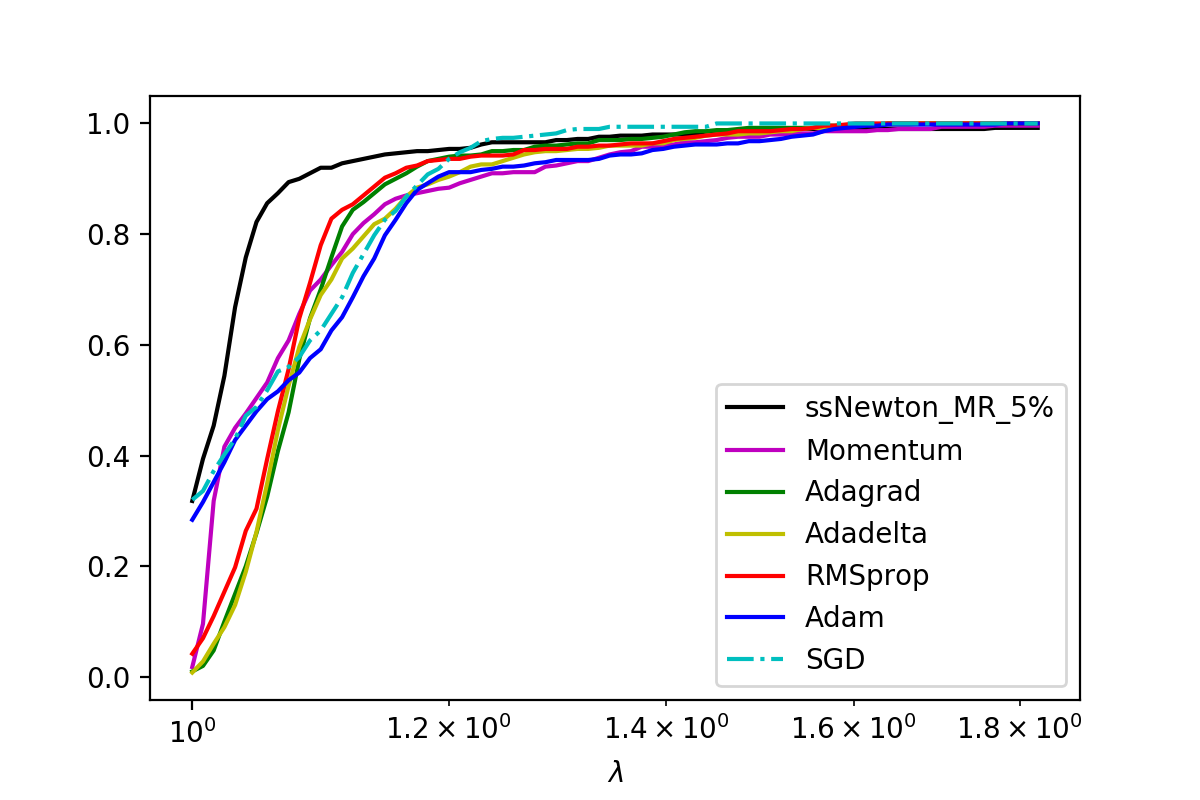}}
	\caption{Performance profile for 500 runs of Newton-MR variants and several first-order methods for solving \eqref{eq:gmm} as detailed in \Cref{sec:gmm}. Here, we have set $ s=b=0.05n $. \label{fig:gmm_1st_5}}
\end{figure}

\begin{figure}[!htbp]
	\centering
	% \vspace*{-3mm}
	\subfigure[$ f(\xx) $]
	{\includegraphics[scale=0.35]{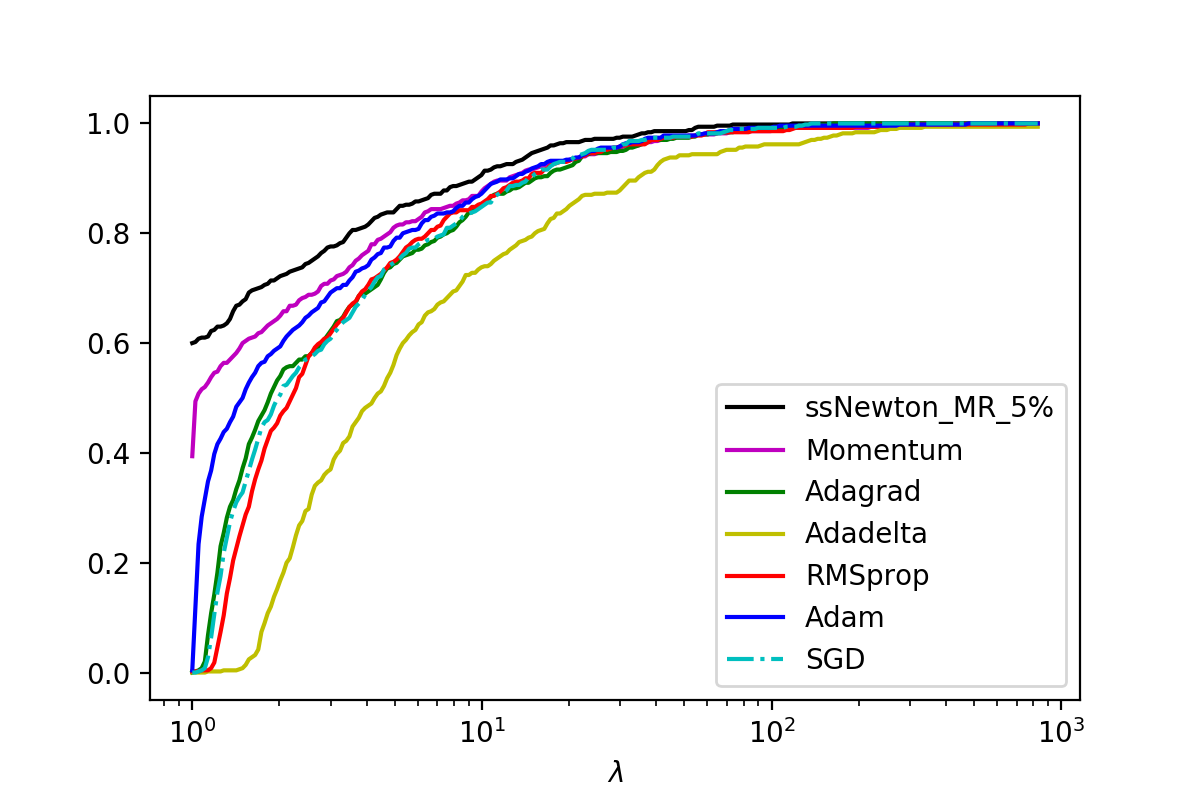}}
	\subfigure[$ \vnorm{\nabla f(\xx)} $]
	{\includegraphics[scale=0.35]{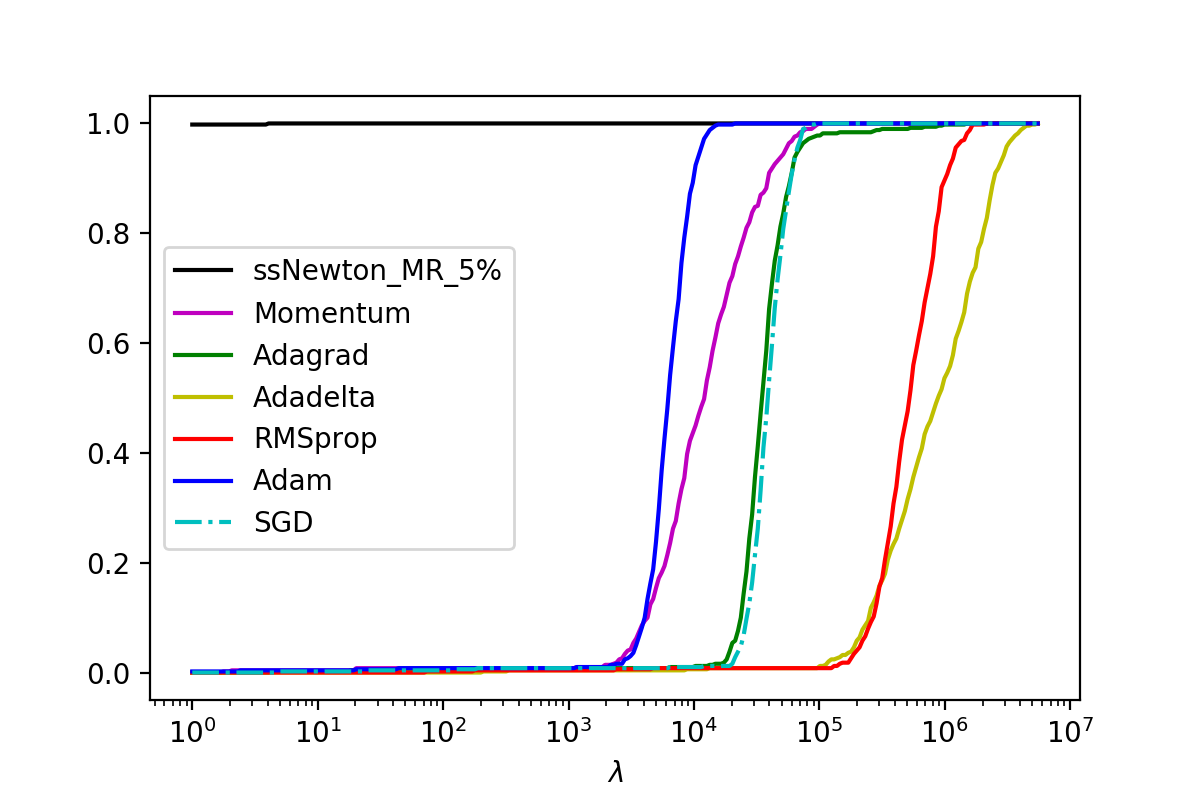}}
	\subfigure[Estimation error]
	{\includegraphics[scale=0.35]{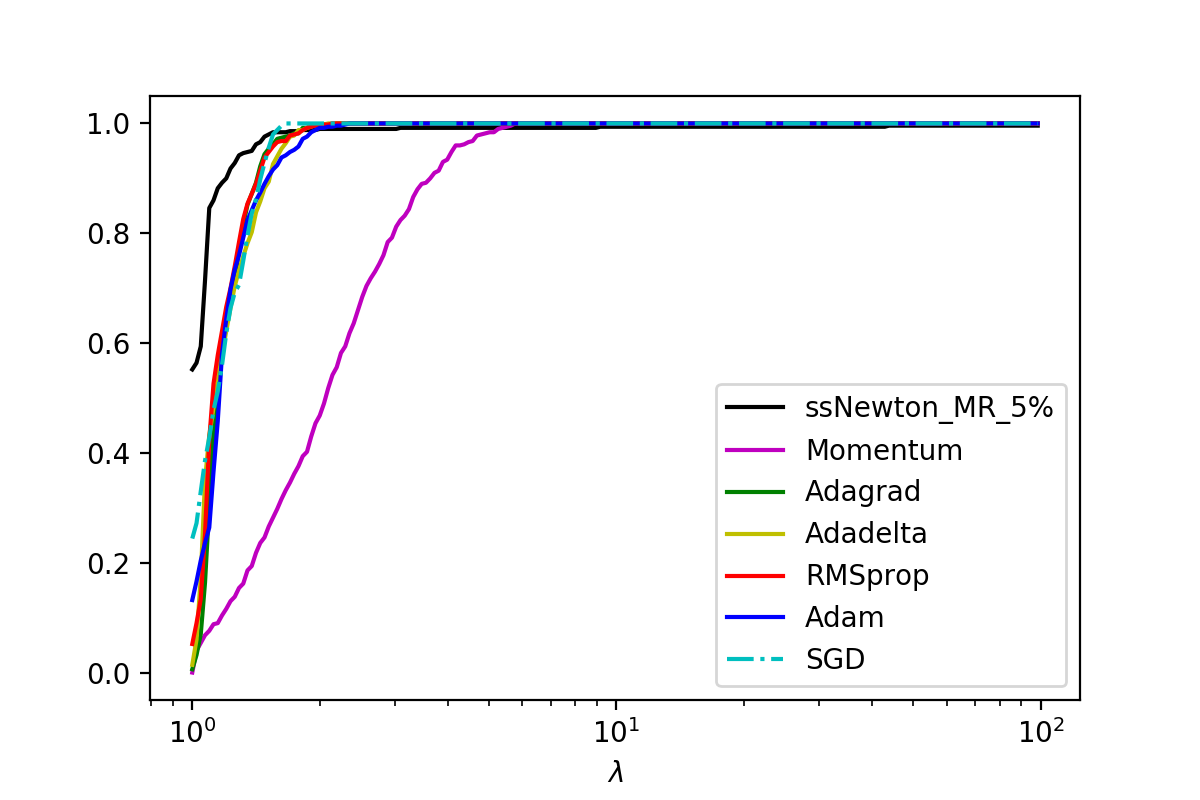}}
	\caption{Performance profile for 500 runs of Newton-MR variants and several first-order methods for solving \eqref{eq:gmm} as detailed in \Cref{sec:gmm}. Here, we have set $ s=0.05n $, $ b = n $. \label{fig:gmm_1st_full}}
\end{figure}

In our experiments, the classical Gauss-Newton method performed extremely poorly, and as a result we did not consider its sub-sampled variants. \Cref{fig:gmm_2nd} shows the performance profile plots \cite{dolan2002benchmarking,gould2016note} with 500 runs for Newton-type methods and \cref{fig:gmm_1st_5,fig:gmm_1st_full} depict the corresponding plots comparing variants of Newton-MR with several first-order methods using, respectively, sample/mini-batch sizes of $ 5\% $ and the full gradient. Recall that in performance profile plots, for a given $ \lambda $ in the x-axis, the corresponding value on the y-axis is the proportion of times that a given solver's performance lies within a factor $ \lambda $ of the best possible performance over all runs. 

As demonstrated by \cref{fig:gmm_2nd}, although L-BFGS performs competitively in terms of reducing the objective function, its performance in terms of parameter recovery and estimation error is far worse than all other methods. In contrast, all variants of Newton-MR have stable performance across all 500 runs, with sub-sampled variants exhibiting superior performance. \Cref{fig:gmm_1st_5} and \ref{fig:gmm_1st_full} also demonstrate similar superior performance compared with first-order algorithms.

\section{Conclusions}
\label{sec:conclusion}

We considered the convergence analysis of Newton-MR \cite{roosta2018newton} under inexact Hessian information in the form of additive noise perturbations. It is known that the pseudo-inverse of the Hessian is a discontinuous function of such perturbations. As a result, the pseudo-inverse of the perturbed Hessian can grow unboundedly with diminishing noise. However, our results indicate that it can indeed remain bounded along certain directions and under favorable conditions. We showed that the concept of inherently stable perturbations encapsulates situations under which Newton-MR with noisy Hessian remains stably convergent. Under such conditions, we established global and local convergence results for \cref{alg:Newton_invex_sub} using both exact and inexact updates. We argued that such stability analysis allows for the design of efficient variants of Newton-MR in which Hessian is approximated to reduce the computational costs in large-scale problems. We then numerically demonstrated the validity of our theoretical result and evaluated the performance of several such variants of Newton-MR as compared with various first and second-order methods.

\subsection*{Acknowledgment}
All authors are grateful for the support by the Australian Centre of Excellence for Mathematical and Statistical Frontiers (ACEMS). Fred Roosta was partially supported by DARPA D3M as well as the Australian Research Council through a Discovery Early Career Researcher Award (DE180100923). 
\bibliographystyle{plain}
\bibliography{biblio}
\end{document}